\documentclass[bj,numbers]{imsart}
\usepackage[OT1]{fontenc}

\pdfoutput=1

\RequirePackage{amsthm,amsmath,amsfonts,amssymb}
\RequirePackage[numbers]{natbib}

\usepackage[citecolor=blue,urlcolor=blue]{hyperref} 

\usepackage{graphicx,calc}
\usepackage[latin1]{inputenc}
\usepackage{verbatim}
\usepackage{enumitem}
\usepackage{mathtools}
\usepackage{bbm}
\usepackage{nicefrac}
\usepackage{tikz}
\usepackage{anyfontsize}

\graphicspath{{.}{figs/}}

\startlocaldefs

\theoremstyle{plain}
\newtheorem{Proposition}{Proposition}
\newtheorem{Lemma}{Lemma}
\newtheorem{Theorem}{Theorem}
\newtheorem{Corollary}{Corollary}

\theoremstyle{remark}             
\newtheorem{Remark}{Remark}
\newtheorem{Assumption}{Assumption}
\newtheorem{Example}{Example}
\newtheorem{Definition}{Definition}
\newtheorem{Conjecture}{Conjecture}

\newcommand{\proper}{\mathsf}

\newcommand{\pE}{\proper{E}}

\newcommand{\mv}[1]{{\boldsymbol{\mathrm{#1}}}}

\newcommand{\<}{\langle}
\renewcommand{\>}{\rangle}
\renewcommand{\H}{a}
\newcommand{\sol}{G}
\DeclareMathOperator{\Cov}{Cov}
\DeclareMathOperator{\Var}{Var}

\DeclarePairedDelimiter\ceil{\lceil}{\rceil}
\DeclarePairedDelimiter\floor{\lfloor}{\rfloor}

\renewcommand{\theenumi}{\normalfont{(\roman{enumi})}}

\endlocaldefs

\allowdisplaybreaks

\begin{document}
\begin{frontmatter}
\title{Markov properties of Gaussian random fields on compact metric graphs}

\begin{aug}
\author[A]{\fnms{David} \snm{Bolin}\ead[label=e1,mark]{david.bolin@kaust.edu.sa}} 
\author[A]{\fnms{Alexandre B.} \snm{Simas}\ead[label=e2]{alexandre.simas@kaust.edu.sa}} 
\author[B]{\fnms{Jonas} \snm{Wallin}\ead[label=e3]{jonas.wallin@stat.lu.se}}
\runauthor{David Bolin, Alexandre Simas and Jonas Wallin }
	\address[A]{Statistics Program, Computer, Electrical and Mathematical Sciences and Engineering Division, King Abdullah 
	University of Science and Technology (KAUST), 
	\printead{e1}, \printead{e2}}
\address[B]{Department of Statistics, Lund University, \printead{e3}} 
\end{aug}

\begin{abstract}
	There has recently been much interest in Gaussian fields on linear networks and, more generally, on compact metric graphs. One proposed strategy for defining such fields on a metric graph $\Gamma$ is through a covariance function that is isotropic in a metric on the graph. Another is through a fractional-order differential equation $L^{\nicefrac{\alpha}{2}} (\tau u) = \mathcal{W}$ on $\Gamma$, where $L = \kappa^2 - \nabla(\H\nabla)$ for (sufficiently nice) functions $\kappa, \H$, and $\mathcal{W}$ is Gaussian white noise. We study Markov properties of these two types of fields. First, we show that no Gaussian random fields exist on general metric graphs that are both isotropic and Markov. Then, we show that the second type of fields, the generalized Whittle--Mat\'ern fields, are Markov if $\alpha\in\mathbb{N}$, and conversely, if $\H$ and $\kappa$ are constant and $u$ is Markov, then $\alpha\in\mathbb{N}$. Further, if $\alpha\in\mathbb{N}$, a generalized Whittle--Mat\'ern field $u$ is Markov of order $\alpha$, which means that the field $u$ in one region $S\subset\Gamma$ is conditionally independent of $u$ in $\Gamma\setminus S$ given the values of $u$ and its $\alpha-1$ derivatives on $\partial S$. Finally, we provide two results as consequences of the theory developed: first we prove that the Markov property implies an explicit characterization of $u$ on a fixed edge $e$, revealing that the conditional distribution of $u$ on $e$ given the values at the two vertices connected to $e$ is independent of the geometry of $\Gamma$; second, we show that the solution 
	to $L^{\nicefrac{1}{2}}(\tau u) = \mathcal{W}$ on $\Gamma$
	can obtained by conditioning independent generalized Whittle--Mat\'ern processes on the edges, with $\alpha=1$ and Neumann boundary conditions, on being continuous at the vertices.
\end{abstract}



\end{frontmatter}

\section{Introduction}


Markov processes and random fields are essential models in many areas of probability and statistics. 
For example, in spatial statistics, Markov properties are regularly used to reduce computational costs for inference \cite{rue2005gaussian, lindgren2022spde}, and Markov properties play key roles in the theory of Gaussian free fields \cite{wendellin, char_GFF} and Schramm--Loewner evolutions \cite{schramm, lawler}.  
The standard definition of a Markov process can be given by saying that the `past' and `future' of the process are conditionally independent given the `present'. The extension of this definition to random fields dates to McKean \cite{McKean1963}, who proposed conditioning on an arbitrarily small neighborhood of the `present' instead of simply conditioning on the present, which allowed McKean to prove a conjecture by L\'evy \cite{Levy_Conjecture}, demonstrating that the so-called L\'evy Brownian motion is a Markov random field (MRF) if the dimension is odd. Following McKean's work, there was considerable research on Markov properties of random fields in the 1970s and 1980s. In particular, Loren Pitt extended the work by McKean in a seminal paper \cite{Pitt1971}, and Kunsch \cite{kunsch} and Rozanov \citep{Rozanov1977, Rozanov1982Markov} further extended Pitt's work. See also \cite{Mandrekar1976} and \cite{Pitt_Robeva} for a more recent extension.

In this work, we are interested in characterizing Markov properties of Gaussian random fields (GRFs) on compact metric graphs. 
A compact metric graph $\Gamma$ consists of a set of finitely many vertices ${\mathcal{V}=\{v_i\}}$ and a 
set of finitely many edges $\mathcal{E}=\{e_j\}$ connecting the vertices. Each edge $e$ is a curve with finite length $l_e<\infty$ parameterized by arc length, which connects a pair of vertices. A location $s\in \Gamma$ is a position on an edge and can be represented as a pair $(t,e)$, where $t\in[0,l_e]$. 
The graph is equipped with the geodesic metric, $d(\cdot,\cdot)$, which for any two points in $\Gamma$ is defined as the length of the shortest path in $\Gamma$ connecting the two. We assume that $\Gamma$ is connected, so that $d(\cdot,\cdot)$ is well defined.

One important special case of compact metric graphs is linear networks, obtained if each edge is a straight line. 
A second special case are metric trees, obtained if the graph has no loops. A final commonly considered special case is graphs with Euclidean edges \cite{anderes2020isotropic}, for which there is at most one edge between each pair of vertices, there are no edges which connect a vertex with itself, and the distance between two vertices connected by an edge is equal to the length of the edge connecting them.

Because of several statistical applications to data on traffic and river networks  \cite{CressieRiver, Hoef2006, isaak2017scalable, mcgill2021spatiotemporal, borovitskiy2021matern}, much research has recently been devoted to defining random fields on compact metric graphs and their various special cases \cite{anderes2020isotropic, BSW2022, moller2022lgcp, bolinetal_fem_graph}. 
The Gaussian free field has also been introduced on metric graphs (also known as cable systems) in \cite{lupu} and has been recently studied in \cite{ding_gff,drewitz2022cluster,cai2024one}.
Despite this interest, no work has been devoted to Markov properties of GRFs on general compact metric graphs.

Before considering Markov properties, let us review two central approaches for defining GRFs on metric graphs which will be considered. The first, proposed by \cite{anderes2020isotropic} is to define isotropic Gaussian process on graphs through isotropic covariance functions. A Gaussian process $u$ on $\Gamma$ is said to be isotropic if its covariance function is $\Cov(u(s_1), u(s_2)) = \rho(s_1,s_2)= r\circ d(s_1,s_2) = r(d(s_1,s_2))$, where $r : \mathbb{R} \rightarrow \mathbb{R}$ is a univariate function and $d(\cdot, \cdot)$ is a metric on the graph. The difficulty with this definition is to find a functions $r$ so that $r\circ d$ is positive definite on $\Gamma$. \cite{anderes2020isotropic} showed that it is difficult to find valid functions when using $d(\cdot,\cdot)$ as the geodesic metric, whereas there are more valid choices if $d(\cdot,\cdot)$ is chosen as the resistance metric (see Appendix~\ref{app:isotropic_apdx}) and $\Gamma$ has Euclidean edges.
 
An alternative construction by \cite{BSW2022} introduced Gaussian Whittle--Mat\'ern fields on compact metric graphs as solutions to the fractional-order equation
\begin{equation}\label{eq:Matern_spde}
	(\kappa^2 - \Delta_\Gamma)^{\alpha/2} (\tau u) = \mathcal{W}, \qquad \text{on $\Gamma$},
\end{equation}
where  $\mathcal{W}$ is Gaussian white noise, $\Delta_\Gamma$ is the so-called Kirchhoff Laplacian, $\kappa>0$ determines the practical correlation range (i.e., the smallest distance $r$ such that the correlation between $u(s_1)$ and $u(s_2)$ is $0.13$ if $d(s_1,s_2) = r$) and $\tau>0$ marginal variances of $u$, and $\alpha>\nicefrac 12$ determines the $L_2(\Omega)$ regularity of $u$. The precise definition, including the definitions of $\mathcal{W}$ and $\Delta_\Gamma$, is given in Section~\ref{sec:GWM}.
These fields are important since they provide a natural extension of GRFs with a Mat\'ern covariance \cite{matern60} to the metric graph setting, because Gaussian Mat\'ern fields on $\mathbb{R}^d$ can be obtained as a solution to a stochastic partial differential equation $(\kappa^2 - \Delta)^{\alpha/2} u = \mathcal{W}$ on $\mathbb{R}^d$, where $\Delta$ is the Laplacian and $\mathcal{W}$ is Gaussian white noise. The Whittle--Mat\'ern fields are, in fact, as far as we know the only known class of GRFs that are valid on general metric graphs and which can be made differentiable depending on the choice of the parameter $\alpha$ \cite{BSW2022} and they importantly do not require the metric graph to have Euclidean edges to be well-defined.
These fields can be generalized by considering a more general differential operator in \eqref{eq:Matern_spde}. Specifically, 
\cite{bolinetal_fem_graph} introduced the generalized Whittle--Mat\'ern fields by the fractional-order equation
\begin{equation}\label{eq:Matern_spde_general}
	(\kappa^2 - \nabla(\H\nabla))^{\alpha/2} (\tau u) = \mathcal{W}, \qquad \text{on $\Gamma$},
\end{equation}
where $\nabla$ denotes the differential operator that acts as the first order derivative operator on each edge  
from the graph (and will be formally introduced later), and $\kappa$ and $\H$ are (sufficiently nice) functions controlling the local practical correlation ranges and marginal variances of $u$, respectively. The fact that these models have precise control over the local properties is useful for applications on metric graphs for the same reasons as for their counterparts on manifolds and subsets of $\mathbb{R}^d$.
Specifically, Generalized Whittle--Mat\'ern fields have previously been considered on both manifolds \cite{lindgren11} and compact subsets of $\mathbb{R}^d$ \cite{fuglstad2015exploring,lindgren11,Hildeman2020}, obtained as solutions to \eqref{eq:Matern_spde_general} posed on those domains instead of $\Gamma$. These models have proven to be important non-stationary extensions of the Whittle--Mat\'ern fields, where the functions $\kappa$ and $\H$ can be chosen to mimic non-stationary behavior in data for statistical applications \cite{fuglstad2015does,fuglstad2015exploring,lindgren11,Hildeman2020}, which also should be a useful feature on metric graphs.

The challenge with studying Markov properties of GRFs on metric graphs is that the complicated geometry of the space in combination with differentiable processes such as those in \cite{BSW2022, bolinetal_fem_graph} translates to an irregular behavior of the $\sigma-$algebras that define the Markov property. To explain this further, let us first introduce Markov properties in more detail.
For  $S\subset \Gamma$, $\partial S$ denotes the topological boundary, and $\overline{S}$ the closure, of $S$. For $\varepsilon>0$, we define  $S_\varepsilon = \{s\in \Gamma: \exists z\in S \text{ with }d(s, z)<\varepsilon\}$ as the $\varepsilon$-neighborhood of $S$. We consider a complete probability space $(\Omega, \mathcal{F},\mathbb{P})$ and let $\mathbb{E}(Z) = \int_\Omega Z(\omega) d\mathbb{P}(\omega)$ denote the expectation of a real-valued random variable $Z$. 
For sub-$\sigma$-algebras $\mathcal{A},\mathcal{B}$ and $\mathcal{C}$ of $\mathcal{F}$, we say that $\mathcal{A}$ and $\mathcal{B}$ are conditionally independent 
given $\mathcal{C}$ if 
$\mathbb{P}(A\cap B | \mathcal{C}) = \mathbb{P}(A|\mathcal{C}) \mathbb{P}(B|\mathcal{C})$
for every $A\in\mathcal{A}$ and $B\in\mathcal{B}$,
where $\mathbb{P}(F | \mathcal{C}) = \mathbb{E}(1_F| \mathcal{C})$ is the conditional
probability of a set $F\in\mathcal{F}$ given $\mathcal{C}$.
In this case, we say that $\mathcal{C}$ splits $\mathcal{A}$ and $\mathcal{B}$.
Now, for a stochastic process $u$ on $\Gamma$ and, for each Borel measurable set $S\subset\Gamma$, we define the $\sigma$-algebras $\mathcal{F}^u(S) = \sigma(u(s): s\in S)$ and $\mathcal{F}^u_+(S) = \bigcap_{\varepsilon > 0}\mathcal{F}^u(S_\varepsilon)$. We now present the definition of MRFs on general compact metric graphs. The definition can be seen as a version for metric graphs of the definition of MRFs on subsets of $\mathbb{R}^d$ given in \cite{kunsch}.

\begin{Definition}\label{def:MarkovPropertyField}
	A random field $u$ on a compact metric graph $\Gamma$ is an MRF (of any possible order) if for every open set $S$, $\mathcal{F}^u_+(\partial S)$ splits $\mathcal{F}^u_+(\overline{S})$ and $\mathcal{F}^u_+(\Gamma \setminus S)$.
\end{Definition} 

Typically, we want to characterize the splitting $\sigma$-algebra $\mathcal{F}^u_+(\partial S)$ to use this definition, because this contains the information that needs to be conditioned on to make the field on $\bar{S}$ and $\Gamma\setminus S$ independent. 
 It is in particular important to characterize it for differentiable random fields on $\Gamma$, where higher-order Markov properties are defined as follows. 

\begin{Definition}\label{def:MarkovPropertyFieldOrderP}
	A random field $u$ on a compact metric graph $\Gamma$ is Markov of order $p$ if it has $p-1$ weak derivatives in the mean-squared sense (see Definition~\ref{def:RandomFieldOrderP}) and
	$\sigma(u(s), u'(s),\ldots, u^{(p-1)}(s): s\in\partial S)$ splits $\mathcal{F}_+^u(\overline{S})$ and $\mathcal{F}_+^u(\Gamma \setminus S)$ for every open set $S$.
\end{Definition}

Observe that every MRF of order $p$ is, in particular, an MRF in the sense of Definition~\ref{def:MarkovPropertyField}.
To characterize $\mathcal{F}^u_+(\partial S)$, it is useful to consider the $\sigma$-algebras generated by $u$ on the `inner' and `outer' boundaries of $S$: 
$
\mathcal{F}^u_{+,I}(\partial S) = \bigcap_{\varepsilon>0} \mathcal{F}^u(S\cap (\partial S)_\varepsilon)$ and
 ${\mathcal{F}^u_{+,O}(\partial S) = \bigcap_{\varepsilon>0} \mathcal{F}^u((\Gamma\setminus S)\cap (\partial S)_\varepsilon)}$, respectively. 

\begin{Remark}\label{rem:inner_outer_example}
	To make the distinction between the inner boundary and the ``full'' boundary clear, consider the example of a metric graph $\Gamma$ which is the union of the edges $e_1 = [0,1]$ and $e_2 = [0,2]$, with vertices $v_1 = (0,e_1)$, $v_2 = (1,e_1) = (0,e_2)$ and $v_3 = (2,e_2)$, so that $v_2$ is a common vertex for both $e_1$ and $e_2$. Then, the ``full'' boundary of $e_1$ consists of $\{(0,e_1), (1,e_1), (0,e_2)\}$, whereas the inner boundary of $e_1$ consists of $\{(0,e_1),(1,e_1)\}$ and the outer boundary consists of $\{(0,e_2)\}$.
\end{Remark}

If the inner and outer boundaries coincide for any set $S$, the boundary $\partial S$ is said to be regular, and if it holds for all sets $S\subset\Gamma$, the field is regular \cite{Rozanov1982Markov,Rozanov1977}. Regularity simplifies the analysis, and this was assumed by Pitt throughout \cite{Pitt1971}. 
The issue that we alluded to above is that differentiable fields on metric graphs often are nonregular. Note that the definition of MRF given in \cite{Pitt1971}, albeit similar to Definition \ref{def:MarkovPropertyField}, requires the field to be regular. For this reason we chose to follow \cite{kunsch} in Definition \ref{def:MarkovPropertyField} rather than \cite{Pitt1971}. One should also note that Definition \ref{def:MarkovPropertyField} is also similar to the ones introduced in \cite{Mandrekar1976,Rozanov1982Markov}.

In the first part of this work, Section~\ref{sec:isotropic}, we focus on isotropic GRFs on compact metric graphs, as introduced by \cite{anderes2020isotropic}. We show that there are no GRFs on general metric graphs that are both isotropic (in either the geodesic or the resistance metric) and Markov. Thus, one must choose between the two properties when determining the model for a particular application. 

Because there are no isotropic GRFs that are Markov on general metric graphs, we consider in the second part of the work the generalized Whittle--Mat\'ern fields of \citep{bolinetal_fem_graph}.
In Sections \ref{sec:GWM} and \ref{sec:proofsGWM}, we show that, similarly to the Whittle--Mat\'ern fields, the parameter ${\alpha>\nicefrac12}$ determines the $L_2(\Omega)$ regularity of the generalized Whittle--Mat\'ern fields, and we prove that under suitable assumptions on $\kappa$ and $\H$, these fields are $k$ times weakly differentiable in the $L_2(\Omega)$ sense for $k\leq \floor{\alpha}-1$. 
Our main result in this part of the work is that the generalized Whittle--Mat\'ern fields with suitable assumptions on $\kappa$ and $\H$ are Markov if and only if $\alpha \in\mathbb{N}$. Further, if $\alpha\in\mathbb{N}$, the field is Markov of order $\alpha$.
As particular cases, we obtain that the Whittle--Mat\'ern fields are Markov if and only if $\alpha\in\mathbb{N}$.

The final part of the manuscript provides two important consequences of the Markov property. 
Section~\ref{sec:edge} demonstrates that it implies an explicit characterization of the generalized Whittle--Mat\'ern fields on a fixed edge, as a sum of a GRF that is zero at the end points of the edge and a low-rank process that only depends on the process and its derivatives at the end points of the edge. This finding reveals that the conditional distribution of the process on an edge given the vertex values is independent of the graph geometry. In the same direction, in Section~\ref{sec:condrepr}, we obtain a conditional representation for the solution to \eqref{eq:Matern_spde_general} with $\alpha=1$. More precisely, we start with a family of independent generalized Whittle--Mat\'ern processes on the edges, where for each edge $e$ the corresponding process is a solution to $(\kappa_e^2 - \nicefrac{d}{dt}(a_e\nicefrac{d}{dt}))^{1/2} (\tau u)= \mathcal{W}$ on $e$, where $\kappa_e^2$ and $a_e$ are the restrictions of $\kappa^2$ and $a$ to $e$, respectively. Then, we show that if we condition these processes on being continuous at the vertices of $\Gamma$, the resulting random field is a solution \eqref{eq:Matern_spde_general}. The two representations in Section~\ref{sec:edge} and Section~\ref{sec:condrepr}
serve as critical first steps towards deriving efficient inference methods for these processes, which will be investigated further in future work. 
Finally, extensions and future work are discussed in Section~\ref{sec:discussion}, and auxiliary definitions and results are presented in four appendices.

\section{Notation and preliminaries}\label{sec:Markov}

If $u$ is an MRF, it is immediate that the $\sigma$-algebra $\mathcal{F}^u_+(\partial S)$ splits $\mathcal{F}^u(S)$ and $\mathcal{F}^u(\Gamma\setminus\overline{S})$ for every open set $S\subset\Gamma$. 
Moreover, the following fundamental characterization holds in terms of the $\sigma$-algebras generated by $u$ on the inner boundary of the sets. 

\begin{Proposition}\label{prp:markovinnerboundary}
	Let $\Gamma$ be a compact metric graph and let $u$ be a random field on $\Gamma$. Then, the statement that $u$ is an MRF is equivalent to any of the following statements:
	\begin{enumerate}[label= \roman{enumi}.]
		\item $\mathcal{F}^u_{+,I}(\partial S)$ splits $\mathcal{F}^u(S)$ and $\mathcal{F}^u_+(\Gamma\setminus S)$ for any open set $S$ whose boundary consists of finitely many points;\label{prp:markovinnerboundary1}
		\item $\mathbb{E}(f|\mathcal{F}^u(S)) = \mathbb{E}(f| \mathcal{F}^u_{+,I}(\partial S))$ for any open set $S$ whose boundary consists of finitely many points and for any bounded $\mathcal{F}^u_+(\Gamma\setminus S)$-measurable function $f$.\label{prp:markovinnerboundary2}
	\end{enumerate}        
\end{Proposition}

The proof uses 
characterizations and properties of conditional independence provided in Appendix~\ref{app:aux_secMarkov}, and is therefore provided there. 
The result shows that, for an MRF $u$, the inner boundary splits an open set $S$ and its (closed) complementary set $\Gamma\setminus S$. This has the remarkable consequence that the restriction of an MRF on a metric graph to an edge is conditionally independent of all other edges, given the information on the inner boundary (i.e., the two endpoints) of the edge. 

\begin{Remark}
	 Consider $\Gamma$ given in Remark \ref{rem:inner_outer_example} and that $u$ is a random field on $\Gamma$ such that $u|e_1$ and $u|e_2$ are continuous. It is clear that if $u$ is continuous on $\Gamma$, we have that ${u(1,e_1) = u(0,e_2) = u(v_2)}$. However if $u$ is not continuous, $u(1,e_1)$ and $u(0,e_2)$ may have different values and, thus, to condition on the inner boundary is different than to condition on the ``full'' boundary. Therefore, it is important to address the differences as in some cases, such as in the edge representation that will be introduced in Section~\ref{sec:edge}, it is natural to consider conditionals with respect to the inner boundary.
\end{Remark}

We let $L_2(\Omega)$ denote the Hilbert space of all (equivalence classes of) square integrable real-valued random variables. For a GRF $u$ on $\Gamma$ and a Borel set $S\subset \Gamma$, we define the Gaussian space ${H(S) = \overline{\textrm{span}(u(s): s\in S)}}$,
where  the closure is taken with respect to the $L_2(\Omega)$-norm. The elements of this space are Gaussian random variables because  each element in $\textrm{span}(u(w): w\in S)$ is Gaussian. Thus, the elements of $H(S)$ are $L_2(\Omega)$-limits of Gaussian variables and, hence, Gaussian \cite[e.g. see][Theorem 1.4.2, p.39]{ashtopics}. 
By defining $H_+(S) = \bigcap_{\varepsilon>0} H(S_\varepsilon)$, we have the following alternative representation of the Markov property in terms of Gaussian spaces (see also \citep[Lemma 1.3]{kunsch}).
In the representation, we use the notation $H_1\oplus H_2 =  \{h_1 + h_2: h_1\in H_1, h_2\in H_2\}$ for the direct sum of two Hilbert spaces $H_1,H_2 \subset H$ with $H_1\cap H_2 = \{0\}$. Similarly, if $H_1\subset H_2$, we write $H_2\ominus H_1$ to denote the orthogonal complement of $H_1$ in $H_2$. 
\begin{Proposition}\label{prp:charmarkovspaces}
	Let $\Gamma$ be a metric graph and let $u$ be a GRF on $\Gamma$. Then, the fact that $u$ is an MRF is equivalent to any of the following statements: For any open set $S$, 
	\begin{enumerate}[label= \roman{enumi}.]
		\item the orthogonal projection of $H_+(\overline{S})$ on $H_+(\Gamma\setminus S)$ is $H_+(\partial S)$;\label{prp:charmarkovspaces1}
		\item $H(\Gamma) = H_+(\overline{S})\oplus(H_+(\Gamma\setminus S)\ominus H_+(\partial S)) = H_+(\Gamma\setminus S)\oplus (H_+(\overline{S})\ominus H_+(\partial S))$. \label{prp:charmarkovspaces2}
	\end{enumerate}
\end{Proposition}
\begin{proof}
The equivalence between i and ii is immediate. We now prove the equivalence between the Markov property and i.
Given a GRF $u$ on $\Gamma$, 
the minimal splitting $\sigma$-algebra of $\mathcal{F}^u_+(\overline{S})$ and $\mathcal{F}^u_+(\Gamma\setminus S)$, where $S$ is an open set, is 
$
\sigma(\mathbb{E}(u(s)|\mathcal{F}^u_+(\Gamma\setminus S)): s\in \overline{S})
$ by \cite[Lemma 2.3]{Pitt1971}. 
Linear combinations of $\mathcal{F}^u(S)$-measurable functions are $\mathcal{F}^u(S)$-measurable, and if $v_n$ is a sequence of $\mathcal{F}^u(S)$-measurable functions that converges to some function $v$ in $L_2(\Omega)$, then there exists a subsequence $(v_{n_j})$ such that $(v_{n_j})$ converges almost surely to $v$, showing that  $v$ is $\mathcal{F}^u(S)$-measurable. 
Therefore, every element of $H(S)$ is $\mathcal{F}^u(S)$-measurable, and it follows that, for every $S\subset \Gamma$, ${\mathcal{F}^u(S) = \sigma(H(S))}$.
We can now apply \cite[Lemma 3.3]{Mandrekar1976} to obtain that $\mathcal{F}_+^u(\overline{S}) = \sigma(H_+(\overline{S}))$ for every $S\subset\Gamma$. Now, observe that, by combining the previous identity with Gaussianity we have that conditional expectation with respect to $\mathcal{F}_+^u(\overline{S})$ is simply the orthogonal projection onto the space $H_+(\overline{S})$, see, e.g., \cite[Theorem 9.1]{janson_gaussian}. Therefore, since the minimal splitting $\sigma$-algebra is $\sigma(\mathbb{E}(u(s)|\mathcal{F}^u_+(\Gamma\setminus S)): s\in \overline{S})$,  the minimal splitting $\sigma$-algebra is generated by the projections of $H_+(\overline{S})$ onto $H_+(\Gamma\setminus S)$. We are now in a position to conclude the proof. Let $u$ satisfy the Markov property, then the minimal splitting $\sigma$-algebra is contained in $\mathcal{F}_+^u(\partial S)$, so that the projection of $H_+(\overline{S})$ onto $H_+(\Gamma\setminus S)$ is contained in $H_+(\partial S)$, however, we also have that $H_+(\partial S) \subset H_+(\overline{S})\cap H_+(\Gamma\setminus S)$. Therefore, the projection of $H_+(\Gamma\setminus S)$ onto $H_+(\Gamma\setminus S)$ is $H_+(\partial S)$. Conversely, if the projection of $H_+(\Gamma\setminus S)$ onto $H_+(\Gamma\setminus S)$ is $H_+(\partial S)$, we have that the minimal splitting $\sigma$-algebra is $\sigma(H_+(\partial S)) = \mathcal{F}_+(\partial S)$, which implies the Markov property.
\end{proof}
Finally, one can also characterize the Markov property through the Cameron--Martin space associated with the random field $u$. If $u$ has a continuous covariance function ${\rho:\Gamma\times\Gamma\to\mathbb{R}}$, we define the Cameron--Martin space $\mathcal{H}(S)$, for $S\subset\Gamma$, as
$
\mathcal{H}(S) = \{h(s) = \mathbb{E}(u(s)v): s\in \Gamma\hbox{ and } v\in H(S)\},
$
with inner product $\langle h_1, h_2\rangle_{\mathcal{H}} = \mathbb{E}(v_1 v_2),$
where $h_j(s) = \mathbb{E}(u(s)v_j)$, $s\in\Gamma$, and $v_j \in H(S)$, where $j=1,2$.  
Then $\mathcal{H}(S)$ is isometrically isomorphic to $H(S)$.  For the following definition, recall that the support of a function $h:\Gamma\to\mathbb{R}$ is defined as $\textrm{supp}\,h = \overline{\{s\in\Gamma: h(s)\neq 0\}}$.
\begin{Definition}\label{def:localCMspaces}
	Let $\Gamma$ be a compact metric graph and let $H$ be an inner product space of functions defined on $\Gamma$, endowed with the inner product
	$\langle \cdot, \cdot\rangle_{H}$. We say that $H$ is local if:
	\begin{enumerate}[label= \roman{enumi}.]
		\item If $h_1,h_2\in H$ and $\textrm{supp}\, h_1 \cap \textrm{supp}\, h_2 = \emptyset$, then $\langle h_1,h_2\rangle_{H} = 0;$\label{def:localCMspaces1}
		\item If $h \in H$ is such that $h = h_1+h_2$, where $\textrm{supp}\, h_1\cap \textrm{supp}\,h_2 = \emptyset$, then $h_1, h_2 \in H$.\label{def:localCMspaces2}
	\end{enumerate}
\end{Definition}

Next, for a Borel set $S\subset\Gamma$, we define 
$\mathcal{H}_+(S) = \bigcap_{\varepsilon > 0} \mathcal{H}(S_\varepsilon)$,
and let $\Phi:H(\Gamma)\to \mathcal{H}(\Gamma)$ be the isometrically isomorphic map $\Phi(v)(\cdot) = \mathbb{E}(u(\cdot)v)$. Then, 
$$
H_+(S) = \bigcap_{\varepsilon>0} H(S_\varepsilon) = \bigcap_{\varepsilon>0}\Phi^{-1}(\mathcal{H}(S_\varepsilon)) = \Phi^{-1}\left(\mathcal{H}_+(S)\right),
$$
so that $\mathcal{H}_+(S) = \{h(s) = \mathbb{E}(u(s)u): s\in \Gamma\hbox{ and } u\in H_+(S)\}.$
Thus, we can rewrite Proposition \ref{prp:charmarkovspaces}.\ref{prp:charmarkovspaces1} as that  
the orthogonal projection of $\mathcal{H}_+(\overline{S})$ on $\mathcal{H}_+(\Gamma\setminus S)$ is $\mathcal{H}_+(\partial S)$
for any open set $S$. Similarly, 
Proposition~\ref{prp:charmarkovspaces}.\ref{prp:charmarkovspaces2} translates to that 
$$
\mathcal{H}(\Gamma) = \mathcal{H}_+(\overline{S})\oplus(\mathcal{H}_+(S^c)\ominus \mathcal{H}_+(\partial S)) = \mathcal{H}_+(S^c)\oplus (\mathcal{H}_+(\overline{S})\ominus \mathcal{H}_+(\partial S)),
$$
for any open set $S$.
Finally, we have the following crucial result, which follows from  \cite[Theorem~5.1]{kunsch}. 

\begin{Theorem}\label{thm:MarkovLocalCM}
	Let $u$ be a GRF defined on a compact metric graph $\Gamma$, with Cameron--Martin space $\mathcal{H}(\Gamma)$. Then, $u$ is a GMRF if, and only if, $\mathcal{H}(\Gamma)$ is local.
\end{Theorem}

\begin{Remark}
	We have introduce three different characterizations of Markov properties of GRFs. The first  involves $\sigma$-algebras and is useful for intuition as it relates to conditional independence and creates a link to the standard Markov property in stochastic processes with a one-dimensional parameter.
	The second characterization involves Gaussian spaces and, even though it is not as intuitive as the first, it is more useful for our proofs as it translates Markov properties into the language of functional analysis and linear algebra. 
	However, the elements in these spaces are still random variables and in this sense we have the third characterization that involves Cameron-Martin spaces 
	that consists of deterministic variables and for which case the results can be interpreted in the language of partial differential equations and functional analysis. Therefore, depending on the problem we will be switching to the most convenient characterization with respect to the result we want to prove.
\end{Remark}

\section{Isotropic random fields}\label{sec:isotropic}

This section explores the Markov properties of isotropic GRFs on compact metric graphs. We focus on processes that exhibit isotropy with respect to either the geodesic metric $d(\cdot,\cdot)$ or the resistance metric $d_R(\cdot,\cdot)$. The resistance metric extends the classical concept of resistance distance from combinatorial graphs to metric graphs, as introduced by \cite{anderes2020isotropic}. A formal definition of the resistance metric, along with key relationships to the geodesic metric, is provided in Appendix~\ref{app:isotropic_apdx}, together with the definition of metric graphs with Euclidean edges, as well as Euclidean cycles.

\begin{Theorem}\label{thm:markovCircle}
	Let $\Gamma$ be a circle or a Euclidean cycle with length $\ell$, and suppose that $X(\cdot)$ is a GRF on $\Gamma$ that is isotropic in the geodesic metric. Then $X(\cdot)$ is Markov of order $1$ if and only if its covariance function ${\rho(s,s') = r_1(d(s,s'))}$ is the Green's function of the operator $Q = \tau^2(\kappa^2 - \Delta_\Gamma)$ for some  $\tau,\kappa>0$ and some function $r_{1,\ell}(\cdot)$.
\end{Theorem}

\begin{proof}
	The result for the case when $\Gamma$ is a circle is proved in  \cite{Pitt1971}. Clearly, one can add vertices of degree 2 to the circle without changing the geodesic distance. Further, any Euclidean cycle with $n$ vertices and total edge length $\ell$ is isomorphic to a circle of perimeter $\ell$ with $n$ vertices. Therefore,  the case when $\Gamma$ is a Euclidean cycle directly follows from the circle case.  
\end{proof}
The precision operator of a Whittle--Mat\'ern field defined  in \eqref{eq:Matern_spde} with $\alpha=1$ is $Q = \tau^2(\kappa^2 - \Delta_\Gamma)$. Thus, it is the only GRF on the circle that is isotropic in the geodesic metric and Markov of order 1. It is easy to calculate that if the circle has perimeter $\ell$, then  the covariance function of this Whittle--Mat\'ern field is given by $r(s,s') = r_{1,\ell}(d(s,s'))$, where
\begin{equation}\label{eq:formula_pitt}
	r_{1,\ell}(h) = \frac{\cosh(\kappa(h-\ell/2))}{2\kappa\tau^2\sinh(\kappa \ell/2)}.
\end{equation}

The same result holds for the resistance metric:

\begin{Corollary}\label{cor:circle}
	Let $X(t)$ be a GRF with covariance function $r(d_R(s,s'))$ on a Euclidean cycle $\Gamma$, that is isotropic in the resistance metric. Then, if $X(\cdot)$ is Markov of order 1, $X(\cdot)$ is a Whittle--Mat\'ern field with $\alpha=1$. That is, $r(d_R(s,s')) = r_{1,\ell}(d(s,s'))$.
\end{Corollary}

\begin{proof}
	Suppose that the Euclidean cycle has a total edge length $\ell>0$. 
	By Proposition \ref{prp:anderes_resistance_geodesic} in Appendix~\ref{app:isotropic_apdx}, the resistance metric on the Euclidean cycle satisfies $d_R(s,s') = d(s,s') - \ell^{-1}d(s,s')^2$. Thus, 
	\begin{equation*}
	r(d_R(s,s')) = r(d(s,s') - \ell^{-1}d(s,s')^2) := \hat{r}(d(s,s'))
	\end{equation*}
	is also isotropic in the geodesic metric. By Theorem~\ref{thm:markovCircle}, we have that $\hat{r}(d(s,s')) = r_{1,\ell}(d(s,s'))$ and the result follows.
\end{proof}

Corollary \ref{cor:circle} shows that, depending on the structure of the metric graph, GRFs can be both isotropic and Markov of order 1. A particular case where this occurs is when the metric graph is a tree, where the resistance metric coincides with the geodesic metric, see Proposition \ref{prp:anderes_resistance_geodesic} in Appendix~\ref{app:isotropic_apdx}. Additionally, \cite{moller2022lgcp} established that a GRF with an isotropic exponential covariance function is Markov on trees. However, this is not always the case. Specifically, we will prove that there exists a class of metric graphs with Euclidean edges for which no Gaussian Markov random field can be both isotropic and Markov, showing that, in general, compact metric graphs cannot exhibit both isotropy and Markovianity simultaneously. To formalize this, consider the following assumption. For the definitions of 1-sums and $k$-step 1-sum of metric graphs, see Definitions \ref{def:1sum_spaces} and \ref{def:kstep_1sum} in Appendix~\ref{app:isotropic_apdx}.

\begin{Assumption}\label{assump:iso_markov}
	Let $S$ be a Euclidean cycle, and let $k_1, k_2 \in \mathbb{N}$. Consider a collection of $k_1 + k_2 - 2$ metric graphs $\Gamma_1, \ldots, \Gamma_{k_1+k_2-2}$ with Euclidean edges, where $\Gamma_i \cap \Gamma_j = \emptyset$ for $i \neq j$. Finally, let $T$ be either a Euclidean cycle of length different from that of $S$, or an edge. Then, the metric graph $\Gamma$ is the $(k_1+k_2)$-step 1-sum of $S, \Gamma_1, \ldots, \Gamma_{k_1-1}, T, \Gamma_{k_1}, \ldots, \Gamma_{k_1+k_2-2}$.
\end{Assumption}

Examples of metric graphs satisfying Assumption \ref{assump:iso_markov} are 1-sums between a Euclidean cycle and a graph with at least one vertex of degree 1, as well as 1-sums between two Euclidean cycles of different lengths. Additional examples are shown in Figure~\ref{fig:graphs_nonEuclidean}. We are now ready to prove the main theorem of this section.

\begin{figure}
	\begin{center}
		\begin{tikzpicture}[scale=0.5] 
			\draw (0,pi/2) circle (pi/2);
			\draw (0,-pi/3) circle (pi/3);
			\draw (-2*pi/3,-pi/3) circle (pi/3);
			\filldraw [black] (0,0) circle (4pt);
			\filldraw [gray] (0,pi) circle (4pt);
			\filldraw [gray] (-pi,-pi/3) circle (4pt);
			\filldraw [gray] (-2*pi/3,0) circle (4pt);
			\filldraw [gray] (pi/2,pi/2) circle (4pt);
			\filldraw [gray] (-pi/2,pi/2) circle (4pt);
			\filldraw [gray] (0,-2*pi/3) circle (4pt);
			\filldraw [black] (-pi/3,-pi/3) circle (4pt);
			\filldraw [gray] (pi/3,-pi/3) circle (4pt);
	
		\draw (5.5,0) circle (pi);
	\filldraw [gray] (5.5-pi,0) circle (4pt);
	\filldraw [black] (5.5+pi,0) circle (4pt);
	\filldraw [gray] (5.5,pi) circle (4pt);
	\filldraw [gray] (5.5,-pi) circle (4pt);
	\filldraw [gray] (5.5,0) circle (4pt);
	\draw (5.5,0) -- (5.5+pi,0)  {};
	
	\draw (10,2.5) -- (15,2.5)  {};
	\draw (10,2.5) -- (11.25,1.25)  {};
	\draw (13.75,1.25) -- (15,2.5)  {};
	\draw (11.25,1.25) -- (13.75,1.25) {};
	\draw (11.25,1.25) -- (10,-2.5)  {};
	\draw (10, -2.5) -- (12.5,-2.5) {};
	\draw (12.5,-2.5) -- (12.5,-1.25) {};
	\draw (12.5,-1.25) -- (15,-2.5) {};
	\draw (13.75,1.25) -- (15,-1.25) {};
	\filldraw [gray] (10,2.5) circle (4pt);
	\filldraw [gray] (10,-2.5) circle (4pt);
	\filldraw [gray] (12.5,-1.25) circle (4pt);
	\filldraw [gray] (15,2.5) circle (4pt);
	\filldraw [gray] (15,-1.25) circle (4pt);
	\filldraw [gray] (15,-2.5) circle (4pt);
	\filldraw [gray] (12.5,-2.5) circle (4pt);
	\filldraw [black] (11.25,1.25) circle (4pt);
	\filldraw [black] (13.75,1.25) circle (4pt);
	
	\draw (18,pi/3) circle (pi/3);
	\draw (18,0) -- (18,-1.5)  {};
	\draw (18,0) -- (16.5,-1)  {};
	\draw (18,0) -- (19,-1)  {};
	\draw (18,-1.5) -- (17,-2.5)  {};
	\draw (18,-1.5) -- (18.5,-2.5)  {};
	\filldraw [black] (18,0) circle (4pt);
	\filldraw [gray] (18,-1.5) circle (4pt);
	\filldraw [gray] (16.5,-1) circle (4pt);
	\filldraw [gray] (19,-1) circle (4pt);
	\filldraw [gray] (17,-2.5) circle (4pt);
	\filldraw [gray] (18.5,-2.5) circle (4pt);
	\filldraw [gray] (19,-1) circle (4pt);
	\end{tikzpicture}
\end{center}
\caption{Examples of metric graphs satisfying Assumption \ref{assump:iso_markov}. The black vertices indicate intersection points of the 1-sums.}
\label{fig:graphs_nonEuclidean}
\end{figure}

\begin{Theorem}\label{thm:markov_resistance}
	Let $\Gamma$ be a metric graph satisfying Assumption \ref{assump:iso_markov}. Suppose $X(\cdot)$ is a GRF on $\Gamma$ with an isotropic covariance function $\rho(s,t) = r(\widetilde{d}(s,t))$, where $r(\cdot)$ is a continuous function and $\widetilde{d}(\cdot,\cdot)$ denotes either the resistance metric or the geodesic metric. Then $X$ is not Markov of order 1.
\end{Theorem}

\begin{proof}
	Let $S$ and $T$ be as in Assumption \ref{assump:iso_markov}, and let $X_S$ (resp. $X_T$) denote the random field with covariance $r(d(\cdot, \cdot))$ (or $r(d_R(\cdot, \cdot))$ for the resistance metric) on the metric graph $S$ (resp. $T$). By the $k$-step 1-sum assumption, the geodesic (resp. resistance) metric on $\Gamma$, when restricted to $S$ (resp. $T$), coincides with the geodesic (resp. resistance) metric on $S$ (resp. $T$). Thus, $X_S$ and $X|_S$ (resp. $X_T$ and $X|_T$) share the same finite-dimensional distributions. Consequently, it suffices to prove the result for $X_{\Gamma_0}$, where $\Gamma_0$ is the 1-sum of a Euclidean cycle $S$ and $T$, with $T$ being either a Euclidean cycle or an edge.

	Let us start with the case in which $\Gamma_0$ is the 1-sum of the Euclidean cycles $S_1$ and $S_2$, with $S_1$ and $S_2$ having lengths $\ell_1\neq \ell_2$, respectively. Let $v$ be be the intersecting vertex from $S_1$ and $S_2$, that is, $\{v\} = S_1\cap S_2$. To this end, let $X_{\Gamma_0}$ be isotropic with respect to the resistance distance and assume, by contradiction, that $X$ is also Markov of order 1. Let  $d_R(\cdot,\cdot)$ denote the resistance metric on $\Gamma_0$, and $d_{R,S_j}(\cdot,\cdot)$ denote the resistance metric in $S_j,j=1,2$. Observe that since $\Gamma_0$ is the 1-sum between $S_1$ and $S_2$, the restriction of $d_R(\cdot,\cdot)$ to $S_j\times S_j$ coincides with $d_{R,S_j}(\cdot,\cdot)$, for $j=1,2$. Now, by Corollary \ref{cor:circle}, we have that for every $s_1,s_1'\in S_1$ and $s_2,s_2'\in S_2$, we have $r(d_R(s_1,s_1')) = r_{1,\ell_1}(d(s_1,s_1'))$, and $r(d_R(s_2,s_2')) = r_{1,\ell_2}(d(s_2,s_2'))$. Now, we have that for any $s_1\in S_1$ and $s_2\in S_2$ such that $d_R(v,s_1) = d_R(v,s_2)$. Then, $r_{1,\ell_1}(d(s_1,v)) = r(d_R(s_1,v)) = r(d_R(s_2,v)) = r_{1,\ell_2}(d(s_2,v))$. However, by Proposition \ref{prp:aux_prop_cycle_res_geo} in Appendix~\ref{app:isotropic_apdx}, this is a contradiction, since it says we can find $s_1\in S_1$ and $s_2\in S_2$ such that $d_R(s_1,v)=d_R(s_2,v)$ but $r_{1,\ell_1}(d(s_1,v)) \neq r_{1,\ell_2}(d(s_2,v))$. 
	
	Now, let $X$ be isotropic with respect to the geodesic distance and Markov of order 1. For any $s_1\in S_1$ and any $s_2\in S_2$ such that $d(v,s_1)=d(v,s_2)$, we have that Corollary \ref{cor:circle} and isotropy imply $r_{1,\ell_1}(d(s_1,v)) = r(d(s_1,v)) = r(d(s_2,v)) = r_{1,\ell_2}(d(s_2,v))$. By \eqref{eq:formula_pitt} this is a contradiction, since $\ell_1\neq \ell_2$. 

	The result for the case when $\Gamma_0$ is a 1-sum of a Euclidean cycle and an edge follows from Proposition~\ref{prp:1sum_euclidean_edge_iso} in Appendix \ref{app:isotropic_apdx}.
\end{proof}

Thus, GRFs that are both Markov and isotropic exist on simple graphs such as trees and circles, but not on more general metric graphs. Another limitation, which we have found to have a significant impact when using real data, is that isotropic models require metric graphs with Euclidean edges. This motivates our examination of generalized Whittle--Mat\'ern fields in the next section. These fields are generally not isotropic with respect to either the resistance metric or the geodesic metric. However, as we will show, they satisfy Markov properties and are well-defined on general compact metric graphs.

\section{Generalized Whittle--Mat\'ern fields}\label{sec:GWM}

This section investigates Markov properties of the generalized Whittle--Mat\'ern fields defined by the stochastic differential equation \eqref{eq:Matern_spde_general}. Section~\ref{sec:GWMintro}  introduces the fields in more detail. 
Section~\ref{sec:GWMmarkov} and Section~\ref{sec:GWMmarkov_alpha} provide the main results regarding the Markov properties of the fields. 
Section~\ref{sec:GWMnonregular} and Section~\ref{sec:CM} provide results regarding the regularity and the Cameron--Martin spaces of the fields, which are important in Section~\ref{sec:edge}. To simplify the presentation, we defer the proofs of all main results in this section to Section~\ref{sec:proofsGWM}, while auxiliary results are provided in Appendix~\ref{app:aux}.

\subsection{Definitions and preliminaries}\label{sec:GWMintro}
First, we introduce some additional notation. For two Hilbert spaces $E$ and $F$, $E\subset F$, the continuous embedding $E\hookrightarrow F$ exists if the inclusion map from $E$ to $F$ is continuous (i.e. a $C>0$ exists such that, for every $f\in E$, ${\|f\|_F \leq C \|f\|_E}$). Further, we write $E\cong F$ if 
${E} \hookrightarrow {F} \hookrightarrow {E}$. 

From now on, we let $\Gamma$ denote an arbitrary compact metric graph. 
For a function $f$ on $\Gamma$, $f_e = f|_e$ denotes the restriction of the function to the edge, and we let $f_e(t)$ for $t\in[0,\ell_e]$ and $f_e(s)$ denote the value of $f(s)$ with $s=(t, e)$.
For $e\in\mathcal{E}$, we let $L_2(e)$ denote the standard $L_2$ space for $[0,\ell_e]$, equipped with the Lebesgue measure, and for $k\in\mathbb{N}$, $H^k(e)$ denotes the standard Sobolev space of order $k$ on $[0,\ell_e]$. 
We define $L_2(\Gamma)$ as the direct sum 
$L_2(\Gamma) = \bigoplus_{e \in \mathcal{E}} L_2(e)$, with norm ${\|f\|_{L_2(\Gamma)}^2 =  \sum_{e\in\mathcal{E}}\|f_e\|_{L_2(e)}^2}$. Similarly, for $k\in\mathbb{N}$, we define $\widetilde{H}^k(\Gamma) = \bigoplus_{e\in \mathcal{E}} H^k(e)$, which is endowed with the Hilbertian norm $\|f\|_{\widetilde{H}^k(\Gamma)}^2 = \sum_{e\in\mathcal{E}} \|f_e\|_{H^k(e)}^2 < \infty$.
We let $C(\Gamma)$ denote the set of continuous functions on $\Gamma$ and, given $S\subset \Gamma$, $C_c(S)$ is the set of continuous functions with support compactly contained in the interior of $S$. Further, for an edge $e\in\mathcal{E}$, we let $C_c^\infty(e)$ denote the set of infinitely differentiable functions with support compactly contained in the interior of $e$.
Finally, we define the Sobolev space $H^1(\Gamma) = C(\Gamma)\cap \widetilde{H}^1(\Gamma)$ as the space of all continuous functions on $\Gamma$ such that $\|f\|_{H^1(\Gamma)} = \|f\|_{\widetilde{H}^1(\Gamma)}  < \infty$. For $u\in H^1(\Gamma)$, the weak derivative $u'\in L_2(\Gamma)$ is defined as the function whose restriction to any edge $e$ coincides almost everywhere with the weak derivative of $u|_e$, which is well defined because ${u|_e\in H^1(e)}$. Finally, for $s>0$, $s\in\mathbb{R}$, let $\widetilde{H}^s(\Gamma) = (\widetilde{H}^{\lfloor\alpha\rfloor}(\Gamma), \widetilde{H}^{\lceil\alpha\rceil}(\Gamma))_{s-\lfloor s\rfloor}$, where $\lfloor\cdot\rfloor$ denotes the floor operator, $\lceil\cdot\rceil$ denotes the ceiling operator, $\widetilde{H}^0(\Gamma) := L_2(\Gamma)$, and given two Hilbert spaces $H_0,H_1$, $(H_0,H_1)_s$ denotes the real interpolation between the spaces $H_0$ and $H_1$ of order $s$ using the $K$ method, see e.g., \cite[Appendix A]{BSW2022} or \cite{chandlerwildeetal} for further details.

We now briefly define and recall some properties of
 the generalized Whittle--Mat\'ern fields introduced in \cite{bolinetal_fem_graph}.
 First, we introduce the differential operator that
 will be used to define the field. To that end, we let $u$
 be a sufficiently smooth function on $\Gamma$, and let $L$ act on $u$
 at a point $t$ on $e\in\mathcal{E}$ as follows:
 \begin{equation}\label{eq:operatorL}
	(Lu_e)(t) = -\frac{d}{dx_e}\left(\H_e(t) \frac{d}{dx_e} u_e(t)\right) + \kappa_e^2(t)u_e(t),
\end{equation}
where $\frac{d}{dx_e}$ is the standard derivative operator on the edge $e$ when it is identified with the interval $[0, l_e]$, $\kappa$ and $\H$ are sufficiently nice functions. Observe that $Lu = \nabla(a\nabla) u$, where $\nabla : \widetilde{H}^1(\Gamma) \to L_2(\Gamma)$ is the global first order derivative operator that acts locally on each edge $e\in \mathcal{E}$ as $\nabla u(x)= du_e(x)/dx_e =u_e'(x)=Du_e(x)$,for $u\in\widetilde{H}^1(\Gamma)$ and $x\in e$. The simplest assumption on $\kappa$ and $\H$ is as follows:

\begin{Assumption}\label{assump:simplest_assumption}
	The function $\kappa:\Gamma\to\mathbb{R}$ is bounded (i.e. $\kappa\in L^\infty(\Gamma)$) and satisfies
		$\textrm{ess inf}_{s\in\Gamma} \kappa(s) \geq \kappa_0 > 0$,
	for some  $\kappa_0>0$, and $\H:\Gamma\to\mathbb{R}$ is a positive Lipschitz function. As $\Gamma$ is compact, the minimum of $\H$ is achieved for some $s_0\in\Gamma$. Thus, with $\H_{0} = \H(s_0)$, we have that 
		$\H(s) \geq \H_{0} > 0$ for all $s\in\Gamma$.
\end{Assumption}

We augment $L$ with the 
vertex conditions $\big\{\mbox{$u\in C(\Gamma)$ and $\forall v\in\mathcal{V}: \sum_{e\in\mathcal{E}_v} \partial_e u(v) = 0$} \big\}$, where $\mathcal{E}_v$ denotes the set of edges incident to $v\in\mathcal{V}$. Therefore $L$ has the domain
$$D(L) = \big\{u\in\widetilde{H}^2(\Gamma): \mbox{$u$ is continuous on $\Gamma$ and $\forall v\in\mathcal{V}: \sum_{e\in\mathcal{E}_v} \partial_e u(v) = 0$} \big\},$$
where $\partial_e u(v)$ is the directional derivative of $u$ on the edge $e$ in the directional away from the vertex $v$, that is, if $e = [0,l_e]$, then $\partial_e u(0) = du_e/dx_e(0)$, and $\partial_e u(l_e) = -du_e/dx_e(l_e)$ and $\mathcal{E}_v$ denotes the set of edges connected to the vertex $v$.

\begin{Remark}\label{rem:laplacian}
	The operator $\kappa^2 - \Delta_\Gamma$ in \eqref{eq:Matern_spde} is the special case of $L$ with $\H \equiv 1$ and $\kappa>0$ constant. Equivalently, the Kirchhoff Laplacian  ${\Delta_\Gamma : \mathcal{D}(\Delta_\Gamma) = \mathcal{D}(L) \subset L_2(\Gamma) \rightarrow L_2(\Gamma)}$ is defined 
	by ${\Delta_\Gamma := \oplus_{e\in \mathcal{E}}\Delta_e}$, where $\Delta_e u_e(x) = u_e''(x)$ with $u_e\in H^2(e)$. 
\end{Remark}
According to \cite[Theorem 2.5]{bolinetal_fem_graph}, $L:D(L)\to L_2(\Gamma)$
is a positive-definite and self-adjoint operator with a compact inverse. 
A complete orthonormal system of eigenvectors $\{e_k\}_{k\in \mathbb{N}}$ on $L_2(\Gamma)$, and a corresponding set of 
  eigenvalues $\{\lambda_k\}_{k\in \mathbb{N}}$, exists for the operator $L$. This allows us to study fractional powers of $L$ in the spectral sense.
Thus, for $s\geq 0$, the space $(\dot{H}^s_L(\Gamma),
\|\cdot\|_{\dot{H}_{L}^s(\Gamma)}$) is defined as follows:
\begin{eqnarray}\label{Hdotspace}
	\dot{H}_L^s(\Gamma)=\Bigl\{u\in L_2(\Gamma):\|u\|_{\dot{H}_L^s(\Gamma)}:=\Bigl(\sum_{k\in \mathbb{N}}\lambda_{k}^{s}(u,e_k)^{2}_{L_2(\Gamma)}\Bigr)^{\nicefrac12}<\infty\Bigr\}.
\end{eqnarray}
For $\alpha>0$ and for $\phi\in D(L^{\alpha/2}) =\dot{H}^{\alpha}_L(\Gamma)$, 
we define the action of the operator $L^{\alpha/2}$ on $\phi$ as 
$
{L^{\alpha/2} \phi = \sum_{k\in \mathbb{N}_0}\lambda_{k}^{\alpha/2}(u,e_k)_{L_2(\Gamma)} e_k}.
$
We obtain $\dot{H}^0_L(\Gamma) \cong L_2(\Gamma)$, and let $\dot{H}^{-s}_L(\Gamma)$ denote the dual space of $\dot{H}^s_L(\Gamma)$, whose norm is defined as follows:
\begin{eqnarray}
	\|g\|_{\dot{H}_L^{-s}(\Gamma)}
	=\sup_{\phi\in \dot{H}_L^s(\Gamma)\setminus \{0\} } \frac{\langle g,\phi \rangle }{\|\phi\|_{\dot{H}_L^s(\Gamma)}}
	=\Bigl(\sum_{k\in{ \mathbb{N}}}\lambda_{k}^{-s}\langle g,e_k\rangle^2\Bigr)^{\nicefrac12},
\end{eqnarray}
where $\langle\cdot,\cdot\rangle$ denotes the pairing between $\dot{H}^{-s}_L(\Gamma)$ and $\dot{H}^{s}_L(\Gamma).$

\begin{Remark}\label{rmk:isometry}
For $0\leq \alpha \leq \beta$, 
$\dot{H}^{\beta}_L(\Gamma) \hookrightarrow \dot{H}^{\alpha}_L(\Gamma) \hookrightarrow L_2(\Gamma) \hookrightarrow \dot{H}^{-\alpha}_L(\Gamma) \hookrightarrow \dot{H}^{-\beta}_L(\Gamma)$. By \cite[][Lemma 2.1]{BKK2020}, $L^\beta$ has a unique bounded extension to an isometric isomorphism $L^\beta : \dot{H}^{s}_L(\Gamma) \rightarrow \dot{H}_L^{s-2\beta}(\Gamma)$ for $s\in\mathbb{R}$.
\end{Remark}
We can now define the generalized Whittle--Mat\'ern field
with exponent $\alpha>1/2$ as the solution of the following fractional
stochastic differential equation:
\begin{align}\label{eq:genWM}
	L^{\alpha/2} u=\mathcal{W}\quad\text{on } \Gamma,
\end{align}
where $\mathcal{W}$ represents Gaussian white noise defined on $L_2(\Gamma)$ with respect to $(\Omega, \mathcal{F},\mathbb{P})$. Specifically, $\mathcal{W}$ can be represented as a family of centered Gaussian variables $\{\mathcal{W}(h) : h\in L_2(\Gamma)\}$ satisfying ${\forall h,g\in L_2(\Gamma)}$, $\pE[\mathcal{W}(h)\mathcal{W}(g)] = (h,g)_{L_2(\Gamma)}$.
From \cite[Proposition 3.2]{bolinetal_fem_graph}, if $\alpha > \nicefrac12$ and Assumption \ref{assump:simplest_assumption} holds, then
\eqref{eq:genWM} has a unique solution
$u \in L_2(\Omega, L_2(\Gamma))$. That is, $u$ is an $L_2(\Gamma)$-valued random variable with a finite second moment. Furthermore, recall that from \cite{bolinetal_fem_graph}, the covariance operator of $u$ is $L^{-\alpha}$.

\begin{Remark}\label{rem:CMgeneralizedWM}
    If Assumption \ref{assump:simplest_assumption} holds, $\alpha>\nicefrac12$, and $u$ is the solution \eqref{eq:genWM}, then 
    it follows from the Karhunen--Lo\`eve expansion of $u$ \cite[Proposition 4.7]{bolinetal_fem_graph} that
    the Cameron--Martin space of $u$ is $\dot{H}_L^\alpha(\Gamma).$
\end{Remark}

Now, let us observe that even though the generalized Whittle--Mat\'ern fields obtained as solutions to \eqref{eq:genWM} are, in general, not isotropic random fields, they enjoy an invariance property. To this end, let 
$$\mathcal{S}_{\alpha,L}(\Gamma)  = \{T:L_2(\Gamma)\to L_2(\Gamma): T \hbox{ is unitary and } TL^{-\beta} = L^{-\beta} T\}$$
be the symmetry group of $L^{-\alpha/2}$ in $\Gamma$. Then, for every $T\in \mathcal{S}_{\alpha,L}(\Gamma)$,  $Tu$ has covariance operator 
$$TL^{-\alpha}T^* = T L^{-\alpha/2} L^{-\alpha/2} T^{-1} = L^{-\alpha/2} T L^{-\alpha/2} T^{-1} = L^{-\alpha} TT^{-1} = L^{-\alpha}.$$
Thus, since centered Gaussian distributions are determined by their covariance operators, $u$ and $Tu$ have the same distribution, so the distribution of generalized Whittle--Mat\'ern fields are invariant with respect to the symmetry group of $L^{-\alpha/2}$.

\subsection{Markov properties of generalized Whittle--Mat\'ern fields}\label{sec:GWMmarkov}

As mentioned in the introduction, an issue with the differentiable generalized Whittle--Mat\'ern fields (which we will formally prove in Section~\ref{sec:GWMnonregular}) is that they are nonregular. Not being able to assume regularity complicates the analysis of Markov properties as we cannot follow the approach of \cite{Pitt1971}.
However, the key to deriving Markov properties 
is to characterize the $\dot{H}_L^{\alpha}(\Gamma)$ spaces for $\alpha\in\mathbb{N}$. For this, we need to recall some notations for function spaces defined on intervals in $\mathbb{R}$. For $k\in\mathbb{N}$, and $I\subset \mathbb{R}$ an interval (recall that we identify edges as intervals) we let $C^k(I)$ denote the space of functions with $k$ continuous derivatives in $I$, and $C^0(I)$ the space of continuous functions in $I$. We also let $C^{k,s}(I)$, where $k\in\mathbb{N}$ and $s\in (0,1]$, denote the space of functions with $k$ continuous derivatives in $I$ and whose $k$th derivative is $s$-H\"older continuous, and we recall that functions that are $1$-H\"older continuous are Lipschitz functions.
For this, we extend Assumption \ref{assump:simplest_assumption} for higher values of $\alpha$ and then obtain a characterization of the spaces $\dot{H}^\alpha_L(\Gamma)$.

\begin{Assumption}\label{assump:basic_assump2}
	Let Assumption \ref{assump:simplest_assumption} hold. In addition, if $\alpha\geq 3$, for every edge ${e\in\mathcal{E}}$,  $\kappa_e\in C^{\floor{\alpha}-3,1}(e)$ and $\H_e\in C^{\floor{\alpha}-2,1}(e)$. That is, $\kappa_e\in C^{\floor{\alpha}-3}(e), \H_e\in C^{\floor{\alpha}-2}(e)$ and $\kappa_e^{(\floor{\alpha}-3)}(\cdot),\H_e^{(\floor{\alpha}-1)}(\cdot)$ are Lipschitz. Here, $\lfloor \zeta\rfloor$ is the largest integer $k$ such that $k\leq \zeta$. 
\end{Assumption}
%

\begin{Proposition}\label{prp:CharGen_CM_natural_alpha}
	Let Assumption \ref{assump:basic_assump2} hold and $\alpha\in\mathbb{N}$. If $\alpha = 1$, $\dot{H}^\alpha_L(\Gamma) \cong \widetilde{H}^\alpha(\Gamma) \cap C(\Gamma)$ and if $\alpha\geq 2$,
		\begin{equation*}\label{eq:HdotAlphaNatural}
			\dot{H}^\alpha_L(\Gamma) \cong \left\{f\in \widetilde{H}^\alpha(\Gamma): L^{\left\lfloor \nicefrac{(\alpha-1)}{2}\right\rfloor} f \in C(\Gamma), \,\forall m = 0,\ldots, \left\lfloor\nicefrac{(\alpha-2)}{2} \right\rfloor, L^{m} f\in\dot{H}^2_L(\Gamma) \right\}.
		\end{equation*}
		Furthermore, the norms $\|\cdot\|_{\dot{H}_L^\alpha(\Gamma)}$ and $\|\cdot\|_{\widetilde{H}^\alpha(\Gamma)}$ are equivalent on $\dot{H}^\alpha_L(\Gamma)$.
\end{Proposition}

Proposition \ref{prp:CharGen_CM_natural_alpha} shows that, under Assumption \ref{assump:basic_assump2}, the spaces $\dot{H}^\alpha_L(\Gamma)$ can be identified as $\widetilde{H}^\alpha(\Gamma)$ with some vertex conditions (that act as constraints on $\widetilde{H}^\alpha(\Gamma)$). In particular, we obtain the following result:

\begin{Proposition}\label{prp:Hdot3Local}
	If $\alpha\in\mathbb{N}$ and Assumption \ref{assump:basic_assump2} holds, then the space $\dot{H}_L^{\alpha}(\Gamma)$ is local.
\end{Proposition}

We can now state our first main result, the Markov property of the generalized Whittle--Mat\'ern fields.

\begin{Theorem}\label{thm:MarkovGenWM}
	Let $u$ be a generalized Whittle--Mat\'ern field obtained as the solution to \eqref{eq:Matern_spde_general},
	with $\alpha > 1/2$. 
	\begin{enumerate}[label= \roman{enumi}.]
		\item Under Assumption \ref{assump:basic_assump2}, if $\alpha\in\mathbb{N}$, then $u$ is a GMRF. \label{thm:MarkovGenWM2}
		\item Let Assumption \ref{assump:simplest_assumption} hold and, for every edge $e\in\mathcal{E}$, $\kappa_e, \H_e \in C^\infty(e)$. If $\alpha\in\mathbb{N}$, then $u$ is a GMRF. Conversely, if $\kappa$ and $a$ are constant functions, and $u$ is a GMRF, then $\alpha\in\mathbb{N}$. \label{thm:MarkovGenWM3}
	\end{enumerate}
\end{Theorem}
\begin{Remark}
The assumptions of Theorem \ref{thm:MarkovGenWM}.\ref{thm:MarkovGenWM3} are satisfied if $\kappa$ and $\H$ are constants, which means that the Whittle--Mat\'ern fields are GMRFs if and only if $\alpha\in\mathbb{N}$. Observe that typical examples of coefficients that satisfy assumption \ref{thm:MarkovGenWM2} from Theorem \ref{thm:MarkovGenWM} but not assumption \ref{thm:MarkovGenWM3} are functions whose restriction to edges are piecewise smooth functions. For example, for $\alpha=1$, continuous functions whose restrictions to edges are piecewise linear functions are Lipschitz functions that do not satisfy assumption~\ref{thm:MarkovGenWM3}.
\end{Remark}

\subsection[Order alpha Markov Properties]{Order $\alpha$ Markov properties}\label{sec:GWMmarkov_alpha}
We now  refine Theorem~\ref{thm:MarkovGenWM} by characterizing the boundary $\sigma$-algebra $\mathcal{F}_+^u(\partial S)$. We aim to demonstrate that, if $\alpha\in\mathbb{N}$, the generalized Whittle--Mat\'ern field is Markov of order $\alpha$. For this, we first define weak mean-square differentiability. 
Suppose $u$ is a GRF with corresponding Gaussian space $H(\Gamma)$,  $e$ is an edge, and $v:e \to H(\Gamma)$ a function. Then, $v$ is weakly differentiable at $s = (t,e)\in\Gamma$ in the $L_2(\Omega)$ sense if  $v'(s)\in H(\Gamma)$ exists such that, for every $w\in H(\Gamma)$ and every sequence $s_n\to s$, where $s_n\neq s$, $\mathbb{E}(w(v(s_n)-v(s))/(s_n-s)) \to \mathbb{E}(wv'(s))$. 
We define higher-order weak derivatives in the $L_2(\Omega)$ sense inductively: for $k\geq 2$, $v$ has a $k$th order weak
derivative at $s = (t,e) \in \Gamma$ if $v_e^{(k-1)}(\tilde{t})$ exists for every $\tilde{t}\in [0,\ell_e]$, and it is weakly differentiable at $t$. 
Finally, a function $v: e \to H(\Gamma)$ is weakly continuous in the $L_2(\Omega)$ sense if, for every $w\in H(\Gamma)$, the function $s\mapsto \mathbb{E}(w v(s))$ is continuous.

\begin{Definition}\label{def:RandomFieldOrderP}
	Let $u$ be a GRF on $\Gamma$ that has weak derivatives, in the 
	$L_2(\Omega)$ sense, of orders $1,\ldots, p$ for $p\in\mathbb{N}$.
	Also assume that $u_e^{(j)}$, for $j=0,\ldots,p$ and $e\in\mathcal{E}$, is weakly continuous in the 
	$L_2(\Omega)$ sense.
	Then the weak derivatives are well-defined for each $s\in\Gamma$ and we say that $u$ is a differentiable GRF of order $p$. A GRF that is continuous in $L_2(\Omega)$ is said to be a differentiable GRF of order $0$.
\end{Definition} 

To simplify the  notation, we define $\mathcal{E}_s$ for $s\in\Gamma$ as the set of edges incident to $s$ and note that if $s$ is an interior point of $e$, then the only edge incident to $s$ is $e$. Further, for a set $S\subset\Gamma$, $\mathcal{E}_S$ denotes the set of edges with a nonempty intersection with the interior of the set $S$.  We also introduce
\begin{align*}
	H_\alpha(\partial S) &= \textrm{span}\{u_e(s), u_e'(s),\ldots, u_e^{(\alpha-1)}(s):
	s\in \partial S, e\in \mathcal{E}_s \},\\
	\mathcal{F}^u_{\alpha}(\partial S) &= \sigma(u_e(s), u_e'(s),\ldots, u_e^{(\alpha-1)}(s):
	s\in \partial S, e\in \mathcal{E}_s).
\end{align*}   
First, we show that  the weak derivatives in the $L_2(\Omega)$ sense 
	only depend on the boundary data.

\begin{Proposition}\label{prp:boundaryweakderiv}
	If $u$ is a differentiable GRF of order $p$, where $p\geq 0$, then
	given any open set $S$ whose boundary is given by finitely many points,
	we obtain, for each $s\in\partial S$ and each $e\in \mathcal{E}_s$, that 
	$u_e(s), u_e'(s), \ldots, u_e^{(p)}(s) \in H_+(\partial S).$
\end{Proposition}
Second, under suitable assumptions on $\kappa$ and $\H$, a generalized Whittle--Mat\'ern field $u$ with $\alpha\geq 1$ is a differentiable GRF of order $\lfloor\alpha\rfloor-1$.

\begin{Proposition}\label{prp:GenWMDiffAlpha}
	Let $u$ be a generalized Whittle--Mat\'ern field on $\Gamma$ obtained as the solution to \eqref{eq:Matern_spde_general} under Assumption \ref{assump:basic_assump2} with $\alpha>1/2$. 
	Then, $u$ is a differentiable GRF of order $\lfloor \alpha\rfloor-1$.
Further,  for $f\in \widetilde{H}^2(\Gamma)$, $e\in\mathcal{E}$ and $t\in [0,\ell_e]$, define $L_e f_e(t) = \kappa_e^2 f_e(t) - \nabla \H_e(t)\nabla f_e(t)$. Let $w\in H(\Gamma)$, and define $h(s) = \mathbb{E}(wu(s))$ for $s\in\Gamma$. Then, for $e\in\mathcal{E}, t\in [0,\ell_e],$
\begin{align*} 
	L_e^k h_e(t) &= \mathbb{E}(w L_e^k u_e(t))   &&\mbox{for $0 \leq k \leq \lfloor \nicefrac{\alpha}{2}\rfloor$,} \\
	\partial_e L_e^k h_e(t) &= \mathbb{E}(w \partial_e L_e^k u_e(t))  &&\mbox{for $0 \leq k \leq \lfloor \nicefrac{\alpha}{2}\rfloor-1$},
	\end{align*}
	where the derivatives of $u_e(\cdot)$ are weak derivatives in the $L_2(\Omega)$ sense.
\end{Proposition}

As a corollary to this result, $u$ satisfies (in the weak $L_2(\Omega)$ sense) the same boundary conditions as the functions in $\dot{H}^\alpha_L(\Gamma)$.

\begin{Corollary}\label{cor:BoundaryConditionsGen_WM}
	Let $u$ be a generalized Whittle--Mat\'ern field on $\Gamma$ obtained as the solution to \eqref{eq:Matern_spde_general} under Assumption \ref{assump:basic_assump2} with $\alpha>1/2$. 
	 Fix any $k \in \{0,\ldots, \ceil{\alpha - \nicefrac{1}{2}}  -1\}$. 
	If $k$ is odd, then $\sum_{e\in\mathcal{E}_v} \partial_{e} L_e^{\nicefrac{(k-1)}{2}} u(v) = 0$ for each $v\in\mathcal{V}$. 
	If $k$ is even, then $L_e^{\nicefrac{k}{2}} u_e(v) = L_{e'}^{\nicefrac{k}{2}}u(v)$
	for each $v\in\mathcal{V}$ and each pair $e,e'\in\mathcal{E}_v$,  where $L_e$ is as in Proposition \ref{prp:GenWMDiffAlpha} and the derivatives are weak in the $L_2(\Omega)$ sense.
\end{Corollary}

Finally, we can characterize the boundary spaces $H_+(\partial S)$ of generalized Whittle--Mat\'ern fields. 

\begin{Theorem}\label{thm:boundarySpcsGenWM}
	Let $u$ be a generalized Whittle--Mat\'ern field on $\Gamma$ obtained as the solution to \eqref{eq:Matern_spde_general} under Assumption \ref{assump:basic_assump2} with $\alpha\in\mathbb{N}$. 
	Let, also, $S\subset\Gamma$ be an open set whose
	boundary consists of finitely many points. Then, $H_+(\partial S) = H_\alpha(\partial S)$ and
	$\mathcal{F}^u_+(\partial S) = \mathcal{F}^u_{\alpha}(\partial S).$
\end{Theorem}

In conclusion, the generalized Whittle--Mat\'ern fields with $\alpha\in\mathbb{N}$ and suitable assumptions on $\kappa$ and $\H$ are GMRFs of order $\alpha$.
Theorem \ref{thm:boundarySpcsGenWM} in combination with Definition~\ref{def:MarkovPropertyField} and Proposition~\ref{prp:charmarkovspaces} also allow us to restate the Markov property for generalized Whittle--Mat\'ern fields more precisely:

\begin{Corollary}\label{cor:MarkovOrderp}
	Let $u$ be a generalized Whittle--Mat\'ern field on $\Gamma$ obtained as the solution to \eqref{eq:Matern_spde_general} under Assumption \ref{assump:basic_assump2} with $\alpha\in\mathbb{N}$. Then, for any open set $S\subset\Gamma$ such that $\partial S$ consists
	of finitely many points:
	\begin{enumerate}[label= \roman{enumi}.]
		\item $\mathcal{F}^u_{\alpha}(\partial S)$ splits $\mathcal{F}^u_+(\overline{S})$ and $\mathcal{F}_+^u(\Gamma\setminus S)$;
		\item the orthogonal projection of $H_+(\overline{S})$ on $H_+(\Gamma\setminus S)$ is $H_\alpha(\partial S)$;
		\item $H(\Gamma) = H_+(\overline{S})\oplus(H_+(S^c)\ominus H_\alpha(\partial S)) = H_+(S^c)\oplus (H_+(\overline{S})\ominus H_\alpha(\partial S)).$
	\end{enumerate}
\end{Corollary}

\subsection{Nonregularity of generalized Whittle--Mat\'ern fields}\label{sec:GWMnonregular}
For a differentiable random field of order $p$, $p\in\mathbb{N}$, satisfying Kirchhoff vertex conditions on a metric graph, we expect the inner and outer boundary spaces at the vertices to be different whenever ${\alpha>1}$, and the degree of the vertex is greater than two. We now show that this is, indeed, the case for generalized Whittle--Mat\'ern fields.
First, we study the field on the inner and outer boundaries of an open set ${S\subset\Gamma}$. To this end, we define
	\begin{align*}
		H_{\alpha,I}(\partial S) &= \textrm{span}\{u_e(s), u_e'(s),\ldots, u_e^{(\alpha-1)}(s):
s\in \partial S, e\in \mathcal{E}_S \}, \\
	\mathcal{F}_{\alpha,I}^u(\partial S) &= \sigma(u_e(s), u_e'(s),\ldots, u_e^{(\alpha-1)}(s):
	s\in \partial S, e\in \mathcal{E}_S),
	\end{align*} 
	and introduce
\begin{equation}\label{eq:defSpaceInner}
	H_{+,I}(\partial S) = \bigcap_{\varepsilon>0} H(S\cap (\partial S)_\varepsilon), \quad
	H_{+,O}(\partial S) = \bigcap_{\varepsilon>0} H((\Gamma\setminus \overline{S})\cap (\partial S)_\varepsilon).
\end{equation}

We have the following characterization of these spaces and as a corollary we obtain the regularity result for the generalized Whittle--Mat\'ern fields.

\begin{Proposition}\label{thm:charInnerOuterEdge}
		Let $u$ be a generalized Whittle--Mat\'ern field on $\Gamma$ obtained as the solution to \eqref{eq:Matern_spde_general} under Assumption \ref{assump:basic_assump2} with $\alpha\in\mathbb{N}$.  If $S\subset\Gamma$ is an open set whose
	boundary $\partial S$ consists of finitely many points, then
	$H_{+,I}(\partial S) = H_{\alpha,I}(\partial S)$
	and
	$$
	H_{+,O}(\partial S) = \textrm{span}\{u_{\widetilde{e}}(s), u_{\widetilde{e}}'(s),\ldots, u_{\widetilde{e}}^{(\alpha-1)}(s):
	s\in \partial S, \widetilde{e}\not\in \mathcal{E}_S \}.
	$$
\end{Proposition}

\begin{Corollary}\label{cor:alpha1regularboundary}
Let $u$ be a generalized Whittle--Mat\'ern field on $\Gamma$ obtained as the solution to \eqref{eq:Matern_spde_general} under Assumption \ref{assump:basic_assump2} with $\alpha\in\mathbb{N}$. Then $H_{+,O}(\partial e) = H_+(\partial e)$ for any edge $e$ of $\Gamma$. 
	Further, if $\alpha=1$, then 
	$H_{+,I}(\partial e) = H_{+,O}(\partial e) = H_{+}(\partial e) = \textrm{span}\{u(s):s\in\partial e\}.$
	Finally, if $\alpha>1$ and at least one of the vertices in $\partial e$ has degree greater
	than 2, then
	$H_{+,I}(\partial e) \subsetneqq H_{+,O}(\partial e).$
\end{Corollary}

Thus, generalized Whittle--Mat\'ern fields with $\alpha=1$ are regular, whereas 
if $\alpha\in\mathbb{N}$ is greater than 1, and the 
metric graph $\Gamma$ has at least one vertex with degree greater than 2,
then the fields are not regular.
Finally, we state the Markov property on the edges in terms of their inner
derivatives. The following result follows directly from Proposition~\ref{thm:charInnerOuterEdge}
and Proposition~\ref{prp:markovinnerboundary}.

\begin{Corollary}\label{cor:MarkovOrderpInner}
	Let $u$ be a generalized Whittle--Mat\'ern field on $\Gamma$ obtained as the solution to \eqref{eq:Matern_spde_general} under Assumption \ref{assump:basic_assump2} with $\alpha\in\mathbb{N}$. 
	For any open set $S\subset\Gamma$ such that $\partial S$ consists
	of finitely many points, 
	\begin{enumerate}[label= \roman{enumi}.]
		\item $\mathcal{F}^u_{\alpha, I}(\partial S)$ splits $\mathcal{F}^u(S)$ and $\mathcal{F}^u_+(\Gamma\setminus S)$;
		\item $\mathbb{E}(f|\mathcal{F}^u(S)) = \mathbb{E}(f| \mathcal{F}^u_{\alpha,I}(\partial S))$ for any bounded $\mathcal{F}^u_+(\Gamma\setminus S)$-measurable function $f$;
		\item The orthogonal projection of $H_+(\Gamma\setminus S)$ on $H(S)$ is $H_{\alpha,I}(\partial S)$;
		\item $H(\Gamma) = H_+(\Gamma\setminus S) \oplus (H(S)\ominus H_{\alpha,I}(\partial S)).$
	\end{enumerate}
\end{Corollary}

Corollary \ref{cor:MarkovOrderpInner} is  useful when $S$ is an edge, because the result shows that in order to restrict the generalized Whittle--Mat\'ern field with exponent $\alpha\in\mathbb{N}$ to an edge, given all the information on 
the remaining edges, the only information needed is the inner boundary data.

\subsection{Markov properties and Cameron--Martin spaces}\label{sec:CM}

This section describes the Markov property of the 
generalized Whittle--Mat\'ern fields in terms of their associated Cameron--Martin spaces. 
By applying the isometric isomorphism between the Gaussian spaces and 
Cameron--Martin spaces to Corollaries \ref{cor:MarkovOrderp} and \ref{cor:MarkovOrderpInner}, we obtain the following:

\begin{Corollary}\label{cor:MarkovOrderpCameronMartin}
	Let $u$ be a generalized Whittle--Mat\'ern field on $\Gamma$ obtained as the solution to \eqref{eq:Matern_spde_general} under Assumption \ref{assump:basic_assump2} with $\alpha\in\mathbb{N}$. 
	Define the auxiliary spaces ${\mathcal{H}_\alpha(\partial S) = \{h(s) = \mathbb{E}(u(s)v): s\in\Gamma, v\in {H}_\alpha(\partial S)\}}$ and $
	{\mathcal{H}_{\alpha,I}(\partial S) = \{h(s) = \mathbb{E}(u(s)v): s\in\Gamma, v\in {H}_{\alpha,I}(\partial S)\}}
$, where $S\subset\Gamma$ is any open set such that $\partial S$ consists
	of finitely many points .
	Then,
	\begin{enumerate}[label= \roman{enumi}.]
		\item $\mathcal{H}(\Gamma) = \mathcal{H}_+(\overline{S})\oplus(\mathcal{H}_+(\Gamma\setminus S)\ominus \mathcal{H}_\alpha(\partial S)) = \mathcal{H}_+(\Gamma\setminus S)\oplus (\mathcal{H}_+(\overline{S})\ominus \mathcal{H}_\alpha(\partial S))$;\label{cor:MarkovOrderpCameronMartin1}
		\item The orthogonal projection of $\mathcal{H}_+(\overline{S})$ on $\mathcal{H}_+(\Gamma\setminus S)$ is $\mathcal{H}_+(\partial S)$;\label{cor:MarkovOrderpCameronMartin2}
		\item $\mathcal{H}(\Gamma) = \mathcal{H}_+(\Gamma\setminus S) \oplus (\mathcal{H}(S)\ominus \mathcal{H}_{\alpha,I}(\partial S))$.\label{cor:MarkovOrderpCameronMartin3}
	\end{enumerate}
\end{Corollary}

We now characterize the spaces $\mathcal{H}_+(\Gamma\setminus S)$,
$\mathcal{H}_+(\overline{S})\ominus \mathcal{H}_\alpha(\partial S)$ and $\mathcal{H}(S)\ominus \mathcal{H}_{\alpha,I}(\partial S)$. We define  $\mathcal{H}_{+,0}(S) = \mathcal{H}_+(\overline{S})\ominus \mathcal{H}_\alpha(\partial S)$ and ${\mathcal{H}_{0,I}(S) = \mathcal{H}(S)\ominus \mathcal{H}_{\alpha,I}(\partial S)}$ to simplify the notation.
We begin by characterizing $\mathcal{H}_{+,0}(S)$ and $\mathcal{H}_{0,I}(S)$.

\begin{Proposition}\label{thm:CharOrthComplCM_genWM}
	Let $u$ be a generalized Whittle--Mat\'ern field on $\Gamma$ obtained as the solution to \eqref{eq:Matern_spde_general} under Assumption \ref{assump:basic_assump2} with $\alpha\in\mathbb{N}$. 
	For any open set $S\subset\Gamma$ such that the boundary consists
	of finitely many points, 
	$(\mathcal{H}_{+,0}(S), \|\cdot\|_{\widetilde{H}^\alpha(\Gamma)}) \cong (\mathcal{H}_{0,I}(S), \|\cdot\|_{\widetilde{H}^\alpha(\Gamma)})\cong (\dot{H}_{0,L}^\alpha(S), \|\cdot\|_{\widetilde{H}^\alpha(S)}),$
	where $\dot{H}_{0,L}^\alpha(S)$ is the completion of $C_c(S)\cap \dot{H}_L^\alpha(S)$ 
	with respect to the norm $\|\cdot\|_{\widetilde{H}^\alpha(S)}$. In particular, for any $h\in \dot{H}_{0,L}^\alpha(S)$, the function $\widetilde{h}$, given by the extension to $\Gamma$ as zero in $\Gamma\setminus S$, belongs to $\dot{H}^\alpha_L(\Gamma)$.
\end{Proposition}

\begin{Remark}\label{rem:zeroboundarydata}
	Proposition~\ref{thm:CharOrthComplCM_genWM} shows that $\mathcal{H}_{+,0}(S)$ and $\mathcal{H}_{0,I}(S)$ are given by the functions in $\dot{H}_L^\alpha(S)$ that are zero at the boundary of $S$, and whose weak derivatives up to order $\alpha-1$ are also zero at the boundary.
\end{Remark}

Finally, we characterize the space $\mathcal{H}_+(\Gamma\setminus S)$, where $S$ is an open set whose boundary consists of finitely many points 
(by taking $S$ as $\Gamma\setminus\overline{S}$, 
we obtain the characterization of $\mathcal{H}_+(\overline{S})$).

\begin{Proposition}\label{prp:ChatHPlusS_genWM}
 Let $u$ be a generalized Whittle--Mat\'ern field on $\Gamma$ obtained as the solution to \eqref{eq:Matern_spde_general} under Assumption \ref{assump:basic_assump2} with $\alpha\in\mathbb{N}$. 
 For any open set $S\subset\Gamma$ such that the boundary consists
	of finitely many points,  a real function $h$ defined on
	$\Gamma$ belongs to $\mathcal{H}_+(\Gamma\setminus S)$ if and only if 
	$h\in\dot{H}_L^\alpha(\Gamma)$ and
	$L^\alpha h = 0$ on $S$.
	Furthermore, $h\in \mathcal{H}_+(\partial S)$ if and only if $h\in\dot{H}_L^\alpha(\Gamma)$ and 
	$L^\alpha h = 0$ on $\Gamma\setminus\partial S$.
\end{Proposition}

\section{Proofs of the main results in Section \ref{sec:GWM}}\label{sec:proofsGWM}

In this section, we present the proofs of the main results from Section \ref{sec:GWM}. The proofs of the additional results are provided in Appendix~\ref{app:aux}. We begin by proving the characterization of the spaces $\dot{H}^\alpha_L(\Gamma)$, when $\alpha\in\mathbb{N}$. This result uses a few auxiliary lemmas that are stated and proved in Appendix~\ref{app:aux}.

\begin{proof}[Proof of Theorem \ref{thm:MarkovGenWM}]
	Let $\alpha\in\mathbb{N}$.  By Remark~\ref{rem:CMgeneralizedWM},  $\mathcal{H}(\Gamma) \cong \dot{H}_L^\alpha(\Gamma),$
	where $\mathcal{H}(\Gamma)$ is the Cameron--Martin space associated with $u$.
	Combining this with Proposition \ref{prp:Hdot3Local}, it is immediate that the Cameron--Martin spaces of 
	$u$ are local. Therefore, by Theorem~\ref{thm:MarkovLocalCM}, $u$ is a 
	GMRF. This proves \ref{thm:MarkovGenWM2} and the direct part of \ref{thm:MarkovGenWM3}.
	
	Conversely, let $\alpha\not\in\mathbb{N}$. By Lemma \ref{lem:local_coeff_Peetre}, we have that there exist $N\in\mathbb{N}$ and $b_r\in C^\infty(e)$, $r=0,\ldots,N$, such that for every ${f \in \bigoplus_{e\in\mathcal{E}} C_c^\infty(e)}$, 
	\begin{equation}\label{eq:peetre_finite}
		L^\alpha f(s) = \sum_{e\in\mathcal{E}} \sum_{r=0}^{N} \widetilde{b}_{r}^e(s) \frac{d^r \widetilde{f}_e(s)}{dx_e^r}(s),\quad s\in \Gamma,
	\end{equation}
	where $\widetilde{b}^e_{r}$ and $\widetilde{f}_e$ are the extensions as zero in $\Gamma\setminus e$, that is, if $s\not\in e$, $\widetilde{b}^e_{r}(s) = \widetilde{f}_e(s) = 0$, and for $s\in e$, say $s = (t,e)$, we have $\widetilde{b}^e_{r}(s) = b^e_{r}(t)$ and $\widetilde{f}_e(s) = f_e(t)$. 
	We will now show that this contradicts the
	fact that $\alpha\not\in\mathbb{N}$. To this end, let us obtain an expansion for $L^\alpha$ on a convenient function space. To this end, since $\alpha$ is not a natural number, we can use the Balakrishnan formula, see e.g., \cite{balakrishnan} or \cite[Chapter 9, section 11]{yosida2012functional}, to obtain that, by Proposition \ref{prp:CinfinityCompactHdot}, for every $f\in \bigoplus_{e\in\mathcal{E}} C^\infty_c(e)$, we have
	\begin{equation}\label{eq:fractional_bala}
		L^\alpha f = L^{\lceil \alpha \rceil} L^{\alpha - \lceil \alpha\rceil} f = L^{\ceil{\alpha}}\frac{\sin(\pi\beta)}{\pi} \int_0^\infty \lambda^{-\beta} (\lambda + L)^{-1}f d\lambda,
	\end{equation}
	where $\beta = \lceil\alpha\rceil - \alpha$. For an edge $e = [0,l_e]$ and a point $x$ in the interior of $e$, let $D_{x,e}$ be the set of functions in $\bigoplus_{e\in\mathcal{E}} C^\infty_c(e)$ such that there exist an interval $(I_1,I_2)$ such that ${0 < I_1 < x < I_2 < l_e}$ and $m\in\mathbb{N}$ so that $\bigl( \frac{d}{dx_e} a_e \frac{d}{dx_e}\bigr)^m f = 0$ on $(I_1,I_2)$. By Lemma~\ref{lem:Cameron_Martin_decomposition}, we have that for every $m_0\in\mathbb{N}$, we can find $f\in D_{x,e}$ such that $\bigl( \frac{d}{dx_e} a_e \frac{d}{dx_e}\bigr)^{m_0-1} f(x) \neq 0$ and $\bigl( \frac{d}{dx_e} a_e \frac{d}{dx_e}\bigr)^{m_0} f = 0$ on $(x-\delta,x+\delta)$ some for some $\delta>0$, we will now use this to create a contradiction. We start by noting by definition $D_{x,e}$  that for any $f\in D_{x,e}$, there exists a neighborhood of $x$ such that the following sum is finite:
	$$\sum_{k=0}^\infty \left(\frac{d}{dx_e} a_e \frac{d}{dx_e}\right)^k f(t) = \sum_{k=0}^m \left( \frac{d}{dx_e} a_e \frac{d}{dx_e}\right)^k f(t),$$
	for some $m\in\mathbb{N}$ and $t \in (x-\delta, x+\delta)$ for some $\delta>0$. In particular, it is straightforward to check (this is the point where we need to use that $\kappa$ is a constant function) that for ${t \in (x-\delta, x+\delta)}$,
	$$(\lambda + L) \sum_{k=0}^\infty (\lambda + \kappa^2)^{-k-1}\left(\frac{d}{dx_e} a_e \frac{d}{dx_e}\right)^k f(t)  = f(t).$$
	This gives us that we can express $(\lambda + L)^{-1} f$ in a neighborhood of $x$ as
	\begin{equation}\label{eq:fractional_inv_sum}
		(\lambda + L)^{-1} f(t) = \sum_{k=0}^\infty (\lambda + \kappa^2)^{-k-1}\left(\frac{d}{dx_e} a_e \frac{d}{dx_e}\right)^k f(t),
	\end{equation}
	where we observe that the infinite sum is actually a finite sum for each $f\in D_{x,e}$. We can now combine \eqref{eq:fractional_inv_sum} and \eqref{eq:fractional_bala} to obtain that for every $f\in D_{x,e}$ and $t$ in a neighborhood of $x\in e$, we have
	\begin{align}
		L^\alpha f(t) &= L^{\ceil{\alpha}} \sum_{k=0}^\infty \left(\frac{d}{dx_e} a_e \frac{d}{dx_e}\right)^k f(t) \int_0^\infty \lambda^{-\beta} (\lambda + \kappa^2)^{-k-1}d\lambda 
		= L^{\ceil{\alpha}} \sum_{k=0}^\infty C_k(\kappa) \left(\frac{d}{dx_e} a_e \frac{d}{dx_e}\right)^k f(t)\nonumber\\
		&= \sum_{k=0}^\infty C_k(\kappa) \left(\kappa^2 - \frac{d}{dx_e} a_e \frac{d}{dx_e}\right)^{\ceil{\alpha}}\left(\frac{d}{dx_e} a_e \frac{d}{dx_e}\right)^k f(t),\label{eq:identity_infinite_sum}\nonumber\\
		&= \sum_{k=0}^\infty \sum_{j=0}^{\ceil{\alpha}} \binom{\ceil{\alpha}}{j} \kappa^{2(\ceil{\alpha}-j)} C_k(\kappa) (-1)^j\left(\frac{d}{dx_e} a_e \frac{d}{dx_e}\right)^{k + j} f(t)
	\end{align}
	where $C_k(\kappa) = \int_0^\infty \lambda^{-\beta} (\lambda + \kappa^2)^{-k-1}d\lambda$ satisfies $0 < C_k(\kappa)<\infty$ since $\beta \in (0,1)$, and we can pass the sum outside of the integral since it is a finite sum. We are now in a position to conclude the proof. By \eqref{eq:identity_infinite_sum} and \eqref{eq:peetre_finite}, we have that for every $e\in\mathcal{E}$, every $x\in \textrm{int }e$ and every $f\in D_{x,e}$, the following identity holds, where we can use that the function $a$ is constant to simplify:
	\begin{equation}\label{eq:identity_finite_infinite}
		\sum_{k=0}^\infty C_k(\kappa) \sum_{j=0}^{\ceil{\alpha}} \binom{\ceil{\alpha}}{j} \kappa^{2(\ceil{\alpha}-j)} (-a_e)^{j} a_e^k \frac{d^{k + j}}{dx_e^{k + j}} f(t) = \sum_{r=0}^{N} \widetilde{b}_{r}^e(t) \frac{d^r f_e(t)}{dx_e^r}(t).
	\end{equation}
	Now, observe that since $a$ is a constant function, $D_{x,e}$ consists of polynomials on a neighborhood of $x$. By applying identity \eqref{eq:identity_finite_infinite} for $f$ such that it is constant in a neighborhood of $x$, then, take $f_m$ such that it coincides with $(t-x)^m$, for each $m=1,2,3,\ldots,$, we have that $f_m^{(k)}(x) = 0$ if $k\neq m$ and $f^{(m)}(x) = m!\neq 0$, so we obtain that the coefficients from both expansions must coincide at $x$, and that we must have $\widetilde{a}^e_{2k+1}(x) = 0$, for $k=0,1,\ldots$. However, this means that there must exist some natural number $N>0$ such that coefficients on the left-hand side of expression \eqref{eq:identity_infinite_sum} must vanish if $k+j>N$. We will now show that this leads to a contradiction. We have for any $m > \max\{\ceil{\alpha}, N\}$ that
	$$  
	\smashoperator[r]{\sum_{k,j: k+j = m}} C_k(\kappa) \binom{\ceil{\alpha}}{j} \kappa^{2(\ceil{\alpha}-j)}(-1)^j a_e^m = 0 \Rightarrow  \smashoperator[r]{\sum_{k,j: k+j = m}} C_k(\kappa) \binom{\ceil{\alpha}}{j} \kappa^{2(\ceil{\alpha}-j)}(-1)^j = 0.
	$$
	This, in turn, implies, by using the definition of $C_{m-j}(\kappa)$, that
	\begin{align*}&\sum_{j=0}^{\ceil{\alpha}}C_{m-j}(\kappa) \binom{\ceil{\alpha}}{j} \kappa^{2(\ceil{\alpha}-j)}(-1)^j = 0\Rightarrow \sum_{j=0}^{\ceil{\alpha}}C_{m-j}(\kappa) \binom{\ceil{\alpha}}{j} \kappa^{-2j}(-1)^j = 0 \\
	&	\Rightarrow \int_0^\infty \frac{1}{\lambda^{\beta}(\kappa^2+\lambda)^{m-\ceil{\alpha}+1}} \sum_{j=0}^{\ceil{\alpha}} \binom{\ceil{\alpha}}{j} ((\kappa^2 + \lambda)^{-1})^{\ceil{\alpha}-j}(-\kappa^{-2})^j d\lambda = 0.
		\end{align*}
	Finally, this implies that
	$$\int_0^\infty \frac{1}{\lambda^{\beta}(\kappa^2+\lambda)^{m-\ceil{\alpha}+1}} ((\kappa^2 + \lambda)^{-1} - \kappa^{-2})^{\ceil{\alpha}} d\lambda = 0,$$
	which is a contradiction. This contradiction shows that we cannot have $\alpha\not\in\mathbb{N}$.
\end{proof}

Next, we prove Theorem \ref{thm:boundarySpcsGenWM}, using Lemma \ref{lem:sobTraceZero} from Appendix \ref{app:aux}.

\begin{proof}[Proof of Theorem \ref{thm:boundarySpcsGenWM}]
	First, from Proposition \ref{prp:GenWMDiffAlpha}, $u$ is a differentiable
	GRF of order ${\alpha -1}$. Therefore, 
	$H_\alpha(\partial S) \subset H_+(\partial S)$ by Proposition 
	\ref{prp:boundaryweakderiv}.
	We will show the result by proving that $H_+(\partial S)\ominus H_\alpha(\partial S) = \{0\}$.
	To this end, take $v \in H_+(\partial S)\ominus H_\alpha(\partial S)$, and let
	$h(s) = \mathbb{E}(u(s)v)$. 
	Because $\alpha\in\mathbb{N}$, it follows from Remark \ref{rem:CMgeneralizedWM} that the Cameron--Martin space associated with $u$,
	$\mathcal{H}(\Gamma)$, can be identified with $\dot{H}^\alpha_L(\Gamma)$. 
	Therefore, $h\in \dot{H}^\alpha_L(\Gamma)$. Furthermore, it follows from the same arguments in the
	proof of \cite[Lemma 4]{BSW2022} that for $s\in\partial S$ and $e\in\mathcal{E}_s$,
	$
	h_e^{(k)}(s) = \mathbb{E}\left(v u^{(k)}(s) \right)$ for $k=0,\ldots,\alpha-1$.
	Since $v\perp H_p(\partial S)$, $\mathbb{E}\bigl(v u^{(k)}(s) \bigr) = 0$ and, thus, $h_e^{(k)}(s) = 0$, for $k=0,\ldots,\alpha-1$, $s\in\partial S$ and $e\in\mathcal{E}_s$. 
	Also, by applying Lemma~\ref{lem:sobTraceZero} 
	in Appendix~\ref{app:aux} 
	twice (for $S$ and $\Gamma\setminus\overline{S}$), we can approximate
	$h$ by $\varphi_{n,-} + \varphi_{n,+}$, where ${\varphi_{n,-}, \varphi_{n,+}
	\in\dot{H}^\alpha_L(\Gamma)}$, with $\textrm{supp}\,\varphi_{n,-} \subset S$ and 
	$\textrm{supp}\, \varphi_{n,+} \subset \Gamma\setminus\overline{S}$.
	
	Now, $\mathcal{H}_+(\partial S) \subset 
	\mathcal{H}_+(\overline{S}) \cap \mathcal{H}_+(\Gamma\setminus S)$ since $H_+(\partial S)\subset H_+(\overline{S})\cap H_+(\Gamma\setminus S)$, which shows that ${h\in \mathcal{H}_+(\overline{S}) \cap \mathcal{H}_+(\Gamma\setminus S)}$. 
	Further, for every $n\in\mathbb{N}$, $\varphi_{n,-} \in 
	\mathcal{H}_+(\Gamma\setminus S)^\perp$. Indeed, because 
	${\varphi_{n,-}\in \dot{H}^\alpha_L(\Gamma)}$, $v_{n,-}\in H(\Gamma)$ exists such 
	that for every $s\in\Gamma$, ${\varphi_{n,-}(s) = \mathbb{E}(u(s)v_{n,-})}$.
	Furthermore, since $\textrm{supp}\,\varphi_{n,-} \subset S$, 
	for each $n\in\mathbb{N}$, 
	$\varepsilon_n>0$ exists such that ${\varphi_{n,-}(s) = 0}$ if ${s\in (\Gamma\setminus S)_{\varepsilon_n}}$, which implies that for every $s\in (\Gamma\setminus S)_{\varepsilon_n}$,
	$\mathbb{E}(u(s)v_{n,-}) = 0$ (that is, ${v_{n,-} \perp H((\Gamma\setminus S)_{\varepsilon_n})}$). 
	Finally, $v_{n,-} \perp H_+(\partial S)$ because ${H_+(\partial S) \subset H((\Gamma\setminus S)_{\varepsilon_n})}$. 
	By applying the isometric isomorphism between
	the Gaussian spaces and the Cameron--Martin spaces, $\varphi_{n,-}\perp \mathcal{H}_+(\partial S)$.
	It follows that $\varphi_{n,-}\perp h$ because $h\in\mathcal{H}_+(\Gamma\setminus S)$, and since $\textrm{supp}\, \varphi_{n,-}\cap \textrm{supp}\,\varphi_{n,+}=\emptyset$, $h = \varphi_{n,+}$.
	Similarly, we obtain that $\varphi_{n,+}\perp h$, so that
	$h=0$. Therefore, $v = 0$, concluding  the proof.
\end{proof}

\section{The restriction of a generalized Whittle--Mat\'ern field to an edge}\label{sec:edge}

The goal of this section is to obtain a representation of the restriction of a generalized
Whittle--Mat\'ern field $u$ on a compact metric graph $\Gamma$, with $\alpha\in\mathbb{N}$,
to an edge $e\in\mathcal{E}$. 
This representation reveals that the restriction of $u$ to $e$ is given as a sum of boundary terms and an independent GRF with zero boundary data on
the edge. To this end, let us introduce the following operator acting on a sufficiently differentiable function  defined on $e=[0,\ell_e]$:
$$
B^{\alpha} u =  \left[u(0),u^{(1)}(0),\ldots,u^{(\alpha-1)}(0),u(\ell_e),u^{(1)}(\ell_e),\ldots,u^{(\alpha-1)}(\ell_e) \right]^\top,
$$
where the derivatives are weak derivatives in the $L_2(\Omega)$ sense. Some auxiliary results are given in Appendix \ref{app:sec6sec7}. 

In the following theorem, where we prove the aforementioned representation, we will use Proposition~\ref{prp:charHdot_edge_genWM} and Lemmas \ref{lem:FRKHS} and \ref{lem:Cameron_Martin_decomposition} from Appendix~\ref{app:sec6sec7}.

\begin{Theorem}\label{thm:ReprTheoremEdge_GenWM}
	Let $u$ be a generalized Whittle--Mat\'ern
	field obtained as the solution to \eqref{eq:Matern_spde_general}, where  $\alpha\in\mathbb{N}$
	and Assumption \ref{assump:basic_assump2} holds.
	For any $e\in\mathcal{E}$, we obtain the following representation:
	$$
	u_e(t) = v_{\alpha,0}(t) + \sum_{j=1}^{2\alpha} \sol_j(t)  \left(B^{\alpha}u_e\right)_j = v_{\alpha,0}(t) + \mv{\sol}^\top(t) B^\alpha u_e,\quad t\in [0,\ell_e],
	$$
	where $\mv{\sol}(\cdot) = (\sol_1(\cdot),\ldots,\sol_{2\alpha}(\cdot))$ and $\{\sol_j(\cdot): j=1,\ldots, 2\alpha\}$ is a set of linearly independent solutions of the homogeneous linear equation 
	$L^\alpha \sol = 0$ on $e$, satisfying $B^{\alpha} \sol_j=e_j$, where $e_j$ is the $j$th canonical vector. 
	The random vector $B^{\alpha}u$ is measurable with respect to $\mathcal{F}^u_{\alpha,I} (\partial e)$, and 
	$v_{\alpha,0}(\cdot)$ is the GRF defined on $e$ with associated Cameron--Martin space
	$(\dot{H}_{0,L}^\alpha(e), (\cdot,\cdot)_{\alpha,e})$. This field is independent of $\mathcal{F}^u_{\alpha,I}(\partial e)$,
	is $\alpha-1$ times weakly differentiable in the $L_2(\Omega)$ sense, and its weak derivatives in $L_2(\Omega)$
	are weakly continuous 
	in the $L_2(\Omega)$ sense. Finally, $v_{\alpha,0}(\cdot)$ has zero boundary data (i.e. $B^{\alpha}v_{\alpha,0} = \mv{0}$).
\end{Theorem}

\begin{proof}
	Let $\rho(s,s') = \mathbb{E}(u(s) u(s')),$ for $s,s'\in\Gamma$. The strategy of the proof is as follows. We first use Lemma~\ref{lem:Cameron_Martin_decomposition}
	in Appendix~\ref{app:sec6sec7}
	to establish that, for each edge $e\in\mathcal{E}$, $\rho$ can be decomposed as the sum of two orthogonal kernels $h_1(s,s')$ and $h_2(s,s')$, which are the reproducing kernels of $ \mathcal{H}_+(\Gamma\setminus e)$ and $\mathcal{H}_{0,I}(e)$, respectively.  
	Then, we use the 
	representation of the Cameron--Martin spaces (Proposition~\ref{thm:CharOrthComplCM_genWM} and Proposition \ref{prp:ChatHPlusS_genWM}) to uniquely characterize the kernels $h_1(s,s')$ and $h_2(s,s')$. Finally, we link $h_2(\cdot,\cdot)$ to a unique GRF (by connecting its Cameron--Martin space to its unique Gaussian measure), and link $h_1(\cdot,\cdot)$ to a boundary Gaussian process, that is, a process that is measurable 
	with respect to the inner boundary space $\mathcal{F}^u_{\alpha,I}(\partial e)$.

	Fix $e\in\mathcal{E}$. By Corollary \ref{cor:MarkovOrderpCameronMartin}.\ref{cor:MarkovOrderpCameronMartin3} and Lemma~\ref{lem:Cameron_Martin_decomposition}
	in Appendix~\ref{app:sec6sec7}, 
	 $\rho(\cdot,\cdot) = h_1(\cdot,\cdot) + h_2(\cdot,\cdot)$, where $h_1(\cdot,\cdot)$ is the reproducing kernel of $\mathcal{H}_+(\Gamma\setminus e)$ and $h_2(\cdot,\cdot)$ is the reproducing kernel of $\mathcal{H}_{0,I}(e)$.
	Further, by Proposition~\ref{thm:CharOrthComplCM_genWM}, $\mathcal{H}_{0,I}(e)\cong \dot{H}_{0,L}^\alpha(e)$; thus, given any $h\in \mathcal{H}_{0,I}(e)$ and $s\in e$, let $\widehat{h}$ and $\widehat{\rho}_e(\cdot, s)$ be the elements in $\dot{H}_{0,L}^\alpha(e)$ that correspond to
	$h$ and $h_2(\cdot, s)$. 
	For $s,s'\in e$, $h_2(s,s') = \widehat{\rho}_e(s,s')$ and $h(s) = \widehat{h}(s)$, and by  \eqref{eq:h2ReprKernel},
	$
	(\widehat{h}, \widehat{\rho}_e(\cdot, s))_{\alpha,e} = (h, h_2(\cdot,s))_\alpha = h(s) = \widehat{h}(s),
	$
	since $h\in \mathcal{H}_{0,I}(e)$ and $h_2(\cdot, t)$ is a reproducing kernel for $\mathcal{H}_{0,I}(e)$. 
	Therefore, by Proposition~\ref{prp:charHdot_edge_genWM} from Appendix \ref{app:sec6sec7}, the space $(\dot{H}_{0,L}^\alpha(e), (\cdot, \cdot)_{\alpha,e})$
	is a reproducing kernel Hilbert space with kernel $\widehat{\rho}_e(\cdot,\cdot)$. By the uniqueness of the
	Cameron--Martin space associated with a centered Gaussian measure \cite[e.g.~see][Theorem 2.9]{daPrato2014}, a unique GRF $v_{\alpha,0}(t), t\in [0,\ell_e]$ associated with 
	$(\dot{H}_{0,L}^\alpha(e), (\cdot, \cdot)_{\alpha,e})$ exists and, by Proposition \ref{prp:charHdot_edge_genWM} from Appendix \ref{app:sec6sec7}, it is associated
	with $(\mathcal{H}_{0,I}(e), (\cdot,\cdot)_{\dot{H}^\alpha_L(\Gamma)})$. 

We now show that $\mathcal{H}_+(\Gamma\setminus e)$ is a finite-dimensional space. By Proposition \ref{prp:ChatHPlusS_genWM}, and its proof,  $h_{1}(s,\cdot)$ is a classical 
solution of $L^\alpha w = 0$ on $e$. This homogeneous linear equation on $e$ of order $2\alpha$ has a 
$2\alpha$-dimensional space of solutions which is uniquely determined by $L^\alpha$,
$\alpha\in\mathbb{N}$, on $e$. Let $\mathcal{\sol} = \textrm{span}\{\sol_1(\cdot), \ldots, \sol_{2\alpha}(\cdot)\}$ be the space of such solutions. By standard theory of linear ordinary differential equations, we can assume, without loss of generality, that $\mv{B}^\alpha \sol_j = e_j$, where $e_j$ is the $2\alpha$-dimensional vector with entry 1 at position $j$ and zero elsewhere.
Thus, since $h_1(\cdot,s)$ and $h_1(s',\cdot)$ are solutions of $L^{\alpha}w=0$ for every $s,s'\in e$,  we obtain $h_1(\cdot, s), h_1(s',\cdot) \in \mathcal{G}$, so that
$$
h_1(s',s) = \sum_{i=1}^{2\alpha} \sum_{j=1}^{2\alpha}  \sol_i(s)\mv{F}_{i,j} \sol_j(s') = \mv{\sol}^\top(s)\mv{F}  \mv{\sol}(s'),
$$
where $\mv{F} = \{\mv{F}_{i,j}\}_{1\leq i,j\leq 2\alpha}$ is a symmetric matrix.
If $b_1,b_2\in \{0,\ell_e\}$, then it follows from \cite[Corollary 5]{BSW2022}, the fact that $h_2(\cdot, b_1)$ and $h_2(b_2,\cdot)$ vanish at $b_1$ and $b_2$, and are orthogonal to $\mathcal{H}_{\alpha,I}(\partial e)$ and to $\mathcal{H}_+(\Gamma\setminus e)$, that 
\begin{align*}
\partial^{k}_1 \partial^{l}_2  h_1(b_1,b_2) &=\partial^{k}_1 \partial^{l}_2  \left( h_1(\cdot, b_1), h_1(b_2, \cdot)   \right)_{\alpha} = \partial^{k}_1 \partial^{l}_2  \left( h_1(\cdot, b_1), \rho(b_2, \cdot)   \right)_{\alpha}\\
&= \partial^{k}_1 \partial^{l}_2  \left( \rho(\cdot, b_1), \rho(b_2, \cdot)   \right)_{\alpha} = \partial^{k}_1 \partial^{l}_2\rho(b_1,b_2) = \mathbb{E} \left[ \partial^{k}u(b_1) \partial^{l} u(b_2) \right],
\end{align*}
for ${k,l\in\{0,\ldots,\alpha-1\}}$, where $\partial^k_1$ indicates the $k$th order derivative with respect to the first component, and $\partial^l_2$ indicates the $l$th order derivative with respect to the second component. Furthermore, for ${k,l\in\{0,\ldots,\alpha-1\}}$,
$
\partial^{k}_1 \partial^{l}_2  h_1(b_1,b_2) = \mv{F}_{k+\mathbb{I}\left(b_1=\ell_e\right)(k+1), l+\mathbb{I}\left(b_2=\ell_e\right)(l+1)}.
$
Hence, $\mv{F} = \Cov\left( B^{\alpha}u\right)$,  and for $s, s'\in e$,
\begin{equation}\label{eq:exprH1}
h_1(s,s') = \Cov\left(\mv{\sol}^\top(s)  B^{\alpha}u,\mv{\sol}^\top(s')  B^{\alpha}u\right) = \mv{\sol}^\top(s')\mathbb{E}\left[B^\alpha u (B^\alpha u)^\top\right] \mv{\sol}(s).
\end{equation}

We now link the Cameron--Martin spaces to the Gaussian measures and show the stated properties of these processes.
Let $\Phi:H(\Gamma)\to\mathcal{H}(\Gamma)$ be
 isometric isomorphism map and define
\begin{equation}\label{eq:expressionprocessv_zero}
	v_{\alpha, 0}(s) = \Phi^{-1}(h_{2}(\cdot,s)),\quad s\in e.
\end{equation}
The 
Cameron--Martin space of $v_{\alpha,0}$ is $\mathcal{H}_{0,I}(e)$ (which is isometrically isomorphic to $\dot{H}_{0,L}^\alpha(e)$); therefore, we can follow the same proof of \cite[Lemma~4]{BSW2022}
to obtain that the weak derivatives of $v_{\alpha,0}$ in $L_2(\Omega)$ up to order $\alpha-1$, exist and
are weakly continuous in the $L_2(\Omega)$ sense. Also, by the same arguments as in 
\cite[Proposition 11]{BSW2022},  $B^{\alpha}v_{\alpha,0} = \mv{0}.$ Finally, because $h_2(\cdot,t)$ is orthogonal to $\mathcal{H}_{\alpha,I}(\partial e)$ and $v_{\alpha,0}(\cdot)$ is Gaussian (it belongs to the Gaussian linear space $H(\Gamma)$), $v_{\alpha,0}(t)$ is independent of $H_{\alpha,I}(\partial e)$, in particular, it is independent of $\mathcal{F}^u_{\alpha,I}(\partial e)$.
For $h_1$ we have, by Lemma~\ref{lem:FRKHS} 
in Appendix~\ref{app:sec6sec7}
and \eqref{eq:exprH1}, that $u_B(t) = \Phi^{-1}\left( h_1(\cdot, t)\right)$ is given by
$u_B(\cdot) = \mv{\sol}^\top(\cdot) B^{\alpha}u.$
This concludes the proof, since for every $s\in e$, 
$
u(s) = \Phi^{-1}(\rho(\cdot, s)) = \Phi^{-1}(h_1(\cdot, s)) + \Phi^{-1}(h_2(\cdot,s)) = v_{\alpha,0}(s) + \mv{\sol}^\top(s) B^{\alpha}u.
$
\end{proof}

\section{Conditional representation of generalized Whittle--Mat\'ern fields}\label{sec:condrepr}

In this section, we focus on the generalized Whittle--Mat\'ern fields with $\alpha=1$, which by the results above are Markov or order 1, and are interesting from an applied point of view as they essentially are the analog of a Gaussian process on $\mathbb{R}$ with a (possibly non-stationary) exponential covariance. 

We will use the edge representation (Theorem \ref{thm:ReprTheoremEdge_GenWM}) to construct a solution to the problem
$$(\kappa^2 -\nabla(a\nabla))^{1/2} (\tau u) = \mathcal{W},\quad\hbox{on $\Gamma$},$$
that is, a solution to \eqref{eq:Matern_spde_general} with $\alpha=1$, by conditioning independent processes defined on the edges to a continuity constraint.
To such an end, let $\{\widetilde{u}_e:e\in\mathcal{E}\}$ be a family of independent generalized Whittle--Mat\'ern processes on $e=[0,\ell_e]$.  In other words, for each $e\in\mathcal{E}$, $\widetilde{u}_e$ solves
\begin{equation}\label{eq:Neumann_WM_edge}
	(\kappa^2 - \nabla (a \nabla))^{1/2}\widetilde{u}_e = \mathcal{W}_e\quad\hbox{on $e$},
\end{equation}
where the differential operator $\nabla (a \nabla)$ is augmented with Neumann boundary conditions ($\partial_e u(t) = 0$ for $t=0$ and $t=\ell_e$), $\mathcal{W}_e$ is  Gaussian white noise on $e$, and the processes ${\widetilde{u}_e, e\in\mathcal{E}}$, are mutually independent. Auxiliary results are given in Appendix \ref{app:sec6sec7}.

\begin{Remark}
	The reason for why we choose $\alpha=1$ is that the current proof only works for generalized Whittle--Mat\'ern fields if $\alpha=1$. This proof can be generalized to $\alpha\in\mathbb{N}$, but the coefficients $\kappa$ and $a$ must be constant, that is, we must have a Whittle--Mat\'ern field. We have proved such generalization to $\alpha\in\mathbb{N}$, but with constant coefficients, in \cite{BSW2023_AOS}.
\end{Remark}

Let us now define the process on $\Gamma$, which we will denote by $\widetilde{u}$, whose restriction to each edge $e$ is a solution to \eqref{eq:Neumann_WM_edge}. To this end, define $\widetilde{u}(s) := \widetilde{u}_e(t)$, where $s=(t,e)\in\Gamma$. In the following proposition we will use Lemma \ref{lem:auxLemCMindep} in Appendix \ref{app:sec6sec7}.

\begin{Proposition}\label{prp:CM_indep_field}
The Cameron--Martin space associated to $\widetilde{u}$ is $\widetilde{H}^1(\Gamma)$ endowed with the inner product 
$$\<f,g\>_\Gamma = \sum_{e\in\mathcal{E}} \int_e \kappa^2|_e f|_e(t) g|_e(t) dt + \int_e a|_e(t) f'|_e(t) g'|_e(t) dt,\quad f,g\in\widetilde{H}^1(\Gamma).$$
\end{Proposition}

\begin{proof}
Begin by observing that the Cameron--Martin space associated to $\widetilde{u}_e$ is $H^1(e)$ endowed with the inner product
$$\<f_e,g_e\>_e = \int_e \kappa^2|_e f|_e(t) g|_e(t) dt + \int_e a|_e(t) f'|_e(t) g'|_e(t) dt,\quad f_e,g_e\in H^1(e).$$
Indeed, one can apply Remark \ref{rem:CMgeneralizedWM} and Theorem \ref{thm:characterization} to each edge $e\in\mathcal{E}$. The result now follows from a direct application of Lemma~\ref{lem:auxLemCMindep} 
in Appendix~\ref{app:sec6sec7}.
\end{proof}

\begin{Remark}\label{rem:CM_closed_subspace_indep}
If $f,g\in\dot{H}^1_L(\Gamma) = H^1(\Gamma)$, then 
${\<f,g\>_\Gamma = (f,g)_{\dot{H}^1(\Gamma)}}$ by 
integration by parts.
Therefore, $(H^1(\Gamma), (\cdot,\cdot)_{\dot{H}_L^1(\Gamma)})$ is a closed subspace of $(\widetilde{H}^1(\Gamma), \<\cdot,\cdot\>_\Gamma)$.
\end{Remark}

Now, let $u$ be a generalized Whittle--Mat\'ern field, that is, let $u$ be a solution to \eqref{eq:Matern_spde_general} with $\alpha=1$. Then, it follows by Remark \ref{rem:CMgeneralizedWM} that the Cameron--Martin space associated to $u$ is $(H^1(\Gamma), (\cdot,\cdot)_{\dot{H}_L^1(\Gamma)})$, which is a closed subspace of $(\widetilde{H}^1(\Gamma), \<\cdot,\cdot\>_\Gamma)$ by Remark~\ref{rem:CM_closed_subspace_indep}. Let $\Pi_K:\widetilde{H}^1(\Gamma)\to H^1(\Gamma)$ be the orthogonal projection onto $H^1(\Gamma)$ with respect to the $\<\cdot, \cdot\>_\Gamma$ inner product.
Further, let ${H_{\widetilde{u}}(\Gamma) = \overline{\textrm{span}\{\widetilde{u}(t,e): e\in\mathcal{E}, t\in e\}}}$ be the Gaussian space associated to $\widetilde{u}$ and recall the isometric isomorphism ${\Phi:H_{\widetilde{u}}(\Gamma) \to \widetilde{H}^1(\Gamma)}$ given by ${\Phi(h)(t,e) = \mathbb{E}(h\widetilde{u}_e(t))}$, for $s = (t,e)\in\Gamma$.
We will now show that a generalized Whittle--Mat\'ern field with $\alpha=1$ can be obtained as a projection of $\widetilde{u}$ on a suitable space.

First, observe that we can obtain the reproducing kernel of $(H^1(\Gamma),(\cdot,\cdot)_{\dot{H}^1_L(\Gamma)})$ by projection. Indeed, let $\widetilde{\varrho}(\cdot,\cdot)$ be the covariance function of $\widetilde{u}$. Thus, $\widetilde{\varrho}(\cdot,\cdot)$ is also the reproducing kernel of $(\widetilde{H}^1(\Gamma), \<\cdot,\cdot\>_\Gamma)$. Second, define, for each $s = (t,e)\in\Gamma$, 
$$\varrho(s,\cdot) = \varrho((t,e),\cdot) = \Pi_K(\widetilde{\varrho}(s,\cdot)) = \Pi_K(\widetilde{\varrho}((t,e),\cdot)).$$
This defines a function $\varrho:\Gamma\times \Gamma\to \mathbb{R}$. We now claim that $\varrho(\cdot,\cdot)$ is a reproducing kernel for $(H^1(\Gamma),(\cdot,\cdot)_{\dot{H}^1_L(\Gamma)})$. Indeed, given $h\in H^1(\Gamma)$, we have by Remark \ref{rem:CM_closed_subspace_indep} that for $s=(t,e)\in\Gamma$,
\begin{align*}
(\varrho(s,\cdot), h)_{\dot{H}^1_L(\Gamma)} &= \<\varrho(s,\cdot), h\>_\Gamma = \<\Pi_K(\widetilde{\varrho}(s,\cdot)), h\>_\Gamma\\
&= \<\widetilde{\varrho}(s,\cdot) - \Pi_K^\perp(\widetilde{\varrho}(s,\cdot)), h\>_\Gamma = \<\widetilde{\varrho}(s,\cdot), h\>_\Gamma = h(s),
\end{align*}
where $\Pi_K^\perp$ is the orthogonal projection onto the orthogonal complement of $H^1(\Gamma)$ on $\widetilde{H}^1(\Gamma)$ and we used that $h\in H^1(\Gamma)\subset \widetilde{H}^1(\Gamma)$ and that $\widetilde{\varrho}(\cdot,\cdot)$ is a reproducing kernel for $(\widetilde{H}^1(\Gamma), \<\cdot,\cdot\>_\Gamma)$.

This allows us to define a new field:
$u(s) := \Phi^{-1}(\varrho(s,\cdot)), s\in\Gamma.$
Observe that, by construction, $u$ is a centered GRF. Its Gaussian space is given by 
$
H_u = \Phi^{-1}\circ \Pi_K(\widetilde{H}^1(\Gamma)) = \Phi^{-1}(H^1(\Gamma)),
$ 
and its Cameron--Martin space is $(H^1(\Gamma),(\cdot,\cdot)_{\dot{H}^1_L(\Gamma)})$. 
By uniqueness of the Cameron--Martin space, $u$ is a generalized Whittle--Mat\'ern field on $\Gamma$. Furthermore, let $\Pi_u:H_{\widetilde{u}} \to H_u$ be the orthogonal projection onto $H_u$. Since $\Phi$ is an isometric isomorphism and $\widetilde{u}$ is Gaussian, we have that for every $s\in\Gamma$,
\begin{align*}
u(s) = \Phi^{-1}(\Pi_K(\widetilde{\varrho}(s,\cdot))) = \Pi_u\Big(\Phi^{-1}(\widetilde{\varrho}(s,\cdot)) \Bigr) = \Pi_u(\widetilde{u}(s)) = \mathbb{E}(\widetilde{u}(s)| \sigma(H_u)).
\end{align*}
Therefore, a generalized Whittle--Mat\'ern fields with $\alpha=1$ can be represented as
\begin{equation}\label{eq:repr_gWM1}
u(s) = \mathbb{E}(\widetilde{u}(s) | \sigma(H_u)),\quad s\in\Gamma.
\end{equation}

Our goal now is to study the right-hand side of \eqref{eq:repr_gWM1} to show that the conditional expectation actually only acts on the values of $\widetilde{u}$ at the vertices of $\Gamma$.
To this end, first observe that, for each $e\in\mathcal{E}$, we can apply Theorem \ref{thm:ReprTheoremEdge_GenWM} to $\widetilde{u}_e$ to obtain that
$\widetilde{u}_e(t) = v_{e,0}(t) + \mv{\sol}_e^\top(t) B^1 \widetilde{u}_e, t\in e,$
where $v_{e,0}(\cdot)$ is the GRF defined on the edge $e$ with associated Cameron--Martin space
$(H_{0}^1(e), (\cdot,\cdot)_{1,e})$, where $H^1_0(e)$ is the set of absolutely continuous functions $f$ defined on $e=[0,\ell_e]$ such that $f(0)=f(\ell_e)=0$, and $\mv{\sol}_e(\cdot) = (\sol_{e,1}, \sol_{e,2})$ are linearly independent solutions to $L\sol = 0$ on $e$, and $B^1 \sol_j = e_j$,  $e_j$ being the $j$th canonical vector, $j=1,2$. Further, $v_{e,0}$ is independent of $B^1 \widetilde{u}_e$. This yields a representation on $\Gamma$ that we will describe now. First, we extend $v_{e,0}(\cdot)$ and $\mv{\sol}_e(\cdot)$ as zero outside of $e$. Thus, for $s=(t,e)$, $v_{e',0}(s) = 0$ if $e\neq e'$ and $v_{e',0}(s) = v_{e,0}(t)$ if $e=e'$. Similarly for $\mv{\sol}_e(\cdot)$. Then, it directly follows that 
\begin{equation}\label{eq:bridge_repr_indep}
	\widetilde{u}(s) = \sum_{e\in\mathcal{E}} v_{e,0}(s) + \mv{\sol}_{e}^\top(s) B^1 \widetilde{u}_{e}, \quad s\in\Gamma.
\end{equation}
\begin{Example}
	\label{ex:ex1}
Let $u$ be a Whittle--Mat\'ern field, defined by \eqref{eq:Matern_spde} with $\alpha=1$, on a star graph with three edges.  
Then, $\sol_e(s) = [r(s),r(s-\ell_e)]$ with $r(h)=(2\kappa\tau^2)^{-1}\exp\left(-\kappa|h|\right)$, ${B^1 \widetilde{u}_{e} = \begin{bmatrix}
	\widetilde{u}_{e}(0) &
	\widetilde{u}_{e}(\ell_e)
\end{bmatrix}^\top}$ and $v_{e,0}$ is a GRF with covariance function $r$, conditioned on being zero at the boundary points of $e$.
\end{Example}
Note that for every $t\in e$, $\Phi(v_{e,0}(t))|_e \in H^1_0(e)$. Indeed,
from the independence between $v_{e,0}$ and $B^1 \widetilde{u}_e$, we have that
\begin{align*}
	\Phi(v_{e,0}(t))((t',e)) &= \mathbb{E}(v_{e,0}(t)\widetilde{u}_e(t')) = \mathbb{E}\left(v_{e,0}(t)(v_{e,0}(t') + \mv{\sol}_e^\top(t') B^1 \widetilde{u}_e)\right)\\
	&= \mathbb{E}(v_{e,0}(t)v_{e,0}(t')) = \varrho_{0,e}(t,t'),
\end{align*}
where $\varrho_{0,e}(\cdot,\cdot)$ is the covariance function of $v_{e,0}$, and for every $t\in e$, $\varrho_{0,e}(\cdot,t)\in H_0^1(e)$. By using, again, that $\widetilde{u}_e$ is independent of $\widetilde{u}_{e'}$ is $e\neq e'$, it follows that for every $t\in e$, $\Phi(v_{e,0}(t)) = 0$ in $\Gamma\setminus e$. This implies that $\Phi(v_{e,0}(t))$, $t\in e$, is continuous. Therefore, $\Phi(v_{e,0}(t)) \in H^1(\Gamma), t\in e$. Thus, if we define 
$H_{0,\widetilde{u}} = \overline{\textrm{span}\{v_{e,0}(s): s\in\Gamma\}},$
then, 
\begin{equation}\label{eq:bridge_contained_kirk}
	H_{0,\widetilde{u}}\subset H_u.
\end{equation}
since $\Phi(H_{0,\widetilde{u}}) \subset H^1(\Gamma)$. This tells us that given the representations \eqref{eq:repr_gWM1} and \eqref{eq:bridge_repr_indep}, we have that
\begin{align*}
u(s) &= \mathbb{E}(\widetilde{u}(s) | \sigma(H_u)) = \mathbb{E}\Bigl(\sum_{e\in\mathcal{E}} v_{e,0}(s) + \mv{\sol}_{e}^\top(t) B^1 \widetilde{u}_{e}  \Big| \sigma(H_u)\Bigr) \\
&= \sum_{e\in\mathcal{E}} v_{e,0}(s) + \mv{\sol}_{e}^\top(s) \mathbb{E}(B^1 \widetilde{u}_{e}|\sigma(H_u)),
\end{align*}
for $s\in\Gamma$,
where we have used \eqref{eq:bridge_contained_kirk} in the last equality.

Let $H_\mathcal{G}$ be the orthogonal complement of $H_{0,\widetilde{u}}$ in $H_u$ so that
$H_u = H_{0,\widetilde{u}} \oplus H_\mathcal{G}.$
Note that \eqref{eq:bridge_repr_indep}, together with the independence between $v_{e,0}$ and $B^1\widetilde{u}_e$ for $e\in\mathcal{E}$, gives us that $H_\mathcal{G}\subset \textrm{span}\{B^1\widetilde{u}_e:e\in\mathcal{E}\}$. In particular, we have that 
$H_\mathcal{G} = \textrm{span}\{F_j \widehat{B}^1 \widetilde{u}: j\in J\},$
where $\widehat{B}^1 \widetilde{u} = \{\widetilde{u}_e(0), \widetilde{u}_e(\ell_e): e\in\mathcal{E}\}$, $F_j$ is a one dimensional matrix (i.e., a row vector) and $J$ is a finite set of indexes. That is, $H_\mathcal{G}$ is generated by a set of linear combinations of vertex elements $\widehat{B}^1\widetilde{u}$.
We can now observe that, for each $e\in\mathcal{E}$, $B^1\widetilde{u}_e$ is independent of $\sigma(H_{0,\widetilde{u}})$. Therefore, for $s\in\Gamma$,
\begin{align*}
u(s) &= \sum_{e\in\mathcal{E}} v_{e,0}(s) + \mv{\sol}_{e}^\top(s) \mathbb{E}(B^1 \widetilde{u}_{e}|\sigma(H_\mathcal{G})) 
= \sum_{e\in\mathcal{E}} v_{e,0}(s) + \mv{\sol}_{e}^\top(s) \mathbb{E}(B^1 \widetilde{u}_{e}|\sigma(F_j \widehat{B}^1\widetilde{u}:j\in J)).
\end{align*}

By \cite[Theorem 9.1 and Remark 9.2]{janson_gaussian}, 
$\mathbb{E}(B^1\widetilde{u}_e|\sigma(F_j \widehat{B}^1\widetilde{u}:j\in J)) = \widehat{T}_e \widehat{F} \widehat{B}^1 \widetilde{u} := T_e \widehat{B}^1\widetilde{u},$ for every $e\in\mathcal{E}$,
where $\widehat{F} = [F_1,\ldots,F_{|J|}]$, $\widehat{T}_e, e\in\mathcal{E}$, is some matrix and $T_e = \widehat{T}_e \widehat{F}$.
Thus, we have proved:

\begin{Proposition}\label{prp:CondRepr_gWM_part1}
	The generalized Whittle--Mat\'ern field given by the solution of \eqref{eq:Matern_spde_general} with $\alpha=1$ can be obtained as
	$$u(s) = \sum_{e\in\mathcal{E}} v_{e,0}(s) + \mv{\sol}_{e}^\top(s)T_e \widehat{B}^1\widetilde{u},$$
	where, for  $e\in\mathcal{E}$, $v_{e,0}(\cdot)$ is the GRF defined on the edge $e$ with associated Cameron--Martin space
	$(H_{0}^1(e), (\cdot,\cdot)_{1,e})$; $\mv{\sol}_e(\cdot) = (\sol_{e,1}(\cdot),\sol_{e,2}(\cdot))$, where $\{\sol_{e,1}, \sol_{e,2}\}$ is a set of linearly independent solutions of 
	$L \sol = 0$ on $e$, satisfying $B^{1} \sol_{e,1}=[1,0]$ and $B^1 \sol_{e,2} = [0,1]$; 
	$\widetilde{u}$ is the field obtained by joining independent solutions of \eqref{eq:Neumann_WM_edge} on each edge; and $T_e$ is a matrix.
\end{Proposition}

Our goal now is to better understand the matrix $T_e$. To this end, recall the representation obtained in \eqref{eq:repr_gWM1} and observe that if we write for $s\in\Gamma$
\begin{align*}
\widetilde{u}(s) &= \mathbb{E}(\widetilde{u}(s)| \sigma(H_u)) + \left(\widetilde{u}(s) -\mathbb{E}(\widetilde{u}(s)| \sigma(H_u))\right)
= u(s) + (\widetilde{u}(s) - u(s)).
\end{align*}
Therefore, $u(s)$ and $\widetilde{u}(s) - u(s)$ are independent for every $s\in\Gamma$. Let $P_e$ be the matrix such that ${P_e \widehat{B}^1\widetilde{u} = B^1\widetilde{u}_e}$. Thus, we can rewrite \eqref{eq:bridge_repr_indep} as 
$$\widetilde{u}(s) = \sum_{e\in\mathcal{E}} v_{e,0}(s) + \mv{\sol}_{e}^\top(s) P_e \widehat{B}^1 \widetilde{u}, \quad s\in\Gamma.$$
By using the above expression and Proposition \ref{prp:CondRepr_gWM_part1} we obtain that for every $s\in\Gamma$
\begin{align*}
	\widetilde{u}(s) - u(s) &= \sum_{e\in\mathcal{E}}  \mv{\sol}_{e}^\top(s)P_e \widehat{B}^1\widetilde{u} -  \mv{\sol}_{e}^\top(s)T_e \widehat{B}^1\widetilde{u}=\sum_{e\in\mathcal{E}}   \mv{\sol}_{e}^\top(s)(P_e - T_e) \widehat{B}^1\widetilde{u} := \sum_{e\in\mathcal{E}}   \mv{\sol}_{e}^\top(s) K_e \widehat{B}^1 \widetilde{u},
\end{align*}
where $K_e = P_e - T_e$. Thus,
\begin{equation}\label{eq:decomp_indepGWM_kirk}
\widetilde{u}(s) = u(s) + \sum_{e\in\mathcal{E}}   \mv{\sol}_{e}^\top(s) K_e \widehat{B}^1 \widetilde{u},\quad s\in\Gamma,
\end{equation}
where $u(s)$ and $\{K_e \widehat{B}^1 \widetilde{u}:e\in\mathcal{E}\}$ are independent for every $s\in\Gamma$. We have, from \eqref{eq:decomp_indepGWM_kirk}, that for $s\in\Gamma$,
\begin{align*}
\mathbb{E}(\widetilde{u}(s) | \sigma(K_e \widehat{B}^1 \widetilde{u}:e\in\mathcal{E})) &= \mathbb{E}\left(u(s) + K_e \widehat{B}^1 \widetilde{u} \Big| \sigma(K_e \widehat{B}^1 \widetilde{u}:e\in\mathcal{E})\right) \\
&= \mathbb{E}\left(u(s) \Big| \sigma(K_e \widehat{B}^1 \widetilde{u}:e\in\mathcal{E})\right) + \sum_{e\in\mathcal{E}}   \mv{\sol}_{e}^\top(s) K_e \widehat{B}^1 \widetilde{u}\\
&= \mathbb{E}(u(s)) + \sum_{e\in\mathcal{E}}   \mv{\sol}_{e}^\top(s) K_e \widehat{B}^1 \widetilde{u}= \sum_{e\in\mathcal{E}}   \mv{\sol}_{e}^\top(s) K_e \widehat{B}^1 \widetilde{u},
\end{align*}
since $u(s)$ and $\{K_e \widehat{B}^1 \widetilde{u}:e\in\mathcal{E}\}$ are independent and $u(s)$ is a centered GRF. Let $H_{K\widehat{B}^1\widetilde{u}}$ be the Gaussian space generated by $\{K_e \widehat{B}^1 \widetilde{u}:e\in\mathcal{E}\}$,  $\Pi_B$ be the projection operator from $H_{\widetilde{u}}$ onto $H_{K\widehat{B}^1\widetilde{u}}$, and $\Pi_B^\perp$ be the projection onto the orthogonal complement of $H_{K\widehat{B}^1\widetilde{u}}$ on $H_{\widetilde{u}}$. Hence, the above expression tells us that $\Pi_B(\widetilde{u}(s)) = \sum_{e\in\mathcal{E}}   \mv{\sol}_{e}^\top(s) K_e \widehat{B}^1 \widetilde{u}$ and $\Pi_B^\perp (\widetilde{u}(s)) = u(s)$. 

By \cite[Remark 9.10]{janson_gaussian} we have that $\Pi_B^\perp (\widetilde{u}(s)) = u(s)$ is the field obtained by conditioning $\widetilde{u}(s)$ on $\{K_e\widehat{B}^1\widetilde{u} = 0, e\in\mathcal{E}\}$. Furthermore, let $K$ be the vertical concatenation of $K_e, e\in\mathcal{E}$. We can, then, rewrite the system $\{K_e\widehat{B}^1\widetilde{u} = 0, e\in\mathcal{E}\}$, as $K\widehat{B}^1\widetilde{u} = 0$. Therefore, we have proved the following result:

\begin{Proposition}\label{prp:GWm_prp_kirch_cond}
	Let $\widetilde{u}$ be the field obtained by joining independent solutions of \eqref{eq:Neumann_WM_edge} on each edge, let $K_e, e\in\mathcal{E},$ be the matrices given in \eqref{eq:decomp_indepGWM_kirk} and let $K$ be their vertical concatenation. The generalized Whittle--Mat\'ern field given by the solution of \eqref{eq:Matern_spde_general} with $\alpha=1$ can be obtained by conditioning the field $\widetilde{u}$ on $K\widehat{B}^1\widetilde{u} = 0$.
\end{Proposition}

Observe that we have proved that if we condition $\widetilde{u}(s)$ on $K\widehat{B}^1\widetilde{u} = 0$ for any $s\in\Gamma$, then we obtain a generalized Whittle--Mat\'ern field with Kirchhoff vertex conditions. Further, $\widehat{B}^1\widetilde{u}$ consists of values of the field $\widetilde{u}$ at the vertices of the metric graph. Therefore, $K\widehat{B}^1\widetilde{u} = 0$ enforces continuity of the conditioned field, $u$, at the vertices. More precisely, \eqref{eq:decomp_indepGWM_kirk} gives that
\begin{equation}\label{eq:Matrix_cond_continuity}
	\widetilde{u}_e(v) = \widetilde{u}_{e'}(v),\quad v\in e\cap e',\quad e,e'\in\mathcal{E}\quad \Leftrightarrow \quad K\widehat{B}^1\widetilde{u} = 0.
\end{equation}
It is easy to see that if $\widetilde{K}$ is another matrix that enforces continuity of the conditioned field at the vertices, that is, that the left-hand side of \eqref{eq:Matrix_cond_continuity} holds if, and only if $\widetilde{K}\widehat{B}^1 \widetilde{u} = 0$, then $\sigma(K\widehat{B}^1\widetilde{u}) = \sigma(\widetilde{K}\widehat{B}^1\widetilde{u})$. We, therefore, have proved the following theorem:

\begin{Theorem}\label{thm:GWM_kirch_cond_thm}
	Let $\widetilde{u}$ be the field obtained by joining independent solutions of \eqref{eq:Neumann_WM_edge} on each edge. Let, also, $K$ be any matrix such that \eqref{eq:Matrix_cond_continuity} holds, that is, let $K$ be any matrix that enforces continuity of the conditioned field at the vertices of the metric graph. Then, the generalized Whittle--Mat\'ern field given by the solution of \eqref{eq:Matern_spde_general} with $\alpha=1$ can be obtained by conditioning the field $\widetilde{u}$ on $K\widehat{B}^1\widetilde{u} = 0$.
\end{Theorem}

\begin{Example}[Example 1 continued]
	\label{ex:ex2}
	Let the edges of the star graph be $e_i$, $i=1,2,3$, where $e_i$ starts in vertex $v_i$ and ends in vertex $v_4$. Then, 
	$
	\widehat{B}^1\widetilde{u}= \left[\widetilde{u}_1(0),\widetilde{u}_1(\ell_1),\widetilde{u}_2(0),\widetilde{u}_2(\ell_2),\widetilde{u}_3(0),\widetilde{u}_1(\ell_3)   \right]^\top, 
	$
	where $\ell_i$ is the length of edge $e_i$, and $K $ can for example be taken as 
	$$\begin{bmatrix}
		0 & 1 & 0 & -1 & 0 & 0  \\
		0 & 0 & 0 & 1 & 0 & -1 
	\end{bmatrix} \quad \text{or} \quad 
\begin{bmatrix}
	0 & 1 & 0 & 0 & 0 & -1  \\
	0 & 0 & 0 & 1 & 0 & -1 
\end{bmatrix}.
$$
\end{Example}

\section{Discussion}\label{sec:discussion}
We showed that no GRFs on general metric graphs exist which are both Markov of order 1 and isotropic in the geodesic or resistance metric. This was done by considering a specific class of graphs as a counter example. We conjecture that the class of graphs and metrics where the result holds is much larger. In particular, we conjecture that it holds for any graph which is not a tree or a Euclidean cycle endowed with a metric such as the resistance or geodesic metrics. Another question related to this is whether one could extend the result to higher-order Markov properties. We have not pursued this as it is not even clear whether it is possible to define differentiable and isotropic GRFs on general metric graphs.

Having characterized the Markov properties of the generalized Whittle--Mat\'ern fields, a natural question is how restrictive this class is compared to all Gaussian Markov processes on compact
metric graphs. Because we know that the Cameron--Martin space should be local for the process to be Markov, which translates to the precision operator being local, we can make a reasonable conjecture. There are two properties we can vary in the model and still have a local operator. The first is the type of vertex conditions. The existence of the solution to \eqref{eq:Matern_spde_general} was demonstrated under a much larger class of vertex conditions than the Kirchhoff vertex conditions in \cite{bolinetal_fem_graph}. The second property is how the operator acts on functions on a specific edge. 
Therefore, we make the following conjecture.

\begin{Conjecture}
	Any GMRF $u$ of order $p$ on a compact metric graph $\Gamma$, 
	satisfying some mild regularity conditions, can be obtained as a GRF with a precision operator that acts locally on each edge $e\in\mathcal{E}$ as  
	$
	{Q = \sum_{i=0}^{2p} \kappa_{i,e} \frac{d^{i}}{dx_e^{i}}},
	$
	where $\kappa_{i,e}$ are some (sufficiently nice) functions, and whose action on the vertices is given by a finite-rank operator $M$, such that 
	$Q$
	is self-adjoint, local, and has a compact resolvent.
\end{Conjecture}

The Whittle--Mat\'ern fields are of this form, with the restriction that the functions $\kappa_{i,e}$ are determined by $\kappa$ and $\H$.
Thus, performing a complete characterization of all GMRFs on compact metric graphs is an interesting topic for future research. 
Another topic for future research is how to take advantage of the Markov property for the generalized Whittle--Mat\'ern fields. 
One of the key advantages of having Markov properties is that they tend to make inference easier and more computationally efficient. 
In the case of the Whittle--Mat\'ern fields, it is unclear how to evaluate finite-dimensional distributions to perform likelihood-based inference; however, the edge representation in Section~\ref{sec:edge} along with the conditional representation in Section \ref{sec:condrepr} are crucial steps towards being able to do this. This problem will be investigated further in \cite{BSW2023_AOS}. 
Finally, a direction that we believe deserves attention is to consider the Markov property of solutions to \eqref{eq:Matern_spde_general} in the degenerate case when $\kappa=0$. This includes the Gaussian free field on metric graphs as a particular case when $a=\tau=1$.

\begin{appendix}
	\section{Auxiliary results and definitions for Section \ref{sec:Markov}}\label{app:aux_secMarkov}

	The following two propositions, providing characterizations and properties of conditional independence, are used throughout the paper. We refer the reader to \cite{Mandrekar1983} for their proofs. 

\begin{Proposition}\label{prp:equivcondindep}
Let $\mathcal{A}$ and $\mathcal{C}\subset \mathcal{B}$ be sub-$\sigma$-algebras of $\mathcal{F}$. Then $\mathcal{A}$ and $\mathcal{B}$ are conditionally independent given $\mathcal{C}$ if, and only if, for every bounded $\mathcal{A}$-measurable function $f$, 
$\mathbb{E}(f|\mathcal{B}) = \mathbb{E}(f|\mathcal{C}).$
\end{Proposition}

\begin{Proposition}\label{prp:propertiescondindep}
Let $\mathcal{A}$, $\mathcal{B}$ and
$\mathcal{C}$ be sub-$\sigma$-algebras of $\mathcal{F}$ such that $\mathcal{A}$ and $\mathcal{B}$ are conditionally independent given $\mathcal{C}$. Define $\mathcal{B}\lor \mathcal{C} := \sigma(\mathcal{B}\cup \mathcal{C})$, then
\begin{enumerate}[label= \roman{enumi}.]
	\item For every bounded $\mathcal{A}$-measurable function $f$, $\mathbb{E}(f|\mathcal{B}\lor \mathcal{C}) = \mathbb{E}(f|\mathcal{C})$; \label{prp:propertiescondindep1}
	\item For every $\sigma$-algebra $\widetilde{\mathcal{C}}$ such that $\mathcal{C}\subset\widetilde{\mathcal{C}}\subset \mathcal{C}\lor \mathcal{B}$, $\mathcal{A}$ and $\mathcal{B}$ are conditionally independent given~$\widetilde{\mathcal{C}}$; \label{prp:propertiescondindep2}
	\item For every $\sigma$-algebra $\widetilde{\mathcal{B}}\subset \mathcal{C}\lor \mathcal{B}$,  $\mathcal{A}$ and $\widetilde{\mathcal{B}}$ are conditionally independent given $\mathcal{C}$. \label{prp:propertiescondindep3}
\end{enumerate}
\end{Proposition}

\begin{proof}[Proof of Proposition~\ref{prp:markovinnerboundary}]
By Definition~\ref{def:MarkovPropertyField} and Theorem~\ref{thm:mandrekar_markov_split} (see also \cite[Theorem~2.2]{Mandrekar1976} or \cite[page 69]{Rozanov1982Markov}),  $u$ is an MRF if and only if $\mathcal{F}^u_{+,I}(\partial S)$ splits $\mathcal{F}^u(S)$ and $\mathcal{F}^u(\Gamma\setminus S)$ for every open set $S\subset\Gamma$. Because $\Gamma\setminus S$ is closed, we have that  $\mathcal{F}_+^u(\Gamma\setminus S) = \mathcal{F}^u(\Gamma\setminus S)\lor \mathcal{F}^u_{+,I}(\partial S)$. Thus, by Proposition \ref{prp:propertiescondindep}.\ref{prp:propertiescondindep3},  $\mathcal{F}^u_{+,I}(\partial S)$ splits $\mathcal{F}^u(S)$ and $\mathcal{F}^u(\Gamma\setminus S)$ if and only if $\mathcal{F}^u_{+,I}(\partial S)$ splits $\mathcal{F}^u(S)$ and $\mathcal{F}_+^u(\Gamma\setminus S)$, proving the equivalence between $u$ being Markov and \ref{prp:markovinnerboundary1}. Finally, the equivalence of \ref{prp:markovinnerboundary1} and \ref{prp:markovinnerboundary2} follows from  Proposition~\ref{prp:equivcondindep}.
\end{proof}

	The following theorem is a straightforward adaptation of \cite[Theorem 2.2]{Mandrekar1976} that we add for completeness as the definition of MRFs in \cite{Mandrekar1976} is a bit different from the definition we give in the present paper.

	\begin{Theorem}\label{thm:mandrekar_markov_split}
		Let $\Gamma$ be a compact metric space and let $u$ be a random field on $\Gamma$. A field $u$ on $\Gamma$ is an MRF if, and only if, for every open set $S$, whose boundary consists of finite many points, we have $\mathcal{F}^u_{+,I}(\partial S)$ splits $\mathcal{F}^u(S)$ and $\mathcal{F}^u(\Gamma\setminus S)$. 
	\end{Theorem}
	\begin{proof}
		We begin by showing that $u$ is an MRF if, and only if, for every open set $S$, whose boundary consists of finite many points, we have $\mathcal{F}^u_{+}(\partial S)$ splits $\mathcal{F}^u(\overline{S})$ and $\mathcal{F}^u(\Gamma\setminus S)$. Indeed, let $S$ be any open set whose boundary consists of finitely many points. 
		If $u$ is an MRF, since $\mathcal{F}^u(\overline{S}) \subset \mathcal{F}_+^u(\overline{S})$ and $\mathcal{F}^u(\Gamma\setminus S) \subset \mathcal{F}_+^u(\Gamma\setminus S)$, we have that $\mathcal{F}_+^u(\partial S)$ splits $\mathcal{F}^u(\overline{S})$ and $\mathcal{F}^u(S^c)$. 
		Conversely, if $\mathcal{F}^u_{+}(\partial S)$ splits $\mathcal{F}^u(\overline{S})$ and $\mathcal{F}^u(S^c)$, note that ${\mathcal{F}_+^u(\overline{S}) = \mathcal{F}^u(\overline{S})\lor \mathcal{F}_+^u(\partial S)}$ and that $\mathcal{F}_+^u(\Gamma\setminus S) = \mathcal{F}^u(\Gamma\setminus S)\lor \mathcal{F}_+^u(\partial S)$. Thus, it follows by Proposition~\ref{prp:propertiescondindep}.iii that $\mathcal{F}^u_{+}(\partial S)$ splits $\mathcal{F}_+^u(\overline{S})$ and $\mathcal{F}^u(\Gamma\setminus S)$. By another application of  Proposition~\ref{prp:propertiescondindep}.iii, we obtain that $\mathcal{F}^u_{+}(\partial S)$ splits $\mathcal{F}_+^u(\overline{S})$ and $\mathcal{F}_+^u(\Gamma\setminus S)$. 

		We now show that $\mathcal{F}^u_{+,I}(\partial S)$ splits $\mathcal{F}^u(S)$ and $\mathcal{F}^u(\Gamma\setminus S)$ if, and only if, $\mathcal{F}^u_{+}(\partial S)$ splits $\mathcal{F}^u(\overline{S})$ and $\mathcal{F}^u(\Gamma\setminus S)$. Assume that $\mathcal{F}^u_{+,I}(\partial S)$ splits $\mathcal{F}^u(S)$ and $\mathcal{F}^u(\Gamma\setminus S)$. Then, it follows from Proposition~\ref{prp:propertiescondindep}.ii that for every $\varepsilon>0$, $\mathcal{F}^u(S\cap (\partial S)_\varepsilon)$ splits $\mathcal{F}^u(S)$ and $\mathcal{F}^u(\Gamma\setminus S)$. Now, observe that for every $\varepsilon > 0$, we have 
		$$\mathcal{F}^u(S\cap (\partial S)_\varepsilon) \subset \mathcal{F}^u((\partial S)_\varepsilon) \subset \mathcal{F}^u(S\cap (\partial S)_\varepsilon) \lor \mathcal{F}^u(\Gamma\setminus S).$$
		Therefore, by Proposition~\ref{prp:propertiescondindep}.ii, for every $\varepsilon > 0$, $\mathcal{F}^u(S\cap (\partial S)_\varepsilon)$ splits $\mathcal{F}^u(S)$ and $\mathcal{F}^u(\Gamma\setminus S)$. By the martingale convergence theorem, we have that $\mathcal{F}_+^u(\partial S)$ splits $\mathcal{F}^u(S)$ and $\mathcal{F}^u(\Gamma\setminus S)$, and by Proposition~\ref{prp:propertiescondindep}.iii, we have that $\mathcal{F}_+^u(\partial S)$ splits $\mathcal{F}^u(\overline{S})$ and $\mathcal{F}^u(\Gamma\setminus S)$. Conversely, if  $\mathcal{F}^u_{+}(\partial S)$ splits $\mathcal{F}^u(\overline{S})$ and $\mathcal{F}^u(\Gamma\setminus S)$, we have by Proposition~\ref{prp:propertiescondindep}.ii that for every $\varepsilon > 0$, $\mathcal{F}^u((\partial S)_\varepsilon)$ splits $\mathcal{F}^u(\overline{S})$ and $\mathcal{F}^u(\Gamma\setminus S)$. Further, since $\mathcal{F}^u(S)\subset \mathcal{F}^u(\overline{S})$, we have that $\mathcal{F}^u((\partial S)_\varepsilon)$ splits $\mathcal{F}^u(S)$ and $\mathcal{F}^u(\Gamma\setminus S)$. Let $(\partial S)^\varepsilon := \{s\in\Gamma: \exists z\in \partial S\hbox{ with } d(s,z)\leq \varepsilon\}$ and observe that by Proposition~\ref{prp:propertiescondindep}.ii again, we have that $\mathcal{F}^u((\partial S)^\varepsilon)$ splits $\mathcal{F}^u(\overline{S})$ and $\mathcal{F}^u(\Gamma\setminus S)$. Now, let $S(\varepsilon) := S\cap (S^c\cup (\partial S)^\varepsilon)^c$, and note that since $S$ is an open set whose boundary consists of finitely many points, $S(\epsilon)$ is also an open set whose boundary consists of finitely many points. Thus, by assumption $\mathcal{F}_+^u(\partial(S(\varepsilon)))$ splits $\mathcal{F}^u(\overline{S(\varepsilon)})$ and $\mathcal{F}^u(\Gamma\setminus S(\varepsilon))$. By Proposition~\ref{prp:propertiescondindep}.ii, we obtain that $\mathcal{F}^u(\partial(S(\varepsilon))_\varepsilon)$ splits  $\mathcal{F}^u(\overline{S(\varepsilon)})$ and $\mathcal{F}^u(\Gamma\setminus S(\varepsilon))$ (note that we using the same $\varepsilon$ as the one that defines the set $S(\varepsilon)$). Now, since $\mathcal{F}^u(\Gamma\setminus S) \subset \mathcal{F}^u(\Gamma\setminus S(\varepsilon))$, it follows that $\mathcal{F}^u(\partial(S(\varepsilon))_\varepsilon)$ splits $\mathcal{F}^u(\overline{S(\varepsilon)})$ and $\mathcal{F}^u(\Gamma\setminus S)$. Now, observe that if $0 < \delta <\varepsilon$, then $\mathcal{F}^u(\overline{S(\varepsilon)}) \subset \mathcal{F}^u(\overline{S(\delta)})$. Therefore, we have that for every $delta>0$ such that $\delta<\varepsilon$, $\mathcal{F}^u(\partial(S(\delta))_\delta)$ splits $\mathcal{F}^u(\overline{S(\varepsilon)})$ and $\mathcal{F}^u(\Gamma\setminus S)$. Now, observe that $\mathcal{F}^u_{I,+}(\partial S) = \bigcap_{\delta<\varepsilon} \mathcal{F}^u((\partial S(\delta))_\delta)$. Hence, by the martingale convergence theorem, it follows that for every $\varepsilon>0$, $\mathcal{F}^u_{I,+}(\partial S)$ splits $\mathcal{F}^u(\overline{S(\varepsilon)})$ and $\mathcal{F}^u(\Gamma\setminus S)$. Finally, by another application of the martingale convergence theorem, we obtain that $\mathcal{F}^u_{I,+}(\partial S)$ splits $\sigma\left(\bigcup_{\varepsilon>0} \mathcal{F}^u(\overline{S(\varepsilon)})\right)$ and  $\mathcal{F}^u(\Gamma\setminus S)$, which concludes the proof.
	\end{proof}

	\section{Auxiliary results and definitions for Section~\ref{sec:isotropic}}\label{app:isotropic_apdx}

	Throughout this section we will provide basic definitions and results related to Gaussian isotropic models in compact metric graphs. We refer the reader to \cite{anderes2020isotropic} for further details. We begin this section with the definition of metric graph with Euclidean edges. 
	
	\begin{Definition}\label{def:graph_euclidean_edges}
		Let $\Gamma$ be a metric graph with vertices $\mathcal{V}$ and edges $\mathcal{E}$. We say that $\Gamma$ is a metric graph with Euclidean edges if the following conditions are satisfied:
		\begin{enumerate}
			\item The graph must be compact, that is, $\mathcal{V}$ and $\mathcal{E}$ must be finite;
			\item The combinatorial graph induced by $(\mathcal{V},\mathcal{E})$ cannot have multiple edges nor loops, and must be connected;
			\item $\Gamma$ satisfies a distance consistency: let $d(\cdot,\cdot)$ be the geodesic distance on $\Gamma$. Then, for each edge $e\in\mathcal{E}$ connecting the vertices $v_1,v_2\in\mathcal{V}$, $d(v_1,v_2) = l_e$ holds.
		\end{enumerate}
	\end{Definition}

	Let us recall two specific cases of metric graphs with Euclidean edges, as defined in \cite{anderes2020isotropic}. The first case involves a compact metric graph $\Gamma$ with a vertex set $\mathcal{V}$ and an edge set $\mathcal{E}$, where the combinatorial graph $(\mathcal{V}, \mathcal{E})$ forms a tree. Such a graph is an example of a metric graph with Euclidean edges and is referred to as a \emph{Euclidean tree}. The second case is a compact metric graph $\Gamma$ where the edges are Euclidean, and $\Gamma$ forms a cycle. In this case, we refer to $\Gamma$ as a \emph{Euclidean cycle}.  Examples of Euclidean cycles can be seen in the two leftmost graphs in Figure~\ref{fig:graphs}, while the third graph, from left to right, illustrates an example of a Euclidean tree.

	\begin{figure}
		\begin{center}
			\begin{tikzpicture}[scale=0.5] 
				\draw (0,0) circle (pi);
				\filldraw [gray] (-pi,0) circle (4pt);
				\filldraw [gray] (pi,0) circle (4pt);
				\filldraw [gray] (0,pi) circle (4pt);
				\filldraw [gray] (0,-pi) circle (4pt);
				
				\draw (5,0) -- (7,2)  {};
				\draw (7,2) -- (8,2)  {};
				\draw (8,2) -- (9,3)  {};
				\draw (9,3) -- (10,0)  {};
				\draw (10,0) -- (8,1)  {};
				\draw (8,1) -- (7,-3)  {};
				\draw (7,-3)  to[out=80,in=20] (5,0)  {};
				\filldraw [gray] (5,0) circle (4pt);
				\filldraw [gray] (7,2) circle (4pt);
				\filldraw [gray] (8,2) circle (4pt);
				\filldraw [gray] (9,3) circle (4pt);
				\filldraw [gray] (10,0) circle (4pt);
				\filldraw [gray] (8,1) circle (4pt);
				\filldraw [gray] (7,-3) circle (4pt);

				\draw (12,0) -- (13,1)  {};
			\draw (12,0) -- (13,-1)  {};
			\draw (13,1) -- (15,3)  {};
			\draw (13,1) to[out=-20,in=70] (15,-1)  {};
			\draw (13,-1) -- (14,0)  {};
			\draw (13,-1) -- (16,-3)  {};
				\filldraw [gray] (12,0) circle (4pt);
			\filldraw [gray] (13,1) circle (4pt);
			\filldraw [gray] (13,-1) circle (4pt);
			\filldraw [gray] (15,3) circle (4pt);
			\filldraw [gray] (15,-1) circle (4pt);
			\filldraw [gray] (14,0) circle (4pt);
			\filldraw [gray] (16,-3) circle (4pt);

			\draw (19,pi/2) circle (pi/2);
			\draw (19,-pi/2) circle (pi/2);
			\filldraw [gray] (19,0) circle (4pt);
			\filldraw [gray] (19,pi) circle (4pt);
			\filldraw [gray] (19,-pi) circle (4pt);
			\filldraw [gray] (19-pi/2,-pi/2) circle (4pt);
			\filldraw [gray] (19-pi/2,pi/2) circle (4pt);
			\filldraw [gray] (19+pi/2,-pi/2) circle (4pt);
			\filldraw [gray] (19+pi/2,pi/2) circle (4pt);
			\end{tikzpicture}
	\end{center}
	\caption{Examples of graphs with Euclidean edges: Two Euclidean cycles (the two leftmost graphs), a metric tree, and a $1$-sum of two Euclidean cycles.}
	\label{fig:graphs}
	\end{figure}
   
	We now present the definition of the resistance metric on metric graphs with Euclidean edges, as introduced in \cite{anderes2020isotropic}. Let $\Gamma$ be a metric graph with Euclidean edges with vertices $\mathcal{V}$ and edges $\mathcal{E}$. To define the resistance metric, we first construct an auxiliary process, denoted by $Z_\Gamma$, on $\Gamma$, that has the following form:
	\begin{equation}\label{eq:aux_process_resmetric}
	Z_\Gamma(s) = Z_\mu(s) + \sum_{e\in\mathcal{E}} Z_e(s),\quad s\in\Gamma,
	\end{equation}
	where $\{Z_\mu(\cdot), Z_e(\cdot): e\in\mathcal{E}\}$ is a collection of independent Gaussian fields on $\Gamma$ which we will construct next. We start with $Z_\mu(\cdot)$, which is a field determined by its values on $\mathcal{V}$. To this end, first, define the function $c:\mathcal{V}\times\mathcal{V}\to [0,\infty)$ given by
	$$c(v_1,v_2) = \begin{cases}
	1/d(v_1,v_2),&\hbox{if $u\sim v$},\\
	0,& \hbox{otherwise.}
	\end{cases},$$
	where $u\sim v$ means that there is an edge connecting $v_1$ to $v_2$. Now, fix an arbitrary vertex $v_\in\mathcal{V}$, and let $L:\mathcal{V}\times\mathcal{V}\to\mathbb{R}$ be the function given by
	$$L(v_1,v_2) = \begin{cases}
		1+c(v_0),&\hbox{if $v_1=v_2=v_0$},\\
		c(v_1),&\hbox{if $v_1=v_2\neq v_0$},\\
		-c(v_1,v_2),& \hbox{otherwise},
	\end{cases},$$
	where $c(v_1) = \sum_{v\in\mathcal{V}} c(v_1,v_2) = \sum_{v_2\sim v_1} c(v_1,v_2)$. We are now in a position to define the field $Z_\mu(\cdot)$. First, let $\mathcal{V} = \{v_1,\ldots,v_n\}$ and observe that the function $L(\cdot,\cdot)$ can be identified with a matrix with an $n\times n$ matrix $L$, with $(i,j)$th entry given by $L_{i,j} = L(v_i,v_j)$. It is possible to show that $L$ is strictly positive-definite. Let $L^{-1}$ be the inverse of the matrix $L$. The field $Z_\mu(\cdot)$ is defined on $\mathcal{V}$ as a random vector following the multivariate Gaussian distribution
	$$(Z_\mu(v_1),\ldots,Z_\mu(v_n)) \sim N(0, L^{-1}).$$
	Finally, let us define $Z_\mu(\cdot)$ on an arbitrary point $s\in\Gamma$. To this end, let $s\in\Gamma$ be any point, let $e\in\mathcal{E}$ be the edge that contain $s$, and let $v_1$ and $v_2$ be the vertices, which are endpoints of $e$. Then, we define
	\begin{equation}\label{eq:def_process_Z_mu}
		Z_\mu(s) = (1-d(s)) Z_\mu(v_1) + d(s) Z_\mu(v_2),
	\end{equation}
	where $d(s) = d(s,v_1)/d(v_1,v_2).$ This concludes the construction of $Z_\mu(\cdot)$. Let us now construct the fields $Z_e(\cdot)$, for $e\in\mathcal{E}$. 
	To this end, let $\{B_e(\cdot):e\in\mathcal{E}\}$ be a collection of independent Brownian bridges constructed on the same probability space as $Z_\mu(\cdot)$ and that are also independent of $Z_\mu(\cdot)$, in which for each edge $e$, the Brownian bridge $B_e(\cdot)$ is defined on $e$, which is identified with the interval $[0,l_e]$, so that $P(B_e(0)=0) = P(B_e(l_e) =0 )=1$ and for $t\in (0,l_e)$ we have $P(B_e(t)=0)=0$. Therefore, we define $Z_e(\cdot)$ on an arbitrary point $s\in\Gamma$, $s=(t,\widetilde{e})$, as
	\begin{equation}\label{eq:def_process_Z_e}
		Z_e(s) = Z_e(t,\widetilde{e})= \begin{cases}
			B_e(t),&\hbox{if $\widetilde{e}=e$},\\
			0,&\hbox{otherwise}.
			\end{cases}.
	\end{equation}
	By combining \eqref{eq:def_process_Z_mu}, \eqref{eq:def_process_Z_e} and \eqref{eq:aux_process_resmetric}, we obtain the definition of $Z_\Gamma(\cdot)$. We are finally in a position to define the resistance metric:
	
	\begin{Definition}\label{def:resistance_metric}
		Let $\Gamma$ be a metric graph with Euclidean edges and let $Z_\Gamma$ be constructed by using \eqref{eq:def_process_Z_mu}, \eqref{eq:def_process_Z_e} and \eqref{eq:aux_process_resmetric}. The resistance metric is defined as the variogram of $Z_\Gamma$:
		\begin{equation}\label{eq:res_metric_def}
		d_R(s_1,s_2) = \Var(Z_\Gamma(s_1)-Z_\Gamma(s_2)).
		\end{equation}
	\end{Definition}

	The following proposition, proved in \cite[Proposition 4]{anderes2020isotropic}, provides a comparison between the geodesic metric and the resistance metric:
	\begin{Proposition}\label{prp:anderes_resistance_geodesic}
		Let $\Gamma$ be a graph with Euclidean edges. We have
		$$\forall s_1,s_2\in\Gamma,\quad d_R(s_1,s_2) \leq d(s_1,s_2).$$
		Further, if $\Gamma$ is a tree, Euclidean tree, then the geodesic and resistance metric coincide. Finally, if $\Gamma$ is an Euclidean cycle with total length $\ell = \sum_{e\in\mathcal{E}} l_e$, then 
		\begin{equation}\label{eq:resistance_cycle_geodesic}
			\forall s_1,s_2\in\Gamma,\quad d_R(s_1,s_2) = d(s_1,s_2) - \frac{d(s_1,s_2)^2}{\ell}.
		\end{equation}
	\end{Proposition}

	Finally, let us recall the notion of 1-sum of metric spaces, which, in particular, provides the notion of 1-sum of metric graphs.

	\begin{Definition}\label{def:1sum_spaces}
Let $(X_1, d_1)$ and $(X_2, d_2)$ be two metric spaces such that $X_1 \cap X_2 = \{x_0\}$. The 1-sum of $(X_1, d_1)$ and $(X_2, d_2)$, denoted by $(X_1 \cup X_2, d)$, is the metric space defined as follows:
\[
d(x, y) = 
\begin{cases} 
d_1(x, y), & \text{if } x, y \in X_1, \\
d_2(x, y), & \text{if } x, y \in X_2, \\
d_1(x, x_0) + d_2(x_0, y), & \text{if } x \in X_1 \text{ and } y \in X_2.
\end{cases}
\]
	\end{Definition}

By iterating the previous construction, we arrive at the definition of a $k$-step 1-sum of metric graphs.

\begin{Definition}\label{def:kstep_1sum}
Given an ordered collection of metric graphs $\Gamma_1, \ldots, \Gamma_k$, with $k \in \mathbb{N}$, suppose there exist points $v_2 \in \Gamma_2, \ldots, v_k \in \Gamma_k$ such that:
\[
\Gamma_1 \cap \Gamma_2 = \{v_1\}, \quad (\Gamma_1 \cup \Gamma_2) \cap \Gamma_3 = \{v_3\}, \quad \ldots, \quad (\Gamma_1 \cup \cdots \cup \Gamma_{k-1}) \cap \Gamma_k = \{v_k\}.
\]
Let $\widetilde{\Gamma}_2$ denote the 1-sum of $\Gamma_1$ and $\Gamma_2$, and define $\widetilde{\Gamma}_3$ as the 1-sum of $\widetilde{\Gamma}_2$ and $\Gamma_3$. This process is continued iteratively, so that $\widetilde{\Gamma}_k$ is the 1-sum of $\widetilde{\Gamma}_{k-1}$ and $\Gamma_k$. We call $\widetilde{\Gamma}_k$ the $k$-step 1-sum of $\Gamma_1, \ldots, \Gamma_k$, with intersecting points $\{v_2, \ldots, v_k\}$.
\end{Definition}

We will provide a simple result involving geodesic metric, resistance metric and the covariance function $r_{1,\ell}(\cdot)$ given in \eqref{eq:formula_pitt}.

\begin{Proposition}\label{prp:aux_prop_cycle_res_geo}
Let $\Gamma_0$ be given by the 1-sum between two Euclidean cycles $S_1$ and $S_2$ with respective lengths $\ell_1\neq\ell_2$, and let $v$ be the intersection point between $S_1$ and $S_2$. Then, there exists $s_1\in S_1$ and $s_2\in S_2$ such that $d_R(v,s_1) = d_R(v,s_2)$, and
$$r_{1,\ell_1}(d(s_1,v)) \neq r_{1,\ell_2}(d(s_2,v)).$$
\end{Proposition}

\begin{proof}
Begin by noting that inverting the expression in \eqref{eq:resistance_cycle_geodesic} from Proposition \ref{prp:anderes_resistance_geodesic}, i.e., solving the corresponding quadratic equation, which only has one positive root, gives the following result for $d_R(s_1, v) \leq \ell/4$ and $d_R(s_2, v) \leq \ell/4$:
$$d(s_1,v) = \frac{\ell_1}{2}\left(1 + \frac{\sqrt{\ell_1 - 4 d_R(s_1,v)}}{\sqrt{\ell_1}}\right)\quad\hbox{and}\quad d(s_2,v) = \frac{\ell_2}{2}\left(1 + \frac{\sqrt{\ell_2 - 4 d_R(s_2,v)}}{\sqrt{\ell_2}}\right).$$
Now, let $s_1\in S_1$ and $s_2\in S_2$ be such that $d_R(s_1,v) = d_R(s_2,v) = h$. Then,
\begin{equation}\label{eq:expr_cycle_1}
r_{1,\ell_1}(d(s_1,v)) = \frac{\cosh(\kappa_1(\ell_1\sqrt{\ell_1 - 4 h}/(2\sqrt{\ell_1})))}{2\kappa_1\tau_1^2\sinh(\kappa_1 \ell_1/2)}
\end{equation}
and
\begin{equation}\label{eq:expr_cycle_2}
r_{1,\ell_2}(d(s_2,v)) = \frac{\cosh(\kappa_2(\ell_2\sqrt{\ell_2 - 4 h}/(2\sqrt{\ell_2})))}{2\kappa_2\tau_2^2\sinh(\kappa_2 \ell_2/2)}.
\end{equation}
It is now immediate, by the explicit expressions given in the right-hand side of \eqref{eq:expr_cycle_1} and \eqref{eq:expr_cycle_2} and the fact that $\ell_1\neq\ell_2$, that there exists some ${h^\ast\in [0,\min\{\ell_1/4,\ell_2/4\}]}$ such that these expressions are different. The result thus follows by taking $s_1\in S_1$ and $s_2\in S_2$ such that $d_R(s_1,v) = d_R(s_2,v) = h^\ast$.
\end{proof}

Let us now recall a basic result on the theory of stationary Gaussian processes, for which we provide the proof for completeness.

\begin{Proposition}\label{prp:stationary_gaussian_interval_is_ou}
Let $u(\cdot)$ be a centered non-degenerate stationary Gaussian process on the interval $[0,\ell]$, where $\ell>0$, with a continuous covariance function. If $u(\cdot)$ is a Markov process (equivalently, a Gaussian Markov random field of order 1 on $I$), then its covariance function is given by
$$
\Cov(u(s_1), u(s_2)) = \sigma^2 \exp\{-\kappa |s_2 - s_1|\}, \quad s_1, s_2 \in [0,\ell],
$$
where $\sigma^2 = \Var(u(0))$ and $\kappa \geq 0$.    
\end{Proposition}

\begin{proof}
In \cite[Chapter 3, Section 8]{feller_vol2} it is shown that the process $u(\cdot)$ is Markov if, and only if, for every $s_1<s_2<s_3\in [0,\ell]$, we have
\begin{equation}\label{eq:identity_borisov_feller}
	\Cov(u(s_1), u(s_2))\Cov(u(s_2), u(s_3)) = \Var(u(s_2))\Cov(u(s_1), u(s_3)).
\end{equation}
Since $u(\cdot)$ is non-degenerate, we can define the auxiliary function $h(s) = \Cov(u(0),u(s))/\Var(u(0))$. Therefore, by stationarity of $u(\cdot)$ and \eqref{eq:identity_borisov_feller}, we have that for $s_1,s_2\in [0,\ell]$,
\begin{equation}\label{eq:identity_exp_covariance_function_markov}
	h(s_1 + s_2) = h(s_1)h(s_2).
\end{equation}
Since $u(\cdot)$ has a continuous covariance function, we can fix $q\in \mathbb{N}$ such that $1/q \leq \ell$, and use \eqref{eq:identity_exp_covariance_function_markov} to obtain that for every $s\in [0,\ell]$,
$$h(s) = \exp\left\{q s \log(h(1/q))\right\}.$$
This, directly implies that for every $s\in [0,\ell]$,
$$\Cov(u(s), u(0)) = \Var(u(0)) \exp\left\{-\kappa s\right\} = \sigma^2 \exp\left\{-\kappa s\right\},$$
where $\kappa = q\log(\Var(u(0))) - q \log(\Cov(u(1/q),u(0))) \geq 0$, since by Cauchy-Schwarz and stationarity, we have $\Cov(u(1/q),u(0))\leq \Var(u(0))$.
\end{proof}

We conclude this section with an auxiliary result regarding isotropy and Markovianity of Gaussian processes on metric graphs:

\begin{Proposition}\label{prp:1sum_euclidean_edge_iso}
Let $\Gamma_0$ be the 1-sum between a Euclidean cycle $S$ and an edge $e_1$. Suppose $X(\cdot)$ is a GRF on $\Gamma_0$ with an isotropic covariance function $\rho(s,t) = r(\widetilde{d}(s,t))$, where $r(\cdot)$ is a continuous function and $\widetilde{d}(\cdot,\cdot)$ denotes either the resistance metric or the geodesic metric. Then $X$ is not Markov of order 1.
\end{Proposition}
\begin{proof}
Let $v$ be the intersecting point between $S$ and $e_1$, that is, $\{v\}=S\cap\{e_1\}$, and let $\ell$ denote the length of $S$. Let us assume that $X_{\Gamma_0}$ is isotropic with respect to the resistance metric and Markov of order 1.
 Since the restriction of $X$ to $S$ is Markov of order 1, we have, by Theorem~\ref{thm:markovCircle} and Corollary~\ref{cor:circle}, that for any two locations $s$ and $s'$ on the cycle $S$, $r(d_{R,S}(s,s')) = r_{1,\ell}(d(s,s'))$, where $d_{R,S}$ is the resistance metric on $S$. 
 Since $\Gamma_0$ is the 1-sum between $S$ and $e_1$, we have that the restriction of the resistance metric $d_R$ on $\Gamma$ to $S$ coincides with $d_{R,S}$. Therefore, for every $s_2\in S$, we have 
$$r(d_R(s_2,v)) = r(d_{R,S}(s_2,v)) = r_{1,\ell}(d(s_2,v)).$$ 
Similarly, also due to $\Gamma_0$ being a 1-sum, the restriction of the resistance metric $d_R$ on $\Gamma_0$ to $e_1$ coincides with the resistance metric defined on $e_1$, which in turn, by Proposition \ref{prp:anderes_resistance_geodesic},
 coincides with the geodesic metric on $e_1$. Therefore, for any $\tilde{s},\tilde{s}'\in e_1$, we have $d_R(\tilde{s},\tilde{s}') = d(\tilde{s},\tilde{s}')$. 
 Now, observe that the restriction of $X$ to $e_1$ is Markov of order, Proposition \ref{prp:stationary_gaussian_interval_is_ou} 
 implies that for every $s_1\in e_1$, we have that there must exist $\kappa_1,\sigma>0$ such that
$$r(d_R(v,s_1)) = \sigma^2 \exp\{-\kappa_1 d(s_1,v)\}.$$
Therefore, by isotropy, there must exist $\sigma,\tau,\kappa_1,\kappa_2>0$, such that 
\begin{align}\label{eq:identity_different_cov_iso}
	\sigma^2 \exp\{-\kappa_1 d(s_1,v)\} &= r(d_R(v,s_1)) = r(d_R(v,s_2)) = r_{1,\ell}(d(s_2,v))  \notag \\
	&= \frac{\cosh(\kappa_2(d(v,s_2)-\ell/2))}{2\kappa_2\tau^2\sinh(\kappa_2 \ell/2)}.
\end{align}
for all points $s_1\in e_1$ and $s_2\in S$ such that $d_R(s_1,v) = d_R(s_2,v)$. Now, recall that, on one hand, 
\begin{equation}\label{eq:identity_geodist_resdist_edge}
\forall s_1\in e_1, \quad d(s_1,v) = d_R(s_1,v).
\end{equation} 
On the other hand, for every $s_2\in S$ such that $d_R(s_2,v) \leq \ell/4$, we can invert the expression \eqref{eq:resistance_cycle_geodesic} in Proposition \ref{prp:anderes_resistance_geodesic} to obtain
\begin{equation}\label{eq:identity_geodist_resdist_cycle}
	d(s_2,v) = \frac{\ell}{2}\left(1 + \frac{\sqrt{\ell - 4 d_R(s_2,v)}}{\sqrt{\ell}}\right).
\end{equation}

Therefore, by identities \eqref{eq:identity_different_cov_iso}, \eqref{eq:identity_geodist_resdist_edge}, and \eqref{eq:identity_geodist_resdist_cycle}, we must have, for every $s_1 \in e_1$ and $s_2 \in S$ such that $d_R(s_1, v) = d_R(s_2, v)$:
\begin{align*}
	\sigma^2 \exp\{-\kappa_1 d_R(s_2,v)\} &= \sigma^2 \exp\{-\kappa_1 d_R(s_1,v)\} = \sigma^2 \exp\{-\kappa_1 d(s_1,v)\}\\
	&=\frac{\cosh(\kappa_2(d(s_2,v)-\ell/2))}{2\kappa_2\tau^2\sinh(\kappa_2 \ell/2)}
	= \frac{\cosh(\kappa_2(\ell\sqrt{\ell - 4 d_R(s_2,v)}/(2\sqrt{\ell})))}{2\kappa_2\tau^2\sinh(\kappa_2 \ell/2)},
\end{align*}
that is, we must have for every $0\leq h \leq \min\{l_{e_1}, \ell/4\}$, that
$$\sigma^2 \exp\{-\kappa_1 h\} = \frac{\cosh(\kappa_2(\ell\sqrt{\ell - 4 h}/(2\sqrt{\ell})))}{2\kappa_2\tau^2\sinh(\kappa_2 \ell/2)},$$
which is a contradiction, and proves for the case in which $X(\cdot)$ is isotropic with respect to the resistance metric.

Finally, the proof for the geodesic distance is simpler. By repeating the previous argument, we find that on one hand, we must have
$$\forall s \in e_1, \, r(d(s_1, v)) = \sigma^2 \exp\left\{-\kappa d(s_1, v)\right\},$$
implying that $r(h) = \sigma^2 \exp\left\{-\kappa h\right\}$ for $0 \leq h \leq l_{e_1}$. On the other hand, we must have
$$\forall s \in e_2, \, r(d(s_2, v)) = r_{1, \ell}(d(s_2, v)),$$
which implies $r(h) = r_{1, \ell}(h)$ for all $0 \leq h \leq \ell$. This leads to a contradiction for $0 \leq h \leq \min\{l_{e_1}, \ell\}$, proving the result for the geodesic distance and concluding the proof.
\end{proof}

	\section{Auxiliary results and proofs for Sections \ref{sec:GWM} and \ref{sec:proofsGWM}}\label{app:aux}
	Recall the definition of the operator $L$ in \eqref{eq:operatorL}.

	The following identification of the spaces $\dot{H}_L^\alpha(\Gamma)$ given in \cite[Theorem 4.1]{bolinetal_fem_graph} is important.
	\begin{Theorem}\label{thm:characterization}
		Let $0<\alpha\leq 2$, where $\alpha\neq \nicefrac12$ and $\alpha\neq \nicefrac32$. 
		Then, 
		$$\dot{H}_{L}^\alpha(\Gamma) \cong \begin{cases}
			\widetilde{H}^\alpha(\Gamma),& \hbox{if } 0<\alpha<\nicefrac12,\\
			\widetilde{H}^\alpha(\Gamma) \cap C(\Gamma), & \hbox{if } \nicefrac12<\alpha<\nicefrac32,\\
			\widetilde{H}^\alpha(\Gamma) \cap C(\Gamma) \cap K^\alpha(\Gamma),& \hbox{if } \nicefrac32<\alpha\leq 2,
		\end{cases}$$
		where
		$K^\alpha(\Gamma) = \left\{u\in\widetilde{H}^\alpha(\Gamma): \forall v\in\mathcal{V}, \sum_{e\in \mathcal{E}_v} \partial_e u(v) = 0 \right\}.$
	\end{Theorem}

	To characterize the boundary spaces $H_+(\partial S)$, we require the following  technical lemma that characterizes 
	functions on Sobolev spaces on metric graphs whose trace operator is zero. We
	refer the reader to \cite[Theorem 2.10]{bolinetal_fem_graph} for the definition and main
	properties of the trace operator on metric graphs. For a Borel set 
	$S\subset\Gamma$, we define
	$\dot{H}_L^\alpha(S) = \{f|_S: f\in \dot{H}_L^\alpha(\Gamma)\}$
	endowed with the $\|\cdot\|_{\widetilde{H}^\alpha(S)}$ norm.
	
	We begin by proving the following result which reveals that, under Assumption \ref{assump:basic_assump2}, the operator $L$ behaves similarly to the Laplacian.
	
	\begin{Lemma}\label{lem:LHtildeAlphaHtildeAlphaPlus2}
		Let Assumption \ref{assump:basic_assump2} hold for $\alpha=k+2,$ where $k\in\mathbb{N}$.
		Then, $f\in \widetilde{H}^{k+2}(\Gamma)$ if and only if
		${f\in \widetilde{H}^k(\Gamma)}$ and $Lf\in\widetilde{H}^k(\Gamma)$.
	\end{Lemma}
	
	\begin{proof}
		First, if $f\in \widetilde{H}^{k+2}(\Gamma)$, then clearly $f\in \widetilde{H}^k(\Gamma)$. By
		Assumption \ref{assump:basic_assump2}  with \cite[Theorem 1.4.1.1]{grisvard}, we also obtain that $Lf\in\widetilde{H}^k(\Gamma)$. 
		Conversely, let $f\in \widetilde{H}^k(\Gamma)$ and $Lf\in\widetilde{H}^k(\Gamma)$. 
		Clearly, $f\in \widetilde{H}^k(\Gamma)$ and $Lf\in\widetilde{H}^k(\Gamma)$ if and only if $f_e \in H^k(e)$ and $L f_e \in H^k(e)$ for every edge $e\in\mathcal{E}$. 
		Take any edge $e\in\mathcal{E}$. By applying \cite[Theorem~1.4.1.1]{grisvard}, we obtain that under Assumption \ref{assump:basic_assump2}, $f_e \in H^k(e)$ implies that $\kappa_e^2 f_e \in H^k(e)$. Thus, ${\nabla \H_e \nabla f_e \in H^k(e)}$ since $Lf_e \in H^k(e)$, which shows  that $\H_e \nabla f_e \in H^{k+1}(e)$. 
		By Assumption~\ref{assump:basic_assump2}, $\H_e \in C^{k,1}(e)$ and because $\H_e$ is positive,  $1/\H_e \in C^{k,1}(e)$. 
		Therefore, by \cite[Theorem~1.4.1.1]{grisvard}, ${\nabla f_e = (1/\H_e) \H_e \nabla f_e \in H^{k+1}(e)}$, and thus, $f_e \in H^{k+2}(e)$.
	\end{proof}

	As a consequence, the following lemma is obtained, which is needed to prove Proposition~\ref{prp:CharGen_CM_natural_alpha}.

	\begin{Lemma}\label{lem:LHtildeM}
		Let $m\in\mathbb{N}$, where $m\geq 1$, and let Assumption \ref{assump:basic_assump2} hold for $\alpha = 2m$.
		Then,
		\begin{equation}\label{eq:equivalence1_Htilde}
			f\in \widetilde{H}^{2(m-1)}(\Gamma) \quad\hbox{and}\quad L^kf\in\widetilde{H}^2(\Gamma), k=0,\ldots,m-1 \quad\Leftrightarrow\quad f\in \widetilde{H}^{2m}(\Gamma).
		\end{equation}
		Similarly, if Assumption \ref{assump:basic_assump2} holds for $\alpha = 2m+1$, then
		\begin{equation}\label{eq:equivalence2_Htilde}
			f\in \widetilde{H}^{2m}(\Gamma) \quad\hbox{and}\quad L^{k}f\in\widetilde{H}^1(\Gamma), k=0,\ldots,m \quad\Leftrightarrow\quad f\in \widetilde{H}^{2m+1}(\Gamma).
		\end{equation}
	\end{Lemma}
	
	\begin{proof}
		We begin by proving \eqref{eq:equivalence1_Htilde}. 
		First, if $f\in \widetilde{H}^{2m}(\Gamma)$, then Assumption \ref{assump:basic_assump2} with the product rule for derivatives and \cite[Theorem 1.4.1.1]{grisvard} directly imply that $L^kf\in\widetilde{H}^2(\Gamma)$, for $k=0,\ldots,m-1$, and clearly $f\in \widetilde{H}^{2(m-1)}(\Gamma)$. 
		Conversely, if $m=1$, it is trivial. Now, assume, by induction, that \eqref{eq:equivalence1_Htilde} holds for $m>1$. 
		We must show that if Assumption \ref{assump:basic_assump2} holds for $\alpha=2m+2$, then \eqref{eq:equivalence1_Htilde} holds for $m+1$. 
		We already proved one direction for all $m$; hence, we assume that $f\in \widetilde{H}^{2m}(\Gamma)$ and $L^kf\in\widetilde{H}^2(\Gamma), k=0,\ldots,m$. We set $g = Lf$, and by Lemma \ref{lem:LHtildeAlphaHtildeAlphaPlus2},  $g\in \widetilde{H}^{2(m-1)}(\Gamma)$. Furthermore, by assumption, $L^k g \in\widetilde{H}^2(\Gamma)$, for ${k=0,\ldots,m-1}$. Hence, by the induction hypothesis, $g\in \widetilde{H}^{2m}(\Gamma)$. However, then, $f\in\widetilde{H}^{2m}(\Gamma)$ and $Lf=g\in\widetilde{H}^{2m}(\Gamma)$. Therefore, by Lemma \ref{lem:LHtildeAlphaHtildeAlphaPlus2} again, $f\in \widetilde{H}^{2m+2}(\Gamma)$, proving the induction.
		Finally, \eqref{eq:equivalence2_Htilde} can be proved similarly. 
	\end{proof}

	The following technical lemma will be need in Propositions \ref{prp:CharGen_CM_natural_alpha} and \ref{prp:Hdot3Local}.

	\begin{Lemma}\label{lem:dotHklocal}
		Let Assumption \ref{assump:simplest_assumption} hold. 
		If $\alpha = 2m + 1$, $m\in\mathbb{N}$, then
		for any ${u,v\in\dot{H}^\alpha_L(\Gamma)}$, 
		\begin{equation}\label{eq:localInnerProddotHk_1}
			\begin{aligned}
				(u,v)_{\dot{H}_L^\alpha(\Gamma)} &= (\kappa^2 L^m u, L^m v)_{L_2(\Gamma)} + (\H\nabla L^m u, \nabla L^m v)_{L_2(\Gamma)}\\
				&= \sum_{e\in \mathcal{E}} (\kappa^2 L^m u, L^m v)_{L_2(e)} + (\H\nabla L^m u, \nabla L^mv)_{L_2(e)}.
			\end{aligned}
		\end{equation}
		Similarly, if $\alpha=2m$, where $m\in\mathbb{N}$, 
		\begin{equation}\label{eq:localInnerProddotHk_2}
			(u,v)_{\dot{H}_L^\alpha(\Gamma)} = (L^m u, L^m v)_{L_2(\Gamma)}= \sum_{e\in \mathcal{E}} (L^m u, L^m v)_{L_2(e)}.
		\end{equation}
	\end{Lemma}
	
	\begin{proof}
		Equation~\eqref{eq:localInnerProddotHk_2} holds by definition, so we prove \eqref{eq:localInnerProddotHk_1}. Let $\alpha = 2m+1$ and  $\{e_j\}_{j\in\mathbb{N}}$ be the complete orthonormal system of eigenvectors of $L$, with corresponding eigenvalues $\{\lambda_j\}_{j\in\mathbb{N}}$. For any $v\in \dot{H}^\alpha_L(\Gamma)$,
		\begin{equation}\label{eq:conv_series_L2_v}
			v = \sum_{j\in\mathbb{N}} (v, e_j)_{L_2(\Gamma)} e_j\quad\hbox{and}\quad L^m v = \sum_{j\in\mathbb{N}} \lambda_j^m (v, e_j)_{L_2(\Gamma)} e_j,
		\end{equation}
		where $\sum_{j\in\mathbb{N}} \lambda_j^{2m} \leq C \sum_{j\in\mathbb{N}} \lambda_j^{\alpha} <\infty$; thus, the series $\sum_{j\in\mathbb{N}} \lambda_j^{m} (v, e_j)_{L_2(\Gamma)} e_j$ converges in $L_2(\Gamma)$ and the series $\sum_{j\in\mathbb{N}} (v, e_j)_{L_2(\Gamma)} e_j$ converges in $\dot{H}^\alpha_L(\Gamma)$. 
		Because $\H$ is bounded, the multiplication operator $M_\H:L_2(\Gamma)\to L_2(\Gamma)$ defined by $f \mapsto \H f$ is also bounded.
		The operator $\nabla : H^1(\Gamma)\to L_2(\Gamma)$ is also bounded and  $L^m v \in \dot{H}^1_L(\Gamma) \cong H^1(\Gamma)$, where the last identification comes from Theorem \ref{thm:characterization}.
		Therefore, 
		\begin{equation}\label{eq:identityDerivEigen}
			\nabla L^m v = \sum_{j\in\mathbb{N}} \lambda_j^m (v,e_j) \nabla e_j,
		\end{equation}
		where the above sum converges in $L_2(\Gamma)$. Next,  $e_j\in D(L) = \dot{H}^2_L(\Gamma) \subset K^2(\Gamma)$, where $j\in\mathbb{N}$; therefore, for every $j\in\mathbb{N}$, 
		\begin{align}
			(u,e_j)_{\dot{H}^\alpha_L(\Gamma)} &= (L^{\alpha/2}u, L^{\alpha/2}e_j)_{L_2(\Gamma)} = (L^m u, L L^m e_j)_{L_2(\Gamma)} = (L^m u, (\kappa^2 - \nabla \H \nabla)L^m e_j)\notag\\
			&= (\kappa^2 L^m u, L^m e_j)_{L_2(\Gamma)} + (\nabla L^m u, \H\nabla L^m e_j)_{L_2(\Gamma)}\label{eq:step2_Localpart}\\
			&= (\kappa^2 L^m u, \lambda_j^m e_j)_{L_2(\Gamma)} + (\nabla L^m u, \H\lambda_j^m\nabla e_j)_{L_2(\Gamma)}\label{eq:step3_Localpart},
		\end{align}
		where in the first line, we applied the fact that $L^{1/2}$ is self-adjoint in $L_2(\Gamma)$, and in \eqref{eq:step2_Localpart} we employed integration by parts, the fact that $L^m u\in \dot{H}^1_L(\Gamma)\cong H^1(\Gamma)$ and the fact that $L^m e_j = \lambda_j^m e_j \in K^2(\Gamma)$. We can now multiply both sides of \eqref{eq:step3_Localpart} by $(v, e_j)_{L_2(\Gamma)}$, sum over $j\in\mathbb{N}$, and then use \eqref{eq:conv_series_L2_v} and \eqref{eq:identityDerivEigen} to obtain
		\begin{align*}
			(u,v)_{\dot{H}_L^\alpha(\Gamma)} &= (\kappa^2 L^m u, L^m v)_{L_2(\Gamma)} + (\H\nabla L^m u, \nabla L^m v)_{L_2(\Gamma)}\\
			&= \sum_{e\in \mathcal{E}} (\kappa^2 L^m u, L^m v)_{L_2(e)} + (\H\nabla L^m u, \nabla L^mv)_{L_2(e)}.
		\end{align*}
		This statement concludes the proof.
	\end{proof}

\begin{proof}[Proof of Proposition \ref{prp:CharGen_CM_natural_alpha}]
We begin by proving the equality of sets by induction.
We start assuming that $\alpha\in\mathbb{N}$ is even. The base case $\dot{H}^2_L(\Gamma) = \widetilde{H}^2(\Gamma)\cap C(\Gamma) \cap K^2(\Gamma)$ follows from Theorem \ref{thm:characterization}. Next, we let $\alpha=2m$, where $m>1$, be an even natural number. Then, by the isometry property of $L$ (see Remark \ref{rmk:isometry}), Lemma~\ref{lem:LHtildeM} 
and the induction hypothesis,
\begin{align*}
	&f\in \dot{H}^{2m}_L(\Gamma) 
	\Leftrightarrow f\in \dot{H}^{2(m-1)}_L(\Gamma) \,\, \hbox{and}\,\, L^{m-1}f \in \dot{H}^2_L(\Gamma)\\
	&\Leftrightarrow f\in \dot{H}^{2(m-1)}_L(\Gamma),\,\, f\in \widetilde{H}^{2(m-1)}(\Gamma),\,\, L^{m-1}f \in \widetilde{H}^2(\Gamma) \,\, \hbox{and}\,\, L^{m-1}f \in \dot{H}^2_L(\Gamma)\\
	&\Leftrightarrow f\in \dot{H}^{2(m-1)}_L(\Gamma),\,\, f\in \widetilde{H}^{2(m-1)}(\Gamma),\,\, L^{k}f \in \widetilde{H}^2(\Gamma), k=0,\ldots,m-1,\,\, L^{m-1}f \in \dot{H}^2_L(\Gamma)\\
	&\Leftrightarrow f\in \widetilde{H}^{2m}(\Gamma),\,\, f\in \dot{H}^{2(m-1)}_L(\Gamma) \,\, \hbox{and}\,\, L^{m-1}f \in \dot{H}^2_L(\Gamma),
\end{align*}
which proves the result for even numbers. 
If, on the other hand, $\alpha = 2m+1$, where $m\geq 1$, then 
from the result for even numbers, Lemma~\ref{lem:LHtildeM},
 the induction hypothesis, and Theorem \ref{thm:characterization}, 
\begin{align*}
	f\in \dot{H}^{2m+1}_L(\Gamma) 
	&\Leftrightarrow f\in \dot{H}^{2m}_L(\Gamma)\,\,\hbox{and}\,\, L^m f\in \dot{H}^1_L(\Gamma)\\
	&\Leftrightarrow f\in \dot{H}^{2m}_L(\Gamma),\,\, f\in\widetilde{H}^{2m}(\Gamma),\,\, L^m f\in \widetilde{H}^1(\Gamma) \,\,\hbox{and}\,\, L^m f\in C(\Gamma)\\
	&\Leftrightarrow f\in \dot{H}^{2m}_L(\Gamma),\,\, f\in\widetilde{H}^{2m}(\Gamma),\,\, L^k f\in \widetilde{H}^1(\Gamma),k=0,\ldots,m, \,\, L^m f\in C(\Gamma)\\
	&\Leftrightarrow f\in \widetilde{H}^{2m+1}(\Gamma),\,\, f\in \dot{H}^{2m}_L(\Gamma)\,\,\hbox{and}\,\, L^m f\in C(\Gamma),
\end{align*}
which concludes the proof of the equality of sets for all $\alpha\in\mathbb{N}$, where $\alpha\geq 2$. 

It remains to prove the equivalence of norms.
To that extent, note that if $\alpha = 2m,$  $m\in\mathbb{N}$, 
then by the product rule for derivatives and Assumption \ref{assump:basic_assump2}, a constant $C>0$ exists such that for $f\in \dot{H}^\alpha_L(\Gamma)$,
$
\|f\|_{\dot{H}_L^\alpha(\Gamma)} = \|L^{m} f\|_{L_2(\Gamma)} \leq C \|f\|_{\widetilde{H}^\alpha(\Gamma)}.
$
Similarly, if $\alpha = 2m+1$, where $m\in\mathbb{N}$, Lemma~\ref{lem:dotHklocal}, 
the product rule for derivatives and Assumption \ref{assump:basic_assump2}, shows that constants $C_1,C_2>0$ exist such that for $f\in\dot{H}^\alpha_L(\Gamma)$,
${\|f\|_{\dot{H}^\alpha_L(\Gamma)} \leq C_1 \|L^m f\|_{\widetilde{H}^1(\Gamma)} \leq C_2 \|f\|_{\widetilde{H}^\alpha(\Gamma)}}$.
Therefore, a constant $C>0$ exists such that for any ${f\in\dot{H}^\alpha_L(\Gamma)}$,  ${\|f\|_{\dot{H}^\alpha_L(\Gamma)} \leq C \|f\|_{\widetilde{H}^\alpha(\Gamma)}}$.  

We now prove the converse inequality, that $\dot{H}^\alpha_L(\Gamma) \hookrightarrow \widetilde{H}^\alpha(\Gamma)$. To this end, we must prove that the inclusion map $I:\dot{H}^\alpha_L(\Gamma) \to \widetilde{H}^\alpha(\Gamma)$ is a bounded operator. We have the inclusion by the first part of the proof, and $(\dot{H}^\alpha_L(\Gamma), \|\cdot\|_{\dot{H}_L^\alpha(\Gamma)})$ and $(\widetilde{H}^\alpha(\Gamma),\|\cdot\|_{\widetilde{H}^\alpha(\Gamma)})$ are Hilbert spaces. Next, $\lambda_k\to\infty$ as $k\to\infty$, where $\{\lambda_k\}_{k\in\mathbb{N}}$ are the eigenvalues of $L$, which shows the continuous embedding $\dot{H}^\alpha_L(\Gamma)\hookrightarrow L_2(\Gamma)$. 
Let $\phi_N\to 0$ in $\dot{H}^\alpha_L(\Gamma)$. Then, the continuous embedding of $\dot{H}^\alpha_L(\Gamma)$ in $L_2(\Gamma)$ implies that $\phi_N\to 0$ in $L_2(\Gamma)$. 
Now, assume that $I(\phi_N) \to \phi$ in $\widetilde{H}^\alpha(\Gamma)$. Then, $\|\phi_N-\phi\|_{L_2(\Gamma)}\leq \|\phi_N-\phi\|_{\widetilde{H}^\alpha(\Gamma)} \to 0$. Hence, $\phi = 0$, since $\phi_N\to 0$ in $L_2(\Gamma)$, which implies that there is a subsequence $\phi_{N_k}$ such that $\phi_{N_k} \to 0$ almost everywhere, and thus, $\phi = 0$. Therefore, the inclusion map is closed, and by the closed graph theorem, $I$ is a bounded operator.
\end{proof}

	Next, we show that , if $\alpha\in\mathbb{N}$ and Assumption \ref{assump:basic_assump2} holds, then the spaces $\dot{H}^\alpha_L(\Gamma)$ are local. 

	\begin{proof}[Proof of Proposition \ref{prp:Hdot3Local}]
		By the expressions 
		\eqref{eq:localInnerProddotHk_1} and \eqref{eq:localInnerProddotHk_2} 
		in Lemma~\ref{lem:dotHklocal}, 
		along with the product rule for weak derivatives (which holds under the assumptions on $\kappa$ and $\H$), we directly find that Definition~\ref{def:localCMspaces}.\ref{def:localCMspaces1} is satisfied. 
		For the condition in Definition \ref{def:localCMspaces}.\ref{def:localCMspaces2}, let $f\in \dot{H}^\alpha_L(\Gamma)$ be written as $f = f_1 + f_2$, with ${\textrm{supp}\, f_1 \cap \textrm{supp}\, f_2 = \emptyset}$. Then, by the characterization in Proposition \ref{prp:CharGen_CM_natural_alpha},  ${f \in \widetilde{H}^\alpha(\Gamma)}$. Further, the Sobolev spaces $H^\alpha(e)$ are local for $e\in\mathcal{E}$ as $\alpha\in\mathbb{N}$; thus,  $f_1,f_2 \in \widetilde{H}^\alpha(\Gamma)$. Finally, as $f\in\dot{H}^\alpha_L(\Gamma)$, $f$ satisfies the vertex constraints by Proposition \ref{prp:CharGen_CM_natural_alpha},  and since ${\textrm{supp}\, f_1 \cap \textrm{supp}\, f_2 = \emptyset}$,  $f_1$ and $f_2$ also satisfy the vertex constraints. Thus, $f_1,f_2\in\widetilde{H}^\alpha(\Gamma)$ because we already have  $f_1,f_2\in\dot{H}^\alpha_L(\Gamma)$. Therefore, the space  $\dot{H}_L^\alpha(\Gamma)$ is local.
	\end{proof}

	To prove Theorem \ref{thm:MarkovGenWM} we will some technical lemmata:
	
	\begin{Lemma}\label{lem:gen_polynomials}
		Let $\Gamma$ be a compact metric graph and assume that Assumption 1 hold for $a$ and $\kappa$, and also assume that for every edge $e$, we have $a,\kappa\in C^\infty(e)$. Then, if $e = [0,l_e]$, we have for any $0 < e_1 < e_2 < l_2$ and any $m\in\mathbb{N}$, that there exists some function $f\in C^\infty([e_1,e_2])$, that depends on $e_1, e_2, m$ and $a$, such that $\left( \frac{d}{d_{x_e}} a_e \frac{d}{dx_e}\right)^m f = 0$ on $(e_1,e_2)$.
	\end{Lemma}
	\begin{proof}
		Note that $\left( \frac{d}{d_{x_e}} a_e \frac{d}{dx_e}\right) f = 0$ on $(e_1,e_2)$ implies
		$a_e \frac{d}{dx_e} f = C_1$, for some constant $C_1\in\mathbb{R}$. This, in turn, implies that $f(x) = C_1 \int_{e_1}^x \frac{1}{a_e(t)}dt + C_2$. Observe that $f\in C^\infty([e_1,e_2])$ since $a_e\in C^\infty([e_1,e_2])$ and is bounded away from zero. By iterating this argument we can obtain such functions for every $m\in\mathbb{N}$.
	\end{proof}

	Next, we show that, given sufficiently smooth functions $\kappa$ and $\H$ on the edges, any smooth function with compact support in each edge (meaning that it is zero close to each vertex) belongs to all $\dot{H}^\alpha_L(\Gamma)$ for $\alpha\in\mathbb{N}$.

\begin{Proposition}\label{prp:CinfinityCompactHdot}
Let Assumption \ref{assump:simplest_assumption} hold and let $\kappa_e \in C^\infty(e)$ and ${\H_e\in C^\infty(e)}$ for every $e\in\mathcal{E}$. Then, 
$\bigoplus_{e\in\mathcal{E}} C^\infty_c(e) \subset \bigcap_{n\in\mathbb{N}} \dot{H}^n_L(\Gamma).$
\end{Proposition}
\begin{proof}
Observe that we have, by assumption, that ${\kappa, \H\in\bigoplus_{e\in\mathcal{E}} C^\infty(e)}$. Further, by Theorem \ref{thm:characterization}, we have that $\bigoplus_{e\in\mathcal{E}} C_c^\infty(e) \subset \dot{H}^2_L(\Gamma) = D(L)$.  Thus, for every $f\in\bigoplus_{e\in\mathcal{E}} C_c^\infty(e)$, $Lf \in \bigoplus_{e\in\mathcal{E}} C_c^\infty(e)$, implying that $Lf \in D(L)$ and thus that ${f\in D(L^2) = \dot{H}_L^4(\Gamma)}$. By induction,  ${f\in D(L^n) = \dot{H}_L^{2n}(\Gamma)}$ for every ${n\in\mathbb{N}}$. Therefore, 
$$
f \in \bigcap_{n\in\mathbb{N}} D(L^n) = \bigcap_{n\in\mathbb{N}} \dot{H}_L^{2n}(\Gamma) = \bigcap_{n\in\mathbb{N}} \dot{H}^n_L(\Gamma),
$$
and thus, $\bigoplus_{e\in\mathcal{E}} C^\infty_c(e) \subset \bigcap_{n\in\mathbb{N}} \dot{H}^n_L(\Gamma).$
\end{proof}

The following auxiliary lemma will also be needed for proving Theorem \ref{thm:MarkovGenWM}:

\begin{Lemma}\label{lem:local_coeff_Peetre}
Let Assumption \ref{assump:simplest_assumption} hold, and, for every edge $e\in\mathcal{E}$, let $\kappa_e, a_e\in C^\infty(e)$. Further, let $\alpha>0, \alpha\in\mathbb{R}$, be such that $(\dot{H}_L^\alpha(\Gamma), (\cdot,\cdot)_{\dot{H}^\alpha_L(\Gamma)})$ is local. Then, there exist $N\in\mathbb{N}$, and functions $b_r\in C^\infty(e)$, $r=0,\ldots,N$, such that for every ${f \in \bigoplus_{e\in\mathcal{E}} C_c^\infty(e)}$, we have
\begin{equation*}
	L^\alpha f(s) = \sum_{e\in\mathcal{E}} \sum_{r=0}^{N} \widetilde{b}_{r}^e(s) \frac{d^r \widetilde{f}_e(s)}{dx_e^r}(s),\quad s\in \Gamma,
\end{equation*}
where $\widetilde{b}^e_{r}$ and $\widetilde{f}_e$ are the extensions as zero in $\Gamma\setminus e$, that is, if $s\not\in e$, $\widetilde{b}^e_{r}(s) = \widetilde{f}_e(s) = 0$, and for $s\in e$, say $s = (t,e)$, we have $\widetilde{b}^e_{r}(s) = b^e_{r}(t)$ and $\widetilde{f}_e(s) = f_e(t)$. 
\end{Lemma}
\begin{proof}
Assume that $(\dot{H}_L^\alpha(\Gamma), (\cdot,\cdot)_{\dot{H}^\alpha_L(\Gamma)})$ is local. Fix any edge $e$ of $\Gamma$. Now, observe that in the same ideas as the ones in the proof of Proposition~\ref{prp:CinfinityCompactHdot}, we obtain that for every $f\in \bigoplus_{e\in\mathcal{E}} C^\infty_c(e)$, we have $L^{\alpha}f\in \bigoplus_{e\in\mathcal{E}} C^\infty(e)$. Furthermore, the local assumption gives us that, actually, $L^\alpha : \bigoplus_{e\in\mathcal{E}} C_c^\infty(e) \to \bigoplus_{e\in\mathcal{E}} C_c^\infty(e)$. Given $f_e\in C_c^\infty(e)$, let $\widetilde{f}_e$ be its extension as zero on $\Gamma\setminus e$, so that $L^\alpha \widetilde{f}_e \in  \bigoplus_{e\in\mathcal{E}} C^\infty_c(e)$. By the locality assumption, we have that actually, $supp(L^\alpha \widetilde{f}_e) \subset \textrm{int} (e)$. We can then define the operator $L_{e,\alpha} : C_c^\infty(e) \to C_c^\infty(e)$ by $L_{e,\alpha} f_e = L^\alpha \widetilde{f}_e$. This yields an inner product $\<\cdot,\cdot\>_{\alpha,e}$ on $C_c^\infty(e)$ given by ${\<f_e,g_e\>_{\alpha,e} = (L_{e,\alpha} f_e,g_e)_{L_2(e)}}$. Observe that since $(\dot{H}^\alpha_L(\Gamma),(\cdot,\cdot)_\alpha)$ is local, then $(C_c^\infty(e), \langle \cdot, \cdot\rangle_{e,\alpha})$ is local. By Proposition \ref{prp:CinfinityCompactHdot}, $\widetilde{f}_e$ belongs to the domain of $L^\alpha$. Therefore, since $L$ is self-adjoint, $\langle f_e, g_e\rangle_{e,\alpha} = (L^\alpha \widetilde{f}_e, \widetilde{g}_e)_{L_2(\Gamma)}$.	
By Peetre's theorem \cite{Peetre1,Peetre2}, 
if the (non-complete) inner product space $(C_c^\infty(e), \langle \cdot, \cdot\rangle_{e,\alpha})$ is local,
then the inner product $\langle \cdot, \cdot\rangle_{e,\alpha}$ has the following form:
$$
(\widetilde{f}_e, \widetilde{g}_e)_{\dot{H}^\alpha_L(\Gamma)} = \langle f_e, g_e\rangle_{e,\alpha} = \sum_{r=0}^{N} \int_0^{\ell_e} b_{r}(t)\frac{d^rf_e}{dx^r}(t) g_e(t) dt,
$$
where $N\in\mathbb{N}_0$, 
$f_e,g_e\in C_c^\infty(e)$, for all $r$, we have $b_{r}\in C^\infty(e)$, and $\widetilde{f}_e,\widetilde{g}_e$
are the extensions (as zero outside of $e$) of $f_e$ and $g_e$, respectively. By density of $C_c^\infty(e)$ in $L_2(e)$, we obtain that for $f_e\in C_c^\infty(e)$
$$
L^\alpha \widetilde{f}_e(s) = \sum_{r=0}^{N}  b_{r}(s)\frac{d^rf_e}{dx^r}(s).
$$
Because $e\in\mathcal{E}$ is arbitrary, and we can write any $f \in \bigoplus_{e\in\mathcal{E}} C_c^\infty(e)$ as $f = \sum_{e\in\mathcal{E}} \widetilde{f}_e$, where $f_e\in C_c^\infty(e)$ and $\widetilde{f}_e$ is the extension of $f_e$ as zero in $\Gamma\setminus e$, we obtain (after increasing $N$ if necessary) for ${f \in \bigoplus_{e\in\mathcal{E}} C_c^\infty(e)}$,
\begin{equation*}
	L^\alpha f(s) = \sum_{e\in\mathcal{E}} L^\alpha \widetilde{f}_e(s) = \sum_{e\in\mathcal{E}} \sum_{r=0}^{N} \widetilde{b}_{r}^e(s) \frac{d^r \widetilde{f}_e(s)}{dx_e^r}(s),\quad s\in \Gamma,
\end{equation*}
where $b_{r}^e\in C^\infty(e)$ and $\widetilde{b}^e_{r}$ is the extension by zero in $\Gamma\setminus e$. This concludes the proof.
\end{proof}

We now prove the result regarding the weak derivatives of differentiable GRFs, and that generalized Whittle--Mat\'ern fields are differentiable GRFs of order $\floor{\alpha}-1$. As a direct consequence, we obtain Corollary \ref{cor:BoundaryConditionsGen_WM}.

\begin{proof}[Proof of Proposition \ref{prp:boundaryweakderiv}]
For a subspace of $M$ of a Hilbert space $H$, let $M^\perp$ denote the orthogonal complement of $M$ in $H$. Now, it is clear that $u_e(s)\in H_+(\partial S)$ for each $s\in\partial S$ and each $e\in\mathcal{E}_s$. Next, we show that $u_e'(s)\in H_+(\partial S)$. 
Thus, let $S\subset\Gamma$ be a Borel set and let $(v_n)$ be a sequence of functions in $H(S)$, such that $v_n$ converges weakly (in $L_2(\Omega)$) to $v$. Then,
$v\in H(S)$. Indeed, since $H(S)$ is closed in $L_2(\Omega)$, $H(S)^{\perp\perp} = H(S)$. 
If ${w\in H(S)^\perp}$, then 
${\mathbb{E}(vw) = \lim_{n\to\infty} \mathbb{E}(v_nw) = 0}$,
because $v_n\in H(S)$ and $w\in H(S)^\perp$. Further, since ${w\in H(S)^\perp}$
was arbitrary, $v\in H(S)^{\perp\perp} = H(S)$ follows. 

Take any $\varepsilon>0$, any $s\in\partial S$, an edge $e\in\mathcal{E}_s$, and a sequence $s_n\in e$ such that
$s_n\to s$ as ${n\to\infty}$. Note that $u_e'(s)$ is the weak limit of $(u_e(s_n)-u_e(s))/(s_n-s)$, belonging to $H((\partial S)_\varepsilon)$ for a sufficiently 
large $n$. Therefore, it follows from the first part of the proof that
$u_e'(s) \in H((\partial S)_\varepsilon)$. As $\varepsilon>0$ is arbitrary, it follows from the definition of $H_+(\partial S)$ that $u_e'(s)\in H_+(\partial S)$.
The results for the higher-order derivatives can be proved in the same manner,
along with an inductive argument.
\end{proof}

\begin{proof}[Proof of Proposition \ref{prp:GenWMDiffAlpha}]
The proof that $u$ is a differentiable GRF of order $\floor{\alpha}-1$ follows from the same arguments as in \cite[Lemma 4]{BSW2022} together with Proposition \ref{prp:CharGen_CM_natural_alpha}.
For the second part of the statement, observe that if $\alpha <2$ there is nothing to prove. Now, under Assumption \ref{assump:basic_assump2}, and $\alpha\geq 2$, we have 
$L_e f_e = \kappa_e^2 f_e - \H_e' f_e' - \H_e f_e''(x)$, and for $\alpha\geq 3$,
\begin{align*}
\partial_e L_e f_e &= 2\kappa_e\kappa_e' f_e + \kappa_e^2 f_e' - \H_e'' f_e' - \H_e' f_e'' - \H_e' f_e'' - \H_ef_e''',
\end{align*} 
and so on. Therefore, the second part follows from the same arguments as in \cite[Corollary~5]{BSW2022} together with Proposition~\ref{prp:CharGen_CM_natural_alpha} and Assumption \ref{assump:basic_assump2}.
\end{proof}

\begin{proof}[Proof of Corollary \ref{cor:BoundaryConditionsGen_WM}]
 Proposition \ref{prp:CharGen_CM_natural_alpha} reveals that if $h\in H^\alpha_L(\Gamma)$, where $\alpha\in\mathbb{N}$ and Assumption \ref{assump:basic_assump2} holds, then 
for $k=0,\ldots, \lceil \alpha-\nicefrac{1}{2}\rceil-1$, if $k$ is even, then $L^{\nicefrac{k}{2}} h\in C(\Gamma)$. If $k$ is odd, then $L^{\nicefrac{(k-1)}{2}} h\in K^2(\Gamma)$, so that for every $v\in\mathcal{V}$, $\sum_{e\in\mathcal{E}_v} \partial_e L^{\nicefrac{(k-1)}{2}} h_e(v) = 0.$ Finally, by the arguments in \cite[Proposition 11]{BSW2022}, with Proposition \ref{prp:GenWMDiffAlpha}, we obtain the corresponding result for $u$, where the derivatives are weak derivatives in the $L_2(\Omega)$ sense.
\end{proof}

We require the following lemma to prove Theorem~\ref{thm:charInnerOuterEdge}. We omit the proof of it because it is 
analogous to the proof of Proposition \ref{prp:boundaryweakderiv}.

\begin{Lemma}\label{lem:InnerDerivativeBelongSpace}
Let  $u$ be a
differentiable GRF of order $p$, $p\geq 0$, on $\Gamma$. Then,
given any open set $S\subset\Gamma$ whose boundary consists of
finitely many points, we have for each
edge $e\in\mathcal{E}_S$ and each $s\in(\partial S)\cap e$,
$u_e(s), u_e'(s), \ldots, u_e^{(p)}(s) \in H_{+,I}(\partial S).$
Similarly, for each edge $\widetilde{e}\not\in \mathcal{E}_S$,
and each $s\in(\partial S)\cap \widetilde{e},$ we have
$u_{\widetilde{e}}(s), u_{\widetilde{e}}'(s), \ldots, u_{\widetilde{e}}^{(p)}(s) \in H_{+,O}(\partial S).$
\end{Lemma}


\begin{proof}[Proof of Proposition~\ref{thm:charInnerOuterEdge}]
By Lemma \ref{lem:InnerDerivativeBelongSpace} 
and \eqref{eq:defSpaceInner}, 
\begin{align*}
H_{+,I}(\partial S)\cap H_{\alpha}(\partial S) &= \textrm{span}\{u_e(s), u_e'(s),\ldots, u_e^{(\alpha-1)}(s):
s\in \partial S, e\in \mathcal{E}_S \},\\
H_{+,O}(\partial S)\cap H_{\alpha}(\partial S) &= \textrm{span}\{u_{\widetilde{e}}(s), u_{\widetilde{e}}'(s),\ldots, u_{\widetilde{e}}^{(\alpha-1)}(s):
s\in \partial S, \widetilde{e}\not\in \mathcal{E}_S \}.
\end{align*}
Also, by Theorem \ref{thm:boundarySpcsGenWM}, 
$H_+(\partial S) = H_\alpha(\partial S).$
Since we have $H_{+,I}(\partial S)\subset H_+(\partial S)$ and 
${H_{+,O}(\partial S) \subset H_+(\partial S)}$, we obtain the desired 
equalities.
\end{proof}

\begin{Remark}
One can also provide an alternative proof of Proposition~\ref{thm:charInnerOuterEdge} by mimicking the proof of Theorem~\ref{thm:boundarySpcsGenWM}, with the help of Theorem~\ref{thm:CharOrthComplCM_genWM}.
\end{Remark}
	
The next lemma will be used in the proof of Theorem \ref{thm:boundarySpcsGenWM}.

	\begin{Lemma}\label{lem:sobTraceZero}
		Let Assumption \ref{assump:basic_assump2} hold and $\alpha\in\mathbb{N}$. Further, let $S\subset\Gamma$ be an open set and take $h\in \dot{H}_L^\alpha(S)$. If for every $s\in\partial S$, and every
		edge $e$ of $E$ which is incident to $s$, $h_e^{(k)}(s) = 0$
		in the trace sense, for $k=0,1,\ldots, \alpha-1$. Then, 
		a sequence of functions $(\varphi_n)$ in $\dot{H}_L^\alpha(S)\cap C_c(S)$ exists 
		such that $\varphi_n\to h$ in $\dot{H}_L^\alpha(S)$.
		Further, $h$ can be extended to a function $\widetilde{h}$ in $\dot{H}_L^\alpha(\Gamma)$ by letting
		$\widetilde{h}(s) = h(s)$ if $s\in S$ and $\widetilde{h}(s)= 0$ if $s\in\Gamma\setminus S$.
	\end{Lemma}
	
	\begin{proof}
		Let $\mathcal{E}_S$ be the set of edges in $S$. 
		It is sufficient to prove the first result for each $e\cap S$, where ${e\in\mathcal{E}_S}$ because 
		$\|h - \varphi_n\|_{\widetilde{H}^\alpha(S)}^2 = \sum_{e\in \mathcal{E}_S} \|h - \varphi_n\|_{\widetilde{H}^\alpha(e\cap S)}^2$.
		Furthermore, by a standard partition of unity argument, it is sufficient to approximate a function $h$ in $H^\alpha([0,\infty))$, whose support is contained
		in $[0,b]$, for some $b>0$. The statement for this case follows directly
		from \cite[Theorem 3.40]{mclean}, asserting that, for an interval $I$, the space $H_0^\alpha(I)$ (the closure of $C_c^\infty(I)$ with respect to the $H^\alpha(I)$ norm) coincides with the space of functions $h$ in $H^\alpha(I)$ with $\gamma(h) = \gamma(h') = \cdots = \gamma(h^{(\alpha-1)}) = 0$, where $\gamma$ is the trace operator.
		
		For the second statement, let $h\in\dot{H}_L^\alpha(S)$ be such that a sequence $\varphi_n\in C_c(S)\cap \dot{H}^\alpha_L(S)$ exists satisfying $\varphi_n \to h$ in $\dot{H}^\alpha_L(S)$. Then, for each $n\in\mathbb{N}$, $\psi_n \in \dot{H}^\alpha_L(\Gamma)$ exists such that $\varphi_n = \psi_n|_S$. 
		Let $\widetilde{\varphi}_n(\cdot)$ be the extension of $\varphi_n$ as zero in $\Gamma\setminus S$. Similarly, let  $\widetilde{h}$ be the extension of $h$ as zero in $\Gamma\setminus S$, and define  $\tau_n = \psi_n - \widetilde{\varphi}_n$. 
		Because $\widetilde{\varphi}_n$ has support compactly contained in $S$ (recall that $S$ is open), $\textrm{supp}(\varphi_n) \cap \textrm{supp}(\tau) = \emptyset$. 
		Also, $\widetilde{\varphi}_n$ and $\tau_n$ belong to $\dot{H}^\alpha_L(\Gamma)$ because $\alpha\in\mathbb{N}$ and $\dot{H}^\alpha_L(\Gamma)$ is local. 
		Further, by Proposition \ref{prp:CharGen_CM_natural_alpha},
		the norm $\|\cdot\|_{\dot{H}^\alpha_L(\Gamma)}$ and $\|\cdot\|_{\widetilde{H}^\alpha(\Gamma)}$ are equivalent on $\dot{H}^\alpha_L(\Gamma)$. Therefore, $\widetilde{h}\in\dot{H}^\alpha_L(\Gamma)$ because  $\dot{H}^\alpha_L(\Gamma)$ is closed (it is a Hilbert space) and $\widetilde{\varphi}_n$ converges to $\widetilde{h}$ in $\dot{H}^\alpha_L(\Gamma)$. 
	\end{proof}


\begin{proof}[Proof of Proposition~\ref{thm:CharOrthComplCM_genWM}]
We will prove that $(\mathcal{H}_{+,0}(S),\|\cdot\|_{\widetilde{H}^\alpha(\Gamma)}) \cong (\dot{H}_{0,L}^\alpha(S),\|\cdot\|_{\widetilde{H}^\alpha(S)})$
and further that $(\mathcal{H}_{0,I}(S),\|\cdot\|_{\widetilde{H}^\alpha(\Gamma)})\cong (\dot{H}_{0,L}^\alpha(S), \|\cdot\|_{\widetilde{H}^\alpha(S)}).$  
We begin by proving the latter statement.

Let $h\in\mathcal{H}_{0,I}(S)$ and observe that $h|_S\in \dot{H}_L^\alpha(S)$. By following along the same lines as in
the proof of Theorem \ref{thm:boundarySpcsGenWM}, $h_e^{(k)}(s)=0$, $k=0,\ldots,\alpha-1$,
for every $s\in\partial S$ and $e\in\mathcal{E}_S\cap \mathcal{E}_s$. 
Therefore, it follows from Lemma~\ref{lem:sobTraceZero} 
that
$h|_S$ can be approximated by $\varphi_n \in C_c(S)\cap \dot{H}_L^\alpha(S)$
with respect to the $\|\cdot\|_{\widetilde{H}^\alpha(S)}$ (recall
that the norm $\|\cdot\|_{\widetilde{H}^\alpha(\Gamma)}$ is equivalent
to $\|\cdot\|_{\dot{H}_L^\alpha(\Gamma)}$ in $\dot{H}_L^\alpha(\Gamma)$).

For every $s\in\partial S$ and every $\widetilde{e}\in\mathcal{E}_{\Gamma\setminus S}\cap \mathcal{E}_s$, $L_{\widetilde{e}}^{\nicefrac{j}{2}}h_{\widetilde{e}}(s)= L_{e}^{\nicefrac{j}{2}}h_{e}(s)$,
for $j\leq \alpha -1$, and $j$ even, whereas 
$\sum_{\widetilde{e}\in\mathcal{E}_s\cap \mathcal{E}_{\Gamma\setminus S}} \partial_eL_e^{\nicefrac{j}{2}}h_{\widetilde{e}}(s) = 0,$
for $j\leq \alpha -1$, $s\in\mathcal{V}$ and $j$ odd. 
This implies, together with the fact that $h_e^{(k)}(s)=0$
for every $s\in\partial S$ and $e\in\mathcal{E}_S\cap \mathcal{E}_s$, that if we extend $h|_{\Gamma\setminus S}$ as zero on $S$,
and denote this extension by $\widetilde{h|_{\Gamma\setminus S}}$, then $\widetilde{h|_{\Gamma\setminus S}}\in \dot{H}^\alpha_L(\Gamma)$ by the explicit characterization of $\dot{H}_L^\alpha(\Gamma)$ in Proposition \ref{prp:CharGen_CM_natural_alpha}. 
Therefore, if we let $\widetilde{\varphi}_n$ be the extension of $\varphi_n$ by setting
it to be zero on $\Gamma\setminus S$, then 
for every $n\in\mathbb{N}$, $\widetilde{\varphi}_n + \widetilde{h|_{\Gamma\setminus S}}
\in \dot{H}^\alpha_L(\Gamma)$. 
Let ${h_n = \widetilde{\varphi}_n + \widetilde{h|_{\Gamma\setminus S}}}$, then $h_n \to h$ in $\dot{H}^\alpha_L(\Gamma)$.

Observe that $\widetilde{h|_{\Gamma\setminus S}}\in\dot{H}_L^\alpha(\Gamma)$ implies
that $w\in H(\Gamma)$ exists such that ${\widetilde{h|_{\Gamma\setminus S}}(s) = \mathbb{E}(u(s)w)}$. 
As $\widetilde{h|_{\Gamma\setminus S}}$ vanishes on
$S$, it follows that $w\perp H(S)$, which thus implies, through the
isometric isomorphism between Gaussian spaces and Cameron--Martin spaces,
that $\widetilde{h|_{\Gamma\setminus S}}\perp \mathcal{H}(S)$.
Since $h\in \mathcal{H}(S)$, it follows that $\widetilde{h|_{\Gamma\setminus S}}\perp h$. Furthermore, because $\textrm{supp}\, \widetilde{\varphi}_n \subset S$,
it also follows that $\widetilde{h|_{\Gamma\setminus S}}\perp \widetilde{\varphi}_n$. Then,
\begin{align*}
	0 &= (h, \widetilde{h|_{\Gamma\setminus S}})_{\dot{H}_L^\alpha(\Gamma)} = \lim_{n\to\infty} (\widetilde{\varphi}_n + \widetilde{h|_{\Gamma\setminus S}}, \widetilde{h|_{\Gamma\setminus S}})_{\dot{H}_L^\alpha(\Gamma)}\\
	&= \lim_{n\to\infty} (\widetilde{\varphi}_n, \widetilde{h|_{\Gamma\setminus S}})_{\dot{H}_L^\alpha(\Gamma)} + \left\| \widetilde{h|_{\Gamma\setminus S}}\right\|_{\dot{H}_L^\alpha(\Gamma)}^2
	= \left\| \widetilde{h|_{\Gamma\setminus S}}\right\|_{\dot{H}_L^\alpha(\Gamma)}^2.
\end{align*}
Hence, $\widetilde{h|_{\Gamma\setminus S}} \equiv 0$. Therefore,
$h$ can be approximated by $\widetilde{\varphi}_n$ in the $\dot{H}^\alpha_L(\Gamma)$
norm. Furthermore, $h|_{\Gamma\setminus S} \equiv 0$, and by Proposition~\ref{prp:CharGen_CM_natural_alpha}, the norm 
in ${\dot{H}_L^\alpha(\Gamma)}$ is equivalent to the $\|\cdot\|_{\widetilde{H}^\alpha(\Gamma)}$ norm. 
Therefore, $h$ can be identified with $h|_S$, which can be approximated by $\varphi_n\in C_c(S)\cap \dot{H}_L^\alpha(S)$
in the $\|\cdot\|_{\widetilde{H}^\alpha(S)}$ norm.

The proof $\mathcal{H}_{+,0}(S) \cong \dot{H}_{0,L}^\alpha(S)$ is analogous, 
but simpler. In addition, $\mathcal{H}_{+,0}(S) \cong \dot{H}_{0,L}^\alpha(S)$ was also indirectly proved in Theorem \ref{thm:boundarySpcsGenWM} as 
an intermediate step.
Finally, if $\widetilde{h}$ is the extension of $h\in \dot{H}^\alpha_L(S)$ by setting zero in $\Gamma\setminus S$, then 
$\widetilde{h}\in \dot{H}^\alpha_L(\Gamma)$
by Lemma~\ref{lem:sobTraceZero} 
\end{proof}

	We now prove Corollary \ref{cor:alpha1regularboundary} and the characterizations of $\mathcal{H}_{+,0}(S)$ and $\mathcal{H}_{0,I}(S)$.

	\begin{proof}[Proof of Corollary \ref{cor:alpha1regularboundary}]
		The result for $\alpha=1$ follows directly from Theorem \ref{thm:charInnerOuterEdge} and 
		the continuity of the field in $L_2(\Omega)$ (see \cite[Corollary 4.6]{bolinetal_fem_graph}).
		The remaining statements are direct consequences of Theorem~\ref{thm:charInnerOuterEdge} and Corollary \ref{cor:BoundaryConditionsGen_WM}, since it means that the inner derivatives
		at the boundary of an edge are linear combinations of elements in $H_{+,O}(\partial e)$,
		due to the conditions $L_e^{\nicefrac{k}{2}} u_e(v) = L_{e'}^{\nicefrac{k}{2}}u(v)$ if $k$ is even and $\sum_{e\in\mathcal{E}_v} \partial_{e} L_e^{\nicefrac{(k-1)}{2}} u(v) = 0$ if $k$ is odd, for $k \in \{0,\ldots, \ceil{\alpha - \nicefrac{1}{2}}  -1\}$, $v\in\mathcal{V}$, $e, e'\in\mathcal{E}_v$.
	\end{proof}

	Finally, we have the characterization of the space $\mathcal{H}_+(\Gamma\setminus S)$, where $S$ is an open set whose boundary consists of finitely many points.

	\begin{proof}[Proof of Proposition \ref{prp:ChatHPlusS_genWM}]
		By Corollary \ref{cor:MarkovOrderpCameronMartin}.\ref{cor:MarkovOrderpCameronMartin1}, $\mathcal{H}_+(\Gamma\setminus S) = \mathcal{H}_{+,0}(S)^\perp$, and by
		Theorem~\ref{thm:CharOrthComplCM_genWM}, $\mathcal{H}_{+,0}(S)$
		can be identified with the closure of $C_c(S)\cap \dot{H}_L^\alpha(S)$ with respect to ${\|\cdot\|_{\widetilde{H}^\alpha(S)}}$. Thus, given $h\in \mathcal{H}_+(\Gamma\setminus S)$, for every $\varphi\in C_c(S)\cap \dot{H}_L^{2\alpha}(S) \subset C_c(S)\cap \dot{H}_L^{\alpha}(S)$,
		$$
		0 = (h, \widetilde{\varphi})_{\alpha} = (L^{\alpha/2} h, L^{\alpha/2}\widetilde{\varphi})_{L_2(\Gamma)} = (h, L^\alpha\widetilde{\varphi})_{L_2(\Gamma)} = (h|_S, L^\alpha\varphi)_{L_2(S)},
		$$
		where $\widetilde{\varphi}$ is the extension of $\varphi$ to $\Gamma$ by 
		setting zero at $\Gamma\setminus S$, which belongs to $\dot{H}^\alpha_L(\Gamma)$ by Theorem \ref{thm:CharOrthComplCM_genWM}. 
		Therefore, $h|_S$ is a weak solution of the problem $L^\alpha h|_S = 0$ on $S$.
		By restricting to each edge in $S$,  $h|_{e\cap S}$ is a weak
		solution of $L^\alpha h|_{e\cap S} = 0$ in $S\cap e$. By standard elliptic
		regularity, it follows by Assumption~\ref{assump:basic_assump2} that $h|_{e\cap S} \in C^{\alpha}(e\cap S)$. Therefore,
		$h|_S$ is a classical solution of $L^\alpha h|_S = 0$ on $S$. 
		Further, observe that if $\kappa|_e, H|_e\in C^\infty(e)$ for every $e\in\mathcal{E}$, then ${h|_{e\cap S}\in C^\infty(e\cap S)}$. 
		
		Conversely, let $h$ be a function in $\dot{H}_L^\alpha(\Gamma)$ such that $h|_S$ is a weak solution of $L^\alpha h|_S = 0$ on $S$. Then, for any $\varphi\in C_c(S)\cap \dot{H}_L^\alpha(S),$
		$$
		0 = (L^{\alpha/2} h|_S, L^{\alpha/2}\varphi)_{L_2(S)} =  (L^{\alpha/2} h, L^{\alpha/2}\widetilde{\varphi})_{L_2(\Gamma)} = (h, \widetilde{\varphi})_{\alpha},
		$$
		where $\widetilde{\varphi}$ is the extension of $\varphi$ as zero on $\Gamma\setminus S$, which belongs to $\dot{H}^\alpha_L(\Gamma)$. 
		By Theorem \ref{thm:CharOrthComplCM_genWM}, it follows that
		$h\perp \mathcal{H}_{+,0}(S)$. Therefore, $h\in \mathcal{H}_+(\Gamma\setminus S)$.   
		This proves the first claim.
		Since $\mathcal{H}_+(\partial S) = \mathcal{H}_+(\overline{S})\cap \mathcal{H}_+(\Gamma\setminus S)$, the second claim follows from applying the first claim to $S$ and to $\Gamma\setminus\overline{S}$. 
	\end{proof}

	\section{Auxiliary results for Sections \ref{sec:edge} and \ref{sec:condrepr}}\label{app:sec6sec7}
	We begin with a result that characterizes the space $(\mathcal{H}_{0,I}(e), (\cdot, \cdot)_{\dot{H}^\alpha_L(\Gamma)})$.

	\begin{Proposition}\label{prp:charHdot_edge_genWM}
		Let $\alpha\in\mathbb{N}$, let Assumption \ref{assump:basic_assump2} hold and recall the definition of $\dot{H}^\alpha_{0,L}(\Gamma)$ in the statement of Theorem \ref{thm:CharOrthComplCM_genWM}. Then, there exists an isometric isomorphism between $(\mathcal{H}_{0,I}(e), (\cdot, \cdot)_{\dot{H}^\alpha_L(\Gamma)})$
		and $(\dot{H}_{0,L}^\alpha(e), (\cdot,\cdot)_{\alpha,e})$, where $(\cdot,\cdot)_{\alpha,e}$ is the extension to ${\dot{H}_{0,L}^\alpha(e)\times \dot{H}_{0,L}^\alpha(e)}$ 
		of the bilinear form
		$$(u,v)_{\alpha,e} = (u, L^\alpha v)_{L_2(e)}, \quad u\in \dot{H}_{0,L}^\alpha(e), v\in C_c(e)\cap \dot{H}_L^{2\alpha}(e).$$
	\end{Proposition}
	
	\begin{proof}
		By Theorem \ref{thm:CharOrthComplCM_genWM}, $\mathcal{H}_{0,I}(e) \cong \dot{H}_{0,L}^\alpha(e)$. 
		Next, $\dot{H}^{2\alpha}_L(\Gamma)$ is dense in $\dot{H}^\alpha_L(\Gamma)$, and $\alpha\in\mathbb{N}$, so both $\dot{H}^{2\alpha}_L(\Gamma)$ and $\dot{H}^\alpha_L(\Gamma)$ are local. 
		In addition, by Proposition \ref{prp:CharGen_CM_natural_alpha} the norms $\|\cdot\|_{\dot{H}^\alpha_L(\Gamma)}$ and $\|\cdot\|_{\widetilde{H}^\alpha(\Gamma)}$ are equivalent. 
		Therefore, $\dot{H}^{2\alpha}_L(e)$ is dense in ${\dot{H}_L^\alpha(e)\cap C_c(e)}$
		with respect to the Sobolev norm $\|\cdot\|_{H^\alpha(e)}$. 
		Further, because $\dot{H}_{0,L}^\alpha(e)$ is the completion of $C_c(e)\cap \dot{H}_L^\alpha(e)$ 
		with respect to the Sobolev norm ${\|\cdot\|_{H^\alpha(e)}}$, it follows that $\dot{H}_{0,L}^\alpha(e)$ 
		is the completion of $\dot{H}^{2\alpha}_L(e)$ 
		with respect to $\|\cdot\|_{H^\alpha(e)}$. 
		It remains to demonstrate
		that the bilinear form $(\cdot,\cdot)_{\alpha,e}$ is isometric to $(\cdot,\cdot)_{\dot{H}^\alpha_L(\Gamma)}$, where we identify $f\in \mathcal{H}_{0,I}(e)$ with $f|_e$. Indeed, 
		given $u\in \dot{H}_{0,L}^\alpha(e)$ and ${v\in C_c(e)\cap\dot{H}_L^{2\alpha}(e)}$, we can extend $u$ and $v$ to $\Gamma$ by defining
		them to be zero on $\Gamma\setminus e$. Let $\widetilde{u}$ and $\widetilde{v}$ be these extensions. 
		By Theorem \ref{thm:CharOrthComplCM_genWM}, $\widetilde{u}, \widetilde{v} \in \mathcal{H}_{0,I}(e)$, and
		as $\alpha\in\mathbb{N}$, 
		$
			(u,v)_{\alpha,e} = (u, L^\alpha v)_{L_2(e)} = (\widetilde{u}, L^\alpha \widetilde{v})_{L_2(\Gamma)} = (\widetilde{u}, \widetilde{v})_{\dot{H}^\alpha_L(\Gamma)}.
		$
		The space  $\dot{H}_L^{2\alpha}(e)$ is dense in $\dot{H}_{0,L}^\alpha(e)$ with respect to the Sobolev norm $\|\cdot\|_{H^\alpha(e)}$
		and by Proposition~\ref{prp:CharGen_CM_natural_alpha}, the norm $\|\cdot\|_{\dot{H}_L^\alpha(\Gamma)}$ and the 
		Sobolev norm $\|\cdot\|_{\widetilde{H}^\alpha(\Gamma)}$ are equivalent.  
		It follows that the bilinear form is continuous with respect to the Sobolev norm $\|\cdot\|_{H^\alpha(e)}$, and  $(\cdot,\cdot)_{\alpha,e}$ can therefore be uniquely extend 
		to $\dot{H}_{0,L}^\alpha(e)\times \dot{H}_{0,L}^\alpha(e)$. 
		Finally, for $u,v\in \dot{H}_{0,L}^\alpha(e)$, let $\widetilde{u}$ and $\widetilde{v}$ be their
		extensions to $\Gamma\setminus e$ (as zero), thus, $(u,v)_{\alpha,e} = (\widetilde{u}, \widetilde{v})_{\dot{H}_L^\alpha(\Gamma)}.$ 
		Since $\dot{H}_{0,L}^\alpha(e)\cong \mathcal{H}_{0,I}(e)$, this concludes the proof.
	\end{proof}

	The next lemma is a result regarding finite dimensional Cameron--Martin spaces.
	\begin{Lemma}
		\label{lem:FRKHS}
		Let $e=[0,\ell_e]$, $\ell_e>0$, and suppose that $H(e)$ is a Gaussian space with corresponding Cameron--Martin space $\mathcal{H}(e)$ and isometric isomorphism ${\Phi:H(e)\to\mathcal{H}(e)}$. Further, let ${d\in\mathbb{N}}$ and suppose that $f_1,\ldots f_d\in \mathcal{H}(e)$ are linearly independent, ${f_1,\ldots,f_d:e\to\mathbb{R}}$, and define the function ${\mv{f}(\cdot)=(f_1(\cdot),\ldots,f_d(\cdot))}$. Suppose that $\mathcal{H}(e)$ has the reproducing kernel 
		\begin{equation}\label{eq:exprReprKernel}
			h(t,t') = \mv{f}(t')^\top \mv{C} \mv{f}(t), \quad t,t'\in [0,\ell_e],
		\end{equation}
		where $\mv{C}$ is a positive definite matrix. Then, $\mv{Z} = (Z_1,\ldots,Z_d)$, where $Z_j\in H(e)$, exists such that
		$$
		w(t) = \Phi^{-1}(h(\cdot,t)) = \mv{f}^T(t) \mv{Z}, \quad t\in [0,\ell_e].
		$$
		Finally, $\mv{C} = \mathbb{E}(\mv{Z}\mv{Z}^\top)$, and thus; for $t,t'\in [0,\ell_e]$,
		$h(t,t') = \mv{f}(t')^\top \mathbb{E}(\mv{Z}\mv{Z}^\top) \mv{f}(t).$
	\end{Lemma}
	
	\begin{proof}
		First, observe that \eqref{eq:exprReprKernel} shows that $\mathcal{H}(e)$ is finite-dimensional. Then, 
		by \cite[Corollary~8.16]{janson_gaussian}, $\mathcal{H}(e) = \overline{\textrm{span}\{h(\cdot, t): t\in e\}}$, where the closure is given with respect to the $\mathcal{H}(e)$-norm; therefore,  $\mathcal{H}(e)\subset \textrm{span}\{f_1(t),\ldots,f_d(t):t\in e\}$. 
		As $\mv{C}$ is positive definite, $\mathcal{H}(e) = \textrm{span}\{f_1(t),\ldots,f_d(t):t\in e\}$.  
		Next, let $w(t) = \Phi^{-1}(h(\cdot,t))$, then  by \cite[Example 8.17 and Theorem 8.22]{janson_gaussian},  
		$P(w(\cdot) \in \mathcal{H}(e)) = 1$ because $\mathcal{H}(e)$ is finite-dimensional. Therefore, since $f_1,\ldots,f_d$ is a basis for $\mathcal{H}(e)$, there exists a Gaussian vector $\mv{Z} = (Z_1,\ldots,Z_d)$ with $Z_1,\ldots,Z_d\in H(e)$ such that
		$w(\cdot) = \sum_{j=1}^d Z_j f_j(\cdot) = \mv{f}^\top(\cdot)\mv{Z}.$
		Because $f_1,\ldots,f_d$ are linearly independent, the vector $\mv{Z}$ is, almost surely, uniquely defined. 
		Finally, the above expression allows rewriting \eqref{eq:exprReprKernel} in terms of $\mv{Z}$. 
		Indeed, by \cite[Example~8.17]{janson_gaussian}, 
		$$
		\Phi(w(t))(t') = \mathbb{E}(w(t)w(t')) = \mv{f}(t')^\top \mathbb{E}(\mv{ZZ}^\top) \mv{f}(t), \quad \mbox{for $t,t'\in [0,\ell_e]$},
		$$
		because ${w(t) = \Phi^{-1}(h(\cdot,t))}$.
	\end{proof}
	
	
	\begin{Lemma}\label{lem:Cameron_Martin_decomposition}
		Let  $u:\Gamma\to\mathbb{R}$ be a Gaussian process with continuous covariance function ${\varrho:\Gamma\times\Gamma\to\mathbb{R}}$. 
		Further, let 
		$H(\Gamma) = \overline{\textrm{span}\{u(s): s\in\Gamma\}}$
		and
		$\mathcal{H}(\Gamma) = \{g(s) = \mathbb{E}(wu(s)): w\in H(\Gamma)\}$
		be the Gaussian space and the Cameron--Martin space associated to $u$, respectively. 
		If we have the decomposition
		\begin{equation}\label{eq:orthodecomp_CM}
			\mathcal{H}(\Gamma) = \mathcal{H}_1 \oplus \mathcal{H}_2,
		\end{equation}
		then $\mathcal{H}_1$ and $\mathcal{H}_2$ are reproducing kernel Hilbert spaces. Further, 
		$\rho(s,s') = h_{1}(s,s') + h_{2}(s,s'),$
		where $h_1$ and $h_2$ are the reproducing kernels of $\mathcal{H}_1$ and $\mathcal{H}_2$, respectively.
	\end{Lemma}
	\begin{proof}
		Equation \eqref{eq:orthodecomp_CM} shows that for each $s\in\Gamma$, unique functions $h_{1}(\cdot, s)\in\mathcal{H}_1$ and ${h_{2}(\cdot, s)\in\mathcal{H}_2}$ exists such that $\rho(\cdot,s) = h_{1}(\cdot, s) + h_{2}(\cdot,s)$. 
		By the symmetry of $\rho$ and the inner product of $\mathcal{H}(\Gamma)$, and orthogonality of $ \mathcal{H}_1$ and $\mathcal{H}_2$, 
		\begin{align*}
		h_{1}(s',s) &= \left( \rho(\cdot,s'),h_{1}(\cdot,s) \right)_{\mathcal{H}(\Gamma)}= \left( h_{1}(\cdot,s'),h_{1}(\cdot,s) \right)_{\mathcal{H}(\Gamma)} \\
		&=  \left(\rho(\cdot, s),h_{1}(\cdot,s')\right)_{\mathcal{H}(\Gamma)}=
		h_{1}(s,s'),
		\end{align*}
		for $s,s'\in \Gamma$. Hence, $h_1(s,s')= h_1(s',s)$. Similarly, for $s,s'\in \Gamma$, $h_{2}(s,s') = h_{2}(s',s)$. Thus, $h_1(\cdot,\cdot)$ and $h_2(\cdot,\cdot)$ are symmetric kernels on $e\times e$. Further, 
		for $h\in\mathcal{H}_1$, $h\perp \mathcal{H}_2$ and since  $h\in \mathcal{H}(\Gamma)$, we obtain
		\begin{equation}\label{eq:h2ReprKernel}
			h(s)= (\rho(\cdot, s), h)_{\mathcal{H}(\Gamma)} = (h_1(\cdot, s), h)_{\mathcal{H}(\Gamma)}, \quad s\in \Gamma,
		\end{equation} 
		which shows that $h_1(s,t)$ is the reproducing kernel of  $\mathcal{H}_1$.  
		By the same argument, $h_2$ is the reproducing kernel of $ \mathcal{H}_2$. 
	\end{proof}
	
	Finally, we have the following lemma, used in Section \ref{sec:condrepr},
	regarding Cameron--Martin spaces of centered Gaussian fields on compact metric graphs whose restrictions to edges consist of independent processes.
	
	\begin{Lemma}\label{lem:auxLemCMindep}
		Let $\Gamma$ be a compact metric graph and let $\mathcal{E}$ be the set of its edges. Let, also, $\{v_e:e\in\mathcal{E}\}$, with $v_e:\Omega\times e\to\mathbb{R}$, be a family of independent centered Gaussian processes taking values on $L_2(e)$. Further, for each $e\in\mathcal{E}$, let $(X_e, \<\cdot,\cdot\>_e)$ be the Cameron--Martin space associated to $v_e$ and $k_e(\cdot,\cdot)$ be its covariance function. Define $v:\Omega\times \Gamma\to\mathbb{R}$ be defined as $v(\omega, s) = v_e(\omega,t)$, where $s = (t,e)$. Then, $v$ has covariance function
		\begin{equation}\label{eq:cov_func_indepv}
			k(s,s') = \begin{cases}
				0,&\hbox{if $e\neq e'$},\\
				k_e(t,t'),&\hbox{if $e=e'$ and $s=(t,e),s'=(t',e')$},
			\end{cases}
		\end{equation}
		and the corresponding Cameron--Martin space is
			$X := \bigoplus_{e\in\mathcal{E}} X_e$,
		with inner product
		\begin{equation*}
			\<f,g\>_{\Gamma} := \sum_{e\in\mathcal{E}} \<f|_e, g|_e\>_e,\quad f,g\in X.
		\end{equation*}
		
		Conversely, if a centered Gaussian field on $\Gamma$ has the Cameron--Martin space $(X,\<\cdot,\cdot\>_\Gamma)$. Then, $\{v|e:e\in\mathcal{E}\}$ is a collection of independent stochastic processes, and for each $e\in\mathcal{E}$, $v|_e$ is associated to the Cameron--Martin space $(X_e, \<\cdot,\cdot\>_e)$.
	\end{Lemma}
	
	\begin{proof}
		The expression for the covariance function given in \eqref{eq:cov_func_indepv} follows directly from independence of $\{v_e:e\in\mathcal{E}\}$.
		Let us now obtain the Cameron--Martin space associated to $v$. Observe that the Cameron--Martin space associated to $v$ is the reproducing kernel Hilbert space with kernel given by \eqref{eq:cov_func_indepv}. It is clear that $(X,\<\cdot,\cdot\>_\Gamma)$ is a Hilbert space that contains $k(\cdot, s), s\in\Gamma$. Therefore, to show that $(X,\<\cdot,\cdot\>_\Gamma)$ is the Cameron--Martin space associated to $v$ it is enough to show that $k(\cdot,\cdot)$ given by \eqref{eq:cov_func_indepv} is the reproducing kernel for $(X,\<\cdot,\cdot\>_\Gamma)$.
		To this end, let $f\in X$ and $s = (t,e)\in \Gamma$. Then,
		$$\<f, k(\cdot, s)\>_\Gamma = \<f|_e, k_e(\cdot, t)\>_e = f|_e(t) = f(s).$$
		Therefore, $k(\cdot,\cdot)$ is a reproducing kernel for $(X,\<\cdot,\cdot\>_\Gamma)$, which shows that $(X,\<\cdot,\cdot\>_\Gamma)$ is the Cameron--Martin space associated to $v$.
		The proof of the converse follows from the uniqueness of the Cameron--Martin space associated with a centered Gaussian measure \cite[e.g.~see][Theorem 2.9]{daPrato2014}.
	\end{proof}
	\end{appendix}

\bibliographystyle{imsart-number}
\bibliography{../../../Bib/unified_graph_bib}

\end{document}